\documentclass[a4paper, 12pt, notitlepage]{amsart}
\usepackage[colorlinks]{hyperref} 
\usepackage{latexsym, amssymb, amscd, amsfonts, amsthm, amsmath, graphicx,
mathabx, mathrsfs}
\usepackage{enumitem}

\DeclareFontFamily{OMX}{MnSymbolE}{}
\DeclareSymbolFont{largesymbols}  {OMX}{MnSymbolE}{m}{n}

\DeclareFontShape{OMX}{MnSymbolE}{m}{n}{ 
    <-6>  MnSymbolE5
   <6-7>  MnSymbolE6
   <7-8>  MnSymbolE7
   <8-9>  MnSymbolE8
   <9-10> MnSymbolE9
  <10-12> MnSymbolE10
  <12->   MnSymbolE12}{}

\DeclareMathSymbol{\llangle}{\mathopen}{largesymbols}{116}
\DeclareMathSymbol{\rrangle}{\mathclose}{largesymbols}{121}

\usepackage[numbers]{natbib}

\newcommand{\arxiv}[1]{{\tt \href{http://arxiv.org/abs/#1}{arXiv:#1}}} 

\newcommand{\Poisson}{\operatorname{Poisson}}

\newcommand{\supp}{\operatorname{supp}}

\newcommand{\mix}{\operatorname{mix}}
\newcommand{\TV}{\operatorname{TV}}

\newcommand{\spaN}{\operatorname{span}}
\newcommand{\nbd}{\operatorname{nbd}}
\newcommand{\sav}{\operatorname{sav}}
\newcommand{\RE}{\operatorname{Re}}
\newcommand{\IM}{\operatorname{Im}}

\newcommand{\e}{\mathbf{e}}

\newcommand{\bC}{\mathbb{C}}

\newcommand{\bR}{\mathbb{R}}

\newcommand{\bT}{\mathbb{T}}
\newcommand{\bU}{\mathbb{U}}
\newcommand{\zed}{\mathbb{Z}}

\newcommand{\sgn}{\mathrm{sgn}}

\newcommand{\gap}{\mathrm{gap}}

\newcommand{\one}{\mathbf{1}}
\newcommand{\E}{\mathbf{E}}

\newcommand{\Prob}{\mathbf{Prob}}

\newcommand{\sB}{\mathscr{B}}

\newcommand{\sG}{\mathscr{G}}
\newcommand{\sN}{\mathscr{N}}
\newcommand{\sR}{\mathscr{R}}
\newcommand{\sS}{\mathscr{S}}

\newcommand{\sX}{\mathscr{X}}

\newcommand{\cI}{{\mathcal{I}}}

\newcommand{\sP}{{\mathscr{P}}}
\newcommand{\sH}{{\mathscr{H}}}

\newcommand{\sC}{{\mathscr{C}}}

\newtheorem{theorem}{Theorem}

\newtheorem{lemma}[theorem]{Lemma}
\newtheorem{proposition}[theorem]{Proposition}

\theoremstyle{remark}
\newtheorem*{rem}{Remark}

\title[Sandpiles]{Sandpiles on the square lattice}
\author{Bob Hough}
\address[Bob Hough]{Department of Mathematics, Stony Brook University, Stony Brook,
NY, 11794}
\email{robert.hough@stonybrook.edu}

\author{Daniel C. Jerison}
\address[Daniel C. Jerison]{Department of Mathematics, Malott Hall, Cornell University, Ithaca, NY 14853 \newline Department of Mathematics, Tel Aviv University, Tel Aviv, Israel}
\email{dcjerison@gmail.com}

\author{Lionel Levine}
\address[Lionel Levine]{Department of Mathematics, Malott Hall, Cornell University, Ithaca, NY 14853}
\email{levine@math.cornell.edu}


\thanks{This material is based upon work supported by the National Science
Foundation under agreements No.\ DMS-1128155, \href{http://www.nsf.gov/awardsearch/showAward?AWD_ID=1455272}{DMS-1455272}, DMS-1712682, and DMS-1802336. Any opinions, findings and
conclusions or recommendations expressed in this material are those of the
authors and do not necessarily reflect the views of the National Science
Foundation.}

\begin{document}

\begin{abstract}
We give a non-trivial upper bound for the critical density when stabilizing 
i.i.d.~distributed sandpiles on the lattice $\zed^2$.
We also determine the asymptotic spectral gap, asymptotic mixing time and prove 
a cutoff phenomenon
for the recurrent state abelian sandpile model on the
torus $\left(\zed/m\zed\right)^2$.
The techniques use analysis of the space of functions on $\zed^2$ which are 
harmonic modulo 1.
In the course of our arguments, we characterize the harmonic modulo 1 functions in $\ell^p(\zed^2)$ as linear combinations of certain discrete derivatives of Green's functions, extending a result of Schmidt and Verbitskiy \cite{SV09}.
\end{abstract}

\maketitle

\section{Introduction}

\subsection{Stabilization of i.i.d.~sandpiles}
A \emph{sandpile} on the integer lattice $\zed^2$ is a function $\sigma : \zed^2 \to \zed_{\geq 0}$, where $\sigma(x)$ represents the number of grains of sand at the site $x$. The sandpile $\sigma$ is \emph{stable} if each $\sigma(x) \leq 3$. If some $\sigma(x) \geq 4$, then we may \emph{topple} the sandpile at $x$ by passing one grain of sand from $x$ to each of its four nearest neighbors. We say that $\sigma$ \emph{stabilizes} if it is possible to reach a stable configuration from $\sigma$ by toppling each vertex finitely many times. If the heights $(\sigma(x))_{x \in \zed^2}$ are i.i.d.~random variables, we refer to $\sigma$ as an i.i.d.~sandpile.

Meester and Quant \cite{MQ05} asked which i.i.d.~sandpiles stabilize almost surely. It was proved by Fey and Redig \cite{FR05} that such a sandpile $\sigma$ must satisfy $\E[\sigma(x)] \leq 3$. This condition is not sufficient for stabilization: for every $p > 0$, the i.i.d.~sandpile where each $\sigma(x) = 2$ with probability $1-p$ and $\sigma(x) = 4$ with probability $p$ almost surely fails to stabilize \cite{FLP10}.
Thus, for each $2 < \rho \leq 3$, there are some i.i.d.~sandpiles with $\E[\sigma(x)] = \rho$ that do stabilize almost surely (e.g.~when each $\sigma(x) \in \{0,1,2,3\}$) and others that fail to stabilize. This behavior contrasts with the closely related \emph{divisible sandpile model}, in which stabilization of a nonconstant i.i.d.~initial condition $\sigma$ is determined entirely by the value of $\E[\sigma(x)]$ \cite{LMPU16}.

Our first main theorem shows that an i.i.d.~sandpile with $\E[\sigma(x)]$ slightly less than $3$ cannot stabilize almost surely unless $\sigma(x) \leq 3$ with high probability. 

\begin{theorem} \label{quantitative_theorem}
There are constants $c,d > 0$ such that any i.i.d.~sandpile $\sigma$ on $\zed^2$ that stabilizes almost surely satisfies
\begin{equation}
\E[\sigma(x)] \leq 3 - \min\left( c, d \E[|X-X'|^{2/3}] \right)
\end{equation}
where $X,X'$ are independent and distributed as $\sigma(x)$.
\end{theorem}

If $3 - \E[\sigma(x)]$ is small, then the inequality $\Prob(X \neq X') \leq \E[|X - X'|^{2/3}]$ implies that the law of $\sigma(x)$ is  concentrated at a single value, which must be at most $3$. Some extra work would be required to extract explicit values for the constants $c$ and $d$ from our proof of Theorem \ref{quantitative_theorem}; see the discussion following Lemma \ref{xi_tail}.

Theorem \ref{quantitative_theorem} answers a question posed by Fey, Meester, and Redig \cite{FMR09} by demonstrating that an i.i.d.~Poisson sandpile with mean sufficiently close to $3$ almost surely does not stabilize. 
An interesting question that remains open is whether there exists $\epsilon>0$ such that the only i.i.d.~stabilizing sandpiles with $\E[\sigma(x)]>3-\epsilon$ are those which are already stable.

\subsection{Cutoff for sandpiles on the torus}

We also consider sandpile dynamics on the discrete torus
$\bT_m = \left(\zed/m\zed\right)^2$, given as follows.  The
point $(0,0)$ is designated \emph{sink} and is special.  Each non-sink point on the
torus has a sand allocation \begin{equation}\sigma: \bT_m 
\setminus
\{(0,0)\} \to \zed_{\geq 0}.\end{equation}  As on the integer lattice, if at 
some time a non-sink vertex has allocation at
least 4 it may topple, passing one grain of sand to each of its neighbors; if a
grain of sand falls on the sink it is lost from the model. Those states $\sS_m$ 
for which $\sigma \leq 3$ are stable.  We consider the discrete time dynamics, where a single step consists of 
dropping a grain of sand on a uniformly randomly chosen vertex and then 
performing all legal topplings until the model reaches a stable state. The \emph{abelian property} \cite{D90} ensures that this stable state does not depend on the order in which the topplings were performed.

Those stable states $\sR_m$ which may be reached from the maximal state $\sigma \equiv 3$ are
 recurrent, whereas all other states are
transient. Started from any stable state, the sandpile model forms a Markov
chain with transition kernel $P_m$, which converges to the uniform measure
$\bU_{\sR_m}$ on recurrent states.

\begin{theorem}\label{mixing_time_theorem}
Let $m \geq 2$. There is a constant $c_0 = 0.348661174(3)$ and $t^{\mix}_m= 
c_0
m^2 \log m$ such that the following
holds. For each fixed $\epsilon>0$, 
\begin{align}
&\lim_{m \to \infty} \min_{\sigma \in 
\sS_m}\left\|P_m^{\lceil(1-\epsilon)t^{\mix}_m\rceil}\delta_{\sigma}- 
\bU_{\sR_m} \right\|_{\TV(\sS_m)} = 1, 
\\ \notag & \lim_{m \to \infty} \max_{\sigma \in \sS_m} 
\left\|P_m^{\lfloor(1+\epsilon)t_m^{\mix}\rfloor}\delta_{\sigma}- \bU_{\sR_m} 
\right\|_{\TV(\sS_m)} = 0.
\end{align}
\end{theorem}
Informally, the convergence to uniformity of the sandpile model on the torus 
has total variation mixing time asymptotic to $c_0 m^2 \log m$ and the transition to uniformity 
satisfies a cutoff phenomenon. Implicit in the statement of 
Theorem \ref{mixing_time_theorem} is that, with high probability, the time to 
reach a recurrent state started from a general state in the model is less than 
the mixing time.  In Section \ref{sandpile_section} we give an easy proof 
using a coupon collector-type argument that this hitting time is almost surely $O(m^2 
\sqrt{\log m})$.  Also, the asymptotic mixing time of order $m^2 \log 
m$ is at a later point than is sampled in some statistical physics studies 
regarding sandpiles, see \cite{SMKW15}.

We also determine asymptotically the absolute spectral gap of the torus sandpile Markov 
chain.
\begin{theorem}\label{spectral_gap_theorem}
Let $m \geq 1$. There is a constant 
 \[
  \gamma = 2.868114013(4)
 \]
such that the absolute spectral gap of the sandpile Markov chain restricted to its recurrent states satisfies 
\begin{equation}
 \gap_m = \frac{\gamma + o(1)}{m^2} \qquad \text{as $m \to \infty$.}
\end{equation}
\end{theorem}
The constants in the preceding theorems are reciprocals: $c_0 \gamma = 1$. An explicit formula for $\gamma$ in terms of the Green's function on $\zed^2$ is given in Appendix \ref{spectral_gap_appendix}.


\subsection{Functions harmonic modulo 1}
\label{Functions_harmonic_modulo_1}

Functions which are \emph{harmonic modulo $1$} play a central role in the proofs of Theorems \ref{quantitative_theorem}--\ref{spectral_gap_theorem}. For $X = \zed^2$ or $X = \bT_m$, we say that $f: X \to \bC$ is harmonic modulo 1 if
\begin{equation}
\label{Laplacian}
(\Delta f)(i,j) := 4f(i,j) - f(i-1,j) - f(i+1,j) - f(i,j-1) - f(i,j+1)
\end{equation}
is in $\zed$ for all $(i,j) \in X$. The operator $\Delta$ is the \emph{graph Laplacian} on $X$.

Schmidt and Verbitskiy \cite{SV09} characterized the set of all functions in $\ell^1(\zed^2)$ that are harmonic modulo 1. Their result can be stated using discrete derivatives of the \emph{Green's function} on $\zed^2$. Let
\begin{equation}
\label{nu}
 \nu := \frac{1}{4} \left(\delta_{(-1,0)} + \delta_{(1,0)} + \delta_{(0,-1)} + \delta_{(0,1)} \right)
\end{equation}
be the measure that drives simple random walk on $\zed^2$, and let $\nu^{*n}$ be its $n$-th convolution power, so that $\nu^{*n}(x)$ is the probability that a random walker started from the origin is at site $x$ after $n$ steps. The Green's function is defined by
\begin{equation}
\label{G_Z2_first}
G_{\zed^2}(x) := \frac{1}{4} \sum_{n=0}^\infty \left[ \nu^{*n}(x) - \nu^{*n}(0,0) \right].
\end{equation}
Evidently, $G_{\zed^2}(0,0) = 0$. For nonzero $x = (x_1,x_2)$ with $\|x\|_2 = \sqrt{x_1^2 + x_2^2}$, it is known classically that $G_{\zed^2}(x) = -\frac{1}{2\pi} \log \|x\|_2 + O(1)$. As shown in \cite{FU96}, this is the start of an asymptotic expansion, whose first few terms we quote in Theorem \ref{greens_function_asymptotic}.

It can easily be shown that $\Delta G_{\zed^2}(x) = \e_{(0,0)}(x) := \one\{x = (0,0)\}$, so $G_{\zed^2}$ is harmonic modulo 1. By taking discrete derivatives, we can find harmonic modulo 1 functions that decay to zero with $\|x\|_2$. The discrete derivatives $D_1 f$, $D_2 f$ of any $f: \zed^2 \to \bC$ are defined as
\begin{equation}
\label{D_12}
D_1 f(i,j) := f(i+1,j) - f(i,j), \quad D_2 f(i,j) := f(i,j+1) - f(i,j).
\end{equation}
If $f$ is harmonic modulo 1, then so is any finite linear combination with integer coefficients of translates of $f$, including $D_1 f$ and $D_2 f$.

From the asymptotic expansion, it follows that the $k$-th derivatives of $G_{\zed^2}$ decay like the inverse $k$-th power of the radius. That is, if $a+b = k$, then $D_1^a D_2^b G_{\zed^2}(x) = O\left( \|x\|_2^{-k} \right)$. When $k \geq 3$, this implies that $D_1^a D_2^b G_{\zed^2} \in \ell^1(\zed^2)$. Thus, the third derivatives of $G_{\zed^2}$, and all finite integer linear combinations of their translates, are harmonic modulo 1 functions in $\ell^1(\zed^2)$. (Note that the fourth and higher derivatives are linear combinations of translates of the third derivatives.)

For $1 \leq p < \infty$, let $\sH^p(\zed^2)$ be the set of all functions in $\ell^p(\zed^2)$ that are harmonic modulo 1. Also, let $\llangle f_1,\ldots,f_n \rrangle$ denote the set of all finite integer linear combinations of translates of the functions $f_1,\ldots,f_n$ on the domain $\zed^2$, so that for example $D_1^a D_2^b f \in \llangle f \rrangle$ for any $a,b \geq 0$.

\begin{theorem}
\label{H_p_theorem}
The sets $\sH^p(\zed^2)$, for $1 \leq p < \infty$, admit the following characterization:
\begin{align}
\label{H_p_characterization}
\sH^1(\zed^2) &= \llangle D_1^3 G_{\zed^2}, D_1^2 D_2 G_{\zed^2}, D_1 D_2^2 G_{\zed^2}, D_2^3 G_{\zed^2}, \e_{(0,0)} \rrangle \\
\notag \sH^p(\zed^2) &= \llangle D_1^2 G_{\zed^2}, D_1 D_2 G_{\zed^2}, D_2^2 G_{\zed^2} \rrangle, \qquad 1 < p \leq 2 \\
\notag \sH^p(\zed^2) &= \llangle D_1 G_{\zed^2}, D_2 G_{\zed^2} \rrangle, \hspace{6.52em} 2 < p < \infty.
\end{align}
\end{theorem}

The first equality in \eqref{H_p_characterization}, which is the most delicate part to prove, is essentially a restatement of Theorem 2.4 in \cite{SV09}. We provide a unified proof of all three parts of Theorem \ref{H_p_theorem} in Section \ref{classification_section}.

Since the function $\e_{(0,0)}$ is itself in $\sH^p(\zed^2)$ for all $p$, it is implicit in the theorem statement that $\e_{(0,0)}$ is a linear combination of translates of second derivatives of $G_{\zed^2}$. This is true because $\Delta G_{\zed^2} = \e_{(0,0)}$, and the Laplacian $\Delta$ is a second-order discrete differential operator.

\subsection{Discussion of method} 
\label{Discussion of method}

This section outlines the methods used to prove Theorems \ref{quantitative_theorem}--\ref{spectral_gap_theorem}.

Theorem \ref{quantitative_theorem} says that if $\sigma$ is an i.i.d.~sandpile on $\zed^2$ that stabilizes almost surely, then $3 - \E[\sigma(x)]$ is bounded below by a quantity that measures the typical difference between the heights at two locations $\sigma(x), \sigma(x')$. To prove the theorem, let $u(x)$ be the `odometer' function that counts the number of times a vertex $x$ topples in passing from $\sigma$ to its stabilization $\sigma^\infty$, so that $\sigma^\infty = \sigma - \Delta u$.

In Section \ref{stability_section} we observe that the modulo 1 harmonic functions are dual to toppling in the following sense: If $\xi \in \ell^1(\zed^2)$ is harmonic modulo 1, then
\begin{equation}
\label{pairing-eqn}
\langle \sigma, \xi \rangle \equiv \langle \sigma^\infty, \xi \rangle \mod{1}, \qquad a.s.
\end{equation}
where $\langle f,g \rangle = \sum_{x \in \zed^2} \overline{f(x)} g(x)$ is the usual pairing. This provides a collection of invariants which obstruct stabilization in the sandpile model. 

To prove Theorem \ref{quantitative_theorem}, we consider the characteristic functions
\begin{equation}
\chi(\sigma; \xi) = \E\left[e^{-2\pi i \langle \sigma, \xi \rangle}\right], \qquad \chi(\sigma^\infty; \xi) = \E\left[e^{-2\pi i \langle \sigma^\infty, \xi \rangle}\right]
\end{equation}
which, by \eqref{pairing-eqn}, are equal. If $\E[\sigma(x)] = \E[\sigma^\infty(x)]$ is close to $3$ (which is the maximum possible value), then $\sigma^\infty(x)$ must equal $3$ for most $x \in \zed^2$. Choosing $\xi$ so that $\sum_{x \in \zed^2} \xi(x) = 0$, $\chi(\sigma^\infty; \xi)$ must be near $\chi(3; \xi) = 1$. On the other hand, since the starting values $\sigma(x)$ are i.i.d.,
\begin{equation}
\chi(\sigma; \xi) = \prod_{x \in \zed^2} \E\left[e^{-2\pi i \sigma(x) \xi(x)} \right].
\end{equation}
The modulus of each term $\E\left[e^{-2\pi i \sigma(x) \xi(x)} \right]$ decreases as the possible values of $\sigma(x)$ get more spread-out. In this way, the lower bound on $|\chi(\sigma; \xi)| = |\chi(\sigma^\infty; \xi)|$ translates into an upper bound on the amount that the starting values $\sigma(x)$ can vary.

We now turn to Theorem \ref{spectral_gap_theorem}. The set $\sR_m$ of recurrent sandpiles on the torus has a natural abelian group structure. This identifies $\sR_m$ with the \emph{sandpile group} $\sG_m$, which is formally defined in Section \ref{sandpile_section}. The sandpile Markov chain restricted to its recurrent states is a random walk on $\sG_m$, meaning that its eigenvectors are given by the dual group $\hat{\sG}_m$. We can express $\hat{\sG_m}$ as the additive group of functions $\xi : \bT_m \to \bR / \zed$ such that $\xi(0,0) = 0$ and $\Delta \xi \equiv 0$ in $\bR / \zed$. (The operation of $\xi$ on sandpiles is $\sigma \mapsto \sum_{x \in \bT_m \setminus \{(0,0)\}} \xi(x)\sigma(x)$.) In this way, an element $\xi \in \hat{\sG}_m$ is naturally associated with the set of harmonic modulo 1 functions $\xi' : \bT_m \to \bR$ that reduce mod $\zed$ to $\xi$.

The eigenvalue of the Markov chain associated to $\xi$ is the Fourier coefficient of the measure $\mu$ driving the random walk at frequency $\xi$:
\begin{equation}
\hat{\mu}(\xi) = \frac{1}{m^2} \sum_{x \in \bT_m} e^{2\pi i \xi(x)}.
\end{equation}
The mixing time is controlled by the frequencies for which $|\hat{\mu}(\xi)|$ is close to $1$.

Given a frequency $\xi$, let $\xi' : \bT_m \to \bR$ be one of its harmonic modulo $1$ representatives. The integer-valued function $v = \Delta \xi'$ will be referred to as a `prevector' of $\xi$. To recover $\xi'$ from $v$ up to an additive constant, we convolve $v$ with the Green's function $G_{\bT_m}$ on the torus, which is defined by
\begin{equation}
\label{G_T_m_first}
G_{\bT_m}(x) := \frac{1}{4}\sum_{n=0}^\infty \left(\nu^{*n}(x) - \frac{1}{m^2}\right)
\end{equation}
and is the unique mean-zero function (i.e.~$\sum_{x \in \bT_m} G_{\bT_m}(x) = 0$) satisfying
\begin{equation}
\Delta G_{\bT_m}(x) = \e_{(0,0)}(x) - \frac{1}{m^2}.
\end{equation}
It follows that $(G_{\bT_m} * v)(x) = \xi'(x) - c$, where $c = \frac{1}{m^2} \sum_{y \in \bT_m} \xi'(y)$.

Although we will not use this characterization, $G_{\bT_m}$ can be considered as a mean-zero version of the Green's function for the simple random walk on $\bT_m$ started from the origin and killed at a uniformly random point. To be precise, given $y \in \bT_m$, let $\tau_y$ be the first time $t \geq 0$ that a simple random walker started from the origin reaches $y$, and define $g_y(x)$ to be the expected number of times $0 \leq t < \tau_y$ that the walker visits site $x$. If $g(x) = \frac{1}{m^2} \sum_{y \in \bT_m} g_y(x)$, then $G_{\bT_m}(x) = \frac{1}{4}\left[ g(x) - \frac{1}{m^2}\sum_{x' \in \bT_m} g(x') \right]$.

In Section \ref{Representations_for_frequencies}, we specify for each frequency $\xi \in \hat{\sG_m}$ a particular choice of $\xi'$ such that the `distinguished prevector' $v = \Delta \xi'$ satisfies
\begin{equation}
\label{prevector_heuristic}
1 - |\hat{\mu}(\xi)| \asymp \frac{\|G_{\bT_m} * v \|_{L^2(\bT_m)}^2}{m^2}.
\end{equation}
Each prevector $v$ has mean zero because $v$ is in the image of $\Delta$. To find the absolute spectral gap of the Markov chain, which minimizes $1 - |\hat{\mu}(\xi)|$, we ask which mean-zero integer-valued vectors $v$ make $\|G_{\bT_m} * v \|_{L^2(\bT_m)}^2$ as small as possible.

It is profitable to think of $G_{\bT_m} * v$ as a linear combination of translates of discrete derivatives of $G_{\bT_m}$. For example, if $v(a,b) = -1$, $v(a-1,b) = 1$, and $v(i,j) = 0$ at all other $(i,j) \in \bT_m$, then
\begin{equation}
(G_{\bT_m} * v)(x_1,x_2) = G_{\bT_m}(x_1 + 1-a, x_2 - b) - G_{\bT_m}(x_1 - a, x_2 - b)
\end{equation}
which is the translation by $(a,b)$ of $D_1 G_{\bT_m}$.

The Laplacian operator $\Delta$ acts locally. Its inverse, convolution with $G_{\bT_m}$, is non-local but satisfies an approximate locality in that the discrete derivatives of $G_{\bT_m}$, like those of $G_{\zed^2}$, decay to zero. Using these decay estimates, we show in Section \ref{Determination_of_gap} that $\|G_{\bT_m} * v \|_{L^2(\bT_m)}^2$ is minimized when $G_{\bT_m} * v$ is an integer linear combination of the second derivatives $D_1^2 G_{\bT_m}$, $D_1 D_2 G_{\bT_m}$, $D_2^2 G_{\bT_m}$ and their translates. These lead to gaps of order $1/m^2$ in \eqref{prevector_heuristic}.

It follows from \eqref{prevector_heuristic}, an upper bound on $\|\Delta\|_{L^2 \to L^2}$, and the inequality
\begin{equation}
\|v\|_{L^2(\bT_m)}^2 = \|\Delta(G_{\bT_m} * v)\|_{L^2(\bT_m)}^2 \leq \|\Delta\|_{L^2 \to L^2}^2 \|G_{\bT_m} * v\|_{L^2(\bT_m)}^2
\end{equation}
that if the $L^2$ norm of the prevector $v$ is too high, then $v$ cannot generate the spectral gap. Proposition \ref{gap_achievers} shows that if the support of $v$ is too spread-out over $\bT_m$, then by the approximate locality of convolution with $G_{\bT_m}$, $v$ can be separated into widely spaced clusters whose contributions to $1 - |\hat{\mu}(\xi)|$ are nearly additive. Just keeping one of the clusters and zeroing out the rest of $v$ would produce a smaller gap. By this argument, the only prevectors with any chance of generating the spectral gap have bounded norm and bounded support, so the computation of the gap is reduced to a finite check.

To fill in the details of the proof, we require precise asymptotics for derivatives of $G_{\bT_m}$. We obtain these using a local limit theorem, which is proved in Appendix \ref{local_limit_theorem_appendix}. We also relate $G_{\bT_m}$ as $m \to \infty$ to $G_{\zed^2}$, which translates the finite check for the spectral gap into a minimization problem involving functions in $\ell^2(\zed^2)$ that are harmonic modulo $1$. The resulting search was performed using convex programming in the SciPy scientific computing package \cite{JOP01}, and is described in Appendix \ref{spectral_gap_appendix}. We find that for sufficiently large $m$, the gap is achieved for prevectors of the form $v(a,b) = v(a-1,b-1) = 1$, $v(a-1,b) = v(a,b-1) = -1$, $v(i,j) = 0$ elsewhere, which correspond to translates of $D_1 D_2 G_{\bT_m}$.

For Theorem \ref{mixing_time_theorem}, we prove cutoff in both total variation and $L^2$ at time $\gamma^{-1} m^2 \log m$. The necessary ingredients are a total variation lower bound and an $L^2$ upper bound on mixing time.

First, we use the coupon-collector argument mentioned earlier to reduce to the case where the starting state $\sigma$ is recurrent. Next we observe that due to translation, there are $m^2$ different prevectors $v$ whose corresponding frequencies $\xi$ achieve $1 - |\hat{\mu}(\xi)| = \gap_m$. The $L^2$ distance from the uniform distribution on $\sR_m$ of the chain started from $\sigma$ after $N$ steps satisfies
\begin{equation}
\left\|P_m^N \delta_{\sigma} - \bU_{\sR_m} \right\|_{L^2(d\bU_{\sR_m})}^2 = \sum_{\xi \in \hat{\sG}_m \setminus \{0\}}\left|\hat{\mu}(\xi) \right|^{2N} \geq m^2 (1 - \gap_m)^{2N}.
\end{equation}
Thus the chain cannot mix in $L^2$ before time
\begin{equation}
N = \frac{1}{\gap_m} \log m = \frac{1}{\gamma} m^2 \log m + o(m^2 \log m).
\end{equation}
We strengthen this to a lower bound on total variation mixing time by a second moment method due originally to Diaconis \cite{DS87,D88} that builds a distinguishing statistic out of the top eigenvectors of the chain. See Lemma \ref{lower_bound_lemma}. To apply this lemma, we require an upper bound on $|\hat{\mu}(\xi_1 - \xi_2)|$ when the frequencies $\xi_1,\xi_2$ (both of which achieve the spectral gap) come from prevectors $v_1,v_2$ whose supports are separated. Since the contributions of $v_1$ and $v_2$ are nearly additive, we have $1 - |\hat{\mu}(\xi_1 - \xi_2)| \approx 1 - 2 \cdot \gap_m$, which enables the argument to go through.

For the upper bound on the $L^2$ mixing time, we show that
\begin{equation}
\label{L2_dist}
\sum_{\xi \in \hat{\sG}_m \setminus \{0\}}\left|\hat{\mu}(\xi) \right|^{2N}
\end{equation}
tends to zero as $m \to \infty$ when $N = (1+\epsilon)\gamma^{-1} m^2 \log m$. Our argument uses an agglomeration scheme in which we partition the support of each prevector $v$ into widely spaced clusters. Lemma \ref{savings_lemma}, the main step in the proof, shows that each small cluster contributes additively to the gap $1 - |\hat{\mu}(\xi)|$. The earlier additivity results in Section \ref{Determination_of_gap}, most notably Proposition \ref{gap_achievers}, hold only for prevectors with bounded $L^1$ norm, so the extension to the general case requires new arguments. We use techniques from the theory of exponential sums, including van der Corput's inequality. As a consequence of the clustering scheme, we can control the number of distinct frequencies $\xi$ whose gap $1 - |\hat{\mu}(\xi)|$ might be small, giving the desired bound on \eqref{L2_dist}.

The cutoff argument may be considered an extension of the classical analysis 
of mixing on the hypercube \cite{DGM90}, and exploits the fact that the lattice 
which is quotiented to give the sandpile group is approximately cubic.  See 
\cite{H15} for analysis of some random walks on the cycle where the $L^1$ and 
$L^2$ cutoff times differ by a constant.

\subsection{Historical review} Sandpile dynamics on the square lattice were introduced by Bak, Tang, and Wiesenfeld \cite{BTW87,BTW88} as a model of self-organized criticality. Dhar \cite{D90} considered the case of an arbitrary finite underlying graph, proving many fundamental results. Subsequently, Dhar et al.~\cite{DRSV95} used harmonic modulo 1 functions (there called `toppling invariants') to analyze the algebraic structure of the sandpile group for rectangular subsets of $\zed^2$.

Sandpiles are examples of \emph{abelian networks}, which are systems
of communicating automata satisfying a local commutativity condition
\cite{D99,BL16}. By a theorem of Cairns \cite{C15}, an abelian network on $\zed^2$ can emulate a Turing
machine, as can a sandpile on $\zed^3$. In particular, for a periodic
configuration of sand on $\zed^3$ plus a finite number of additional
sand grains, the question of stabilization is algorithmically
undecidable! It is not known whether the same question is
undecidable on $\zed^2$. A related open problem, highlighted in
\cite{LMPU16}, is the following: ``Given a probability distribution
$\mu$ on $\zed$ (say, supported on $\{0,1,2,3,4\}$ with rational
probabilities), is it algorithmically decidable whether the i.i.d.~abelian sandpile on $\zed^2$ with marginal $\mu$ stabilizes almost
surely?'' Theorem \ref{quantitative_theorem} and its method of proof can be viewed as a slight
advance on this problem.

The question of stabilization of i.i.d.~sandpiles was posed by Meester and Quant \cite{MQ05} and by Fey and Redig \cite{FR05}. A fundamental result is the \emph{conservation of density} proved by Fey, Meester, and Redig \cite{FMR09}, which in particular implies the earlier result of \cite{FR05}: An i.i.d.~stabilizing sandpile $\sigma$ on $\zed^2$ must satisfy $\E[\sigma(x)] \leq 3$. 
To get strictly below $3$ in the upper bound of Theorem \ref{quantitative_theorem}, we use harmonic modulo 1 functions to construct additional conserved quantities; see Lemma \ref{pairing-invariant}.




Theorems \ref{mixing_time_theorem} and \ref{spectral_gap_theorem} are concerned with the sandpile Markov chain on the discrete torus $\bT_m$, whose stationary distribution 
is uniform on the (finite) set of recurrent states.  These finite Markov chains are related to sandpiles on the infinite grid $\zed^2$ by theorems of \cite{AJ04, JR08}.  Athreya and J\'{a}rai \cite{AJ04} proved that the restriction of a uniform recurrent sandpile on the $d$-dimensional cube $[-m,m]^d \cap \zed^d$ to any fixed finite subset of $\zed^d$ converges in law as $m \to \infty$.  Hence there is a limiting measure $\mu$ on recurrent sandpiles on $\zed^d$. By equality of the free and wired uniform spanning forests, replacing the cube with the $d$-dimensional discrete torus results in the same limit $\mu$. J\'{a}rai and Redig \cite{JR08} proved that in dimensions $d \geq 3$, a $\mu$-distributed sandpile plus one additional chip stabilizes almost surely. They used this fact to construct an ergodic Markov process on recurrent sandpiles on $\zed^d$ having $\mu$ as its stationary distribution. In dimension $2$, it is not known whether a $\mu$-distributed sandpile plus one additional chip stabilizes almost surely. (Possibly Lemma \ref{pairing-invariant} could help resolve this question.)
Some further studies of sandpile dynamics on $\zed^d$ are \cite{MRS04,JRS15,BHJ16}.

The mixing of the sandpile Markov chain on finite graphs arises in relating sandpiles with different boundary conditions: the dependence of observables such as the `density' (average amount of sand per vertex) on the boundary conditions is a symptom of slow mixing. In particular, the extra log factor in the mixing time $t_m^{\mix}$ of Theorem \ref{mixing_time_theorem} could be viewed as the cause for the failure of the `density conjecture' \cite{FLW10,L15}.

The proof of cutoff in Theorem \ref{mixing_time_theorem} estimates a 
significant piece of the spectrum of the transition kernel of the sandpile walk on the torus. See \cite{CE02,BS13} for further applications of spectral techniques 
related to sandpiles.

The eigenvectors and eigenvalues of the sandpile Markov chain on an arbitrary finite graph were characterized in \cite{JLP15} using `multiplicative harmonic functions' (these are complex exponentials of the harmonic modulo 1 functions, as explained in Section \ref{Random_walk_on_the_sandpile_group}).
In \cite{JLP15} it was shown that the sandpile Markov chain on any connected graph with $n$ vertices mixes in $O(n^3 \log n)$ steps, and that cutoff for the complete graph (both in total variation and in $L^2$) occurs at time $\frac{1}{4\pi^2} n^3 \log n$.

Regarding the discrete torus $\bT_m$, it was proved in \cite{JLP15} that the sandpile 
chain on any graph with $m^2$ vertices and maximum degree $4$ has spectral gap 
at least $1/(2m^2)$, and mixes in at most $\frac{5}{2} m^2 \log m$ steps. 
Theorems \ref{mixing_time_theorem} and \ref{spectral_gap_theorem} improve these results by obtaining asymptotics for the mixing time and spectral gap, and by demonstrating cutoff. We expect that our techniques can also prove cutoff for the sandpile chain on the finite box $[-m,m]^2 \cap \zed^2$, with boundary vertices identified as the sink, at a constant multiple of $m^2 \log m$ steps.

Schmidt and Verbitskiy \cite{SV09} characterize the set $\sH^1(\zed^2)$ of harmonic modulo 1 functions in $\ell^1(\zed^2)$ in terms of third derivatives of the Green's function $G_{\zed^2}$. 
Our Theorem \ref{H_p_theorem} provides a similar characterization of the sets $\sH^p(\zed^2)$, for $1 \leq p < \infty$. 

\subsection*{Organization}
Section \ref{background_section} fixes notation and provides background on discrete derivatives and Fourier transforms, the graph Laplacian and Green's function on $\zed^2$ and $\bT_m$, and results from the theory of exponential sums. Section \ref{classification_section} proves Theorem \ref{H_p_theorem}, while Section \ref{stability_section} proves Theorem \ref{quantitative_theorem}. Section \ref{sandpile_section} defines the sandpile group $\sG_m$ and its dual $\hat{\sG}_m$, describes the eigenvalues and eigenvectors of the sandpile Markov chain using $\hat{\sG}_m$, and shows that we may assume the starting state is recurrent when proving Theorem \ref{mixing_time_theorem}. Section \ref{spectral_gap_section} proves Theorem \ref{spectral_gap_theorem} and provides the main technical estimates needed for Theorem \ref{mixing_time_theorem}. Finally, Section \ref{proof_mixing_theorem_section} proves Theorem \ref{mixing_time_theorem}.

Appendix \ref{local_limit_theorem_appendix} proves a local limit theorem for repeated convolutions of the simple random walk measure on $\zed^2$ that is used to obtain asymptotics for derivatives of the discrete Green's function on $\bT_m$. Appendix \ref{spectral_gap_appendix} uses convex programming to find an exact formula for the leading constant $\gamma$ in the spectral gap and mixing time of the sandpile chain.

\section{Function spaces and conventions}
\label{background_section}
The additive character on $\bR/\zed$ is $e(x) = e^{2\pi i x}$. Its real part 
is denoted $c(x) = \cos 2\pi x$ and imaginary part $s(x) = \sin 2\pi x$. 
For real $x$, $\|x\|_{\bR/\zed}$ denotes the distance of $x$ to the nearest 
integer. 

We use the notations $A \ll B$ and $A = O(B)$ to mean that there is a constant 
$0<C<\infty$ such that $|A| < CB$, and $A \asymp B$ to mean $A \ll B \ll A$. A subscript such as $A \ll_R B$, $A = O_R(B)$ means that the constant $C$ depends on $R$. The notation $A = o(B)$ means that $A/B$ tends to zero.

Given a measurable space $(\sX, \sB)$, the \emph{total variation distance} between two probability measures $\mu$ and $\nu$ on $(\sX, \sB)$ is
\begin{equation}
\left\|\mu - \nu\right\|_{\TV} = \sup_{A \in \sB}|\mu(A) - \nu(A)|.
\end{equation}
If $\mu$ is absolutely continuous with respect to $\nu$, the total variation distance may be expressed as
\begin{equation}
\|\mu - \nu\|_{\TV} = \frac{1}{2} \int_{\sX}\left|\frac{d\mu}{d\nu}-1 \right|d\nu.
\end{equation}
In this case an $L^2(d\nu)$ distance may be defined  by
\begin{equation}
\left\|\mu - \nu\right\|_{L^2(d\nu)}^2 = \int_{\sX}\left(\frac{d\mu}{d\nu} - 1 \right)^2 d\nu,
\end{equation}
and Cauchy-Schwarz gives $\|\mu - \nu\|_{\TV} \leq \frac{1}{2}\|\mu - \nu\|_{L^2(d\nu)}$.

Consider $\zed^2$ and the discrete torus $\bT_m$ to be metric spaces with the graph distance given by the $\ell^1$ norm on $\zed^2$ and the quotient distance,  for $x, y \in \bT_m$,
\begin{equation}
\|x-y\|_1 = \min\{ \|x'-y'\|_1 : x',y' \in \zed^2, [x'] = x, [y'] = y \} 
\end{equation}
where $[x'],[y']$ are the images of $x',y'$ under the quotient map $\zed^2 \to \bT_m$.  The ball of radius $R > 0$ around a point $x$ is 
\begin{equation}
B_R(x) = \left\{y: \|y-x\|_1 \leq R \right\}.
\end{equation}
We also use $\|x\|_2$ to denote the $\ell^2$ norm of $x = (x_1,x_2) \in \zed^2$. The argument of $x$, denoted $\arg(x)$, is the angle $0 \leq \theta < 2\pi$ such that $(x_1,x_2) = (\|x\|_2 \cos \theta, \|x\|_2 \sin \theta)$.

Denote the usual function spaces
\begin{equation}
\ell^p\left(\zed^2\right) = \left\{f : \zed^2 \to \bC, \|f\|_p^p =  \sum_{x \in \zed^2} |f(x)|^p < \infty\right\}, \quad 1 \leq p < \infty
\end{equation}
and 
\begin{equation}
 L^p\left(\bT_m\right) = \left\{f: \bT_m \to \bC, \|f\|_p^p = \sum_{x \in \bT_m} |f(x)|^p \right\}, \quad 1 \leq p < \infty.
\end{equation}
The latter functions may be considered as functions on $\zed^2$ which are $m\zed^2$-periodic. Let $\ell^\infty(\zed^2)$ and $L^\infty(\bT_m)$ be the spaces of bounded functions on $\zed^2$ and $\bT_m$, with $\|f\|_{\ell^\infty(\zed^2)} = \sup_{x \in \zed^2} |f(x)|$ and $\|f\|_{L^\infty(\bT_m)} = \max_{x \in \bT_m} |f(x)|$.

On the torus, the subspace of mean zero 
functions is indicated by
\begin{equation}L_0^2(\bT_m) = \left\{f \in L^2(\bT_m): 
\sum_{x \in 
\bT_m} f(x) = 0\right\}.\end{equation} 
The notation $\zed^{\bT_m}_0$ is used for the integer-valued functions in $L^2_0(\bT_m)$.

On either $\zed^2$ or the torus, the standard basis vectors are written 
\begin{equation}
\e_{(i,j)}(k, \ell) = \one\{i=k\} \one\{j = \ell\}. 
\end{equation} 
For functions other than the standard basis vectors, the notation $f_x = f(x)$ is used interchangeably.

For $X = \zed^2$ or $X = \bT_m$ the support of a function on $X$ is  \begin{equation}\supp f=\{x \in 
X: f(x) 
\neq 0\}.\end{equation}
Given $(i,j) \in X$, the translation operator $T_{(i,j)}$ acts on 
functions by
\begin{equation}
 T_{(i,j)}f(k, \ell) = f(k-i, \ell-j).
\end{equation}
The convolution of functions $f \in \ell^1(\zed^2)$, $g \in \ell^\infty\left(\zed^2\right)$ or $f, g \in L^2(\bT_m)$ is given by
\begin{equation}
 (f*g)(i,j) = \sum_{(k, \ell)\in X} f(i-k, j-\ell)g(k,\ell)
\end{equation}
where again $X$ represents $\zed^2$ or $\bT_m$.

The averaging operator with respect to the uniform probability measure on 
$\bT_m$ is indicated by \begin{equation}\E_{x \in \bT_m}[f] = \frac{1}{m^2}\sum_{x 
\in \bT_m} f(x).\end{equation}

Given $x, y \in \bR/\zed$ and $f \in \ell^1(\zed^2)$, the Fourier transform of $f$ is
\begin{equation}
 \hat{f}(x,y) = \sum_{(i,j) \in \zed^2} f(i,j) e(-(ix + jy)).
\end{equation}
Given $x, y \in \zed/m\zed$ and $f \in L^2(\bT_m)$ the Fourier transform of $f$ is
\begin{equation}
 \hat{f}(x,y) = \sum_{(i,j) \in \bT_m} f(i,j) e\left(- \frac{ix + jy}{m} \right).
\end{equation}
The Fourier transform has the familiar property of carrying convolution to pointwise multiplication.  For $f \in \ell^2(\zed^2)$, Parseval's identity is
\begin{equation}
 \|f\|_2^2 = \int_{(\bR/\zed)^2} \left|\hat{f}(x,y)\right|^2 dx dy.
\end{equation}
For $f \in L^2(\bT_m)$ the corresponding identity is
\begin{equation}
 \|f\|_2^2 = \frac{1}{m^2} \sum_{x \in \bT_m} \left| \hat{f}(x)\right|^2.
\end{equation}

For a function $f$ on $\zed^2$ or $\bT_m$, the discrete derivatives $D_1 f(i,j)$, $D_2 f(i,j)$ are defined by \eqref{D_12}. Discrete differentiation is expressed as a convolution operator by 
introducing 
\begin{align}
 \delta_1(i,j) &= \left\{\begin{array}{lll}-1 && (i,j) = (0,0)\\ 1 && (i,j)
=(-1,0)\\ 0 && \text{otherwise,}\end{array}\right.\\ \notag
 \delta_2(i,j) &= \left\{\begin{array}{lll}-1 && (i,j) = (0,0)\\ 1 && (i,j)
=(0,-1)\\ 0 && \text{otherwise.}\end{array}\right.
\end{align}
For integers $a,b \geq 0$, one has
\begin{equation}
D_1^a D_2^b f = \delta_1^{*a}* \delta_2^{*b}* f.
\end{equation}

For $X = \zed^2$ or $\bT_m$ and functions $f_1,\ldots,f_n$ on $X$, recall that
\begin{equation}
\llangle f_1,\ldots,f_n \rrangle = \spaN_\zed\{ T_x f_1, \ldots, T_x f_n: x \in X\},
\end{equation}
where $\spaN_\zed$ refers to the finite integer span. It is convenient to introduce classes of integer-valued functions:
\begin{align}
\label{C_0123}
 C^0(X) &= \llangle \e_{(0,0)} \rrangle = \{f: X \to \zed, \|f\|_1<\infty\},\\
\notag C^1(X) &= \llangle \delta_1, \delta_2 \rrangle,\\
\notag C^2(X) &= \llangle \delta_1^{*2}, \delta_1 * \delta_2, \delta_2^{*2} \rrangle,\\
\notag C^3(X) &= \llangle \delta_1^{*3}, \delta_1^{*2} * \delta_2, \delta_1 * \delta_2^{*2}, \delta_2^{*3} \rrangle.
\end{align}
One has the equivalent characterizations
\begin{align}
C^1(X) &= \left\{ f \in C^0(X) : \sum_{x \in X} f(x) = 0 \right\}, \\
C^2(X) &= \left\{ f \in C^0(X) : \sum_{x \in X} f(x) = 0, \, \sum_{x \in X} f(x) x = 0 \right\} \label{C2}
\end{align}
and, for each $1 \leq k \leq 3$,
\begin{equation}
\label{C2_alt}
C^k(X) = \{\delta_1 * f + \delta_2 * g : f,g \in C^{k-1}(X)\}.
\end{equation}
Note the special cases $C^0(\bT_m) = \zed^{\bT_m}$ and $C^1(\bT_m) = \zed_0^{\bT_m}$.

\subsection{The graph Laplacian and Green's function}
The graph Laplacian $\Delta$ on either $\zed^2$ or $\bT_m$ is the second-order discrete differential operator defined by \eqref{Laplacian}. On $\zed^2$ its Fourier transform is given by
\begin{equation}
\label{Delta-Fourier}
\widehat{(\Delta f)}(x,y) = (4 - 2[c(x) + c(y)])\hat{f}(x,y), \quad x,y \in \bR / \zed,
\end{equation}
and on $\bT_m$ the Fourier transform is
\begin{equation}
\widehat{(\Delta f)}(x,y) = (4 - 2[c(x/m) + c(y/m)])\hat{f}(x,y), \quad x,y \in \zed / m\zed.
\end{equation}
\begin{lemma}
The graph Laplacians satisfy the operator bound
\begin{align}
\left\|\Delta\right\|_{\ell^2(\zed^2) \to \ell^2(\zed^2)}, \left\|\Delta\right\|_{L^2(\bT_m) \to L^2(\bT_m)}&\leq 8, \\\notag 
\left\|\Delta\right\|_{\ell^{\infty}(\zed^2) \to \ell^{\infty}(\zed^2)},  \left\|\Delta\right\|_{L^{\infty}(\bT_m) \to L^{\infty}(\bT_m)}&\leq 8.
\end{align}
\end{lemma}
\begin{proof}
The $\ell^\infty$  and $L^\infty$ estimates are immediate.  For $f \in \ell^2(\zed^2)$, by Parseval
\begin{equation}
\|\Delta f\|_2^2 = \int_{(\bR/\zed)^2} (4 - 2(c(x) + c(y)))^2 \left|\hat{f}(x,y)\right|^2 dxdy \leq 64 \|f\|_2^2.
\end{equation}
The bound on $L^2(\bT_m)$ is similar.
\end{proof}

On either $X = \zed^2$ or $X = \bT_m$, let $\nu$ be the probability measure given by \eqref{nu}, which drives simple random walk on $X$. The Green's function $G$ is a distribution on $C^1(X)$ given by
\begin{equation}
 G*f = \frac{1}{4}\sum_{n=0}^\infty (\nu^{*n} * f), \qquad f \in C^1(X).
\end{equation}
Since $\Delta f = 4\left(\delta_{(0,0)} - \nu\right) * f$, the formal computation
\begin{equation}
\Delta^{-1} = \frac{1}{4} \left( \delta_{(0,0)} - \nu \right)^{-1} = \frac{1}{4} \sum_{n=0}^\infty \nu^{*n} = G
\end{equation}
indicates that $G$ is in some sense the inverse of $\Delta$. Precise versions of this statement are given below.

On $\zed^2$, $G$ may be realized as the function \eqref{G_Z2_first}:
\begin{equation}
\label{G_Z2}
G_{\zed^2}(x) = \frac{1}{4}\sum_{n=0}^\infty \left[ \nu^{*n}(x)-\nu^{*n}(0,0) \right].
\end{equation}
This is a classical object of probability theory.  We quote the asymptotics from \cite{FU96}.
\begin{theorem}[\cite{FU96}, Remark 2]\label{greens_function_asymptotic}
Let $x = (x_1,x_2) \in \zed^2$.  There are constants $a,b>0$ such that 
\begin{equation} \label{G_expansion}
G_{\zed^2}(x) = \left\{\begin{array}{lll}0 && x = (0,0)\\ -\frac{\log \|x\|_2}{2\pi}  - a - b \frac{\frac{8x_1^2 x_2^2}{\|x\|_2^4}- 1}{\|x\|_2^2} + O(\|x\|_2^{-4})&& x \neq (0,0). \end{array}\right.
\end{equation}
\end{theorem}

It follows from \eqref{G_Z2} that
\begin{equation}
\Delta G_{\zed^2}(x) = \sum_{n=0}^\infty \left[ \nu^{*n}(x) - \nu^{*(n+1)}(x) \right] = \e_{(0,0)}(x),
\end{equation}
so $\Delta(G_{\zed^2} * f) = f$ for all $f \in C^0(\zed^2)$. The Fourier transform of $G_{\zed^2}$ is
\begin{equation}
\label{G_Z2-Fourier}
\hat{G}_{\zed^2}(x,y) = \frac{1}{4- 2\left(c(x) +c(y) \right)}. \end{equation}
When combined with \eqref{Delta-Fourier}, this shows that $G_{\zed^2} * \Delta f = f$ whenever $f \in \ell^2(\zed^2)$.

On $\bT_m$ a realization of $G$ as a function is obtained by \eqref{G_T_m_first}:
\begin{equation}
\label{G_T_m}
G_{\bT_m}(x) = \frac{1}{4}\sum_{n=0}^\infty \left(\nu^{*n}(x) - \frac{1}{m^2}\right).
\end{equation}

This converges absolutely, as is most easily checked by passing to frequency space, where the zeroth Fourier coefficient vanishes, and the remaining Fourier coefficients are convergent geometric series. Summing \eqref{G_T_m} over all $x \in \bT_m$ shows that $G_{\bT_m}$ has mean zero. As well, it follows from \eqref{G_T_m} that
\begin{equation}
\Delta G_{\bT_m}(x) = \sum_{n=0}^\infty \left[ \nu^{*n}(x) - \nu^{*(n+1)}(x) \right] = \e_{(0,0)}(x) - \frac{1}{m^2}.
\end{equation}
Therefore, $\Delta (G_{\bT_m} * f) = f - \E_{x \in \bT_m}[f]$ for any $f \in L^2(\bT_m)$. In particular, if $f \in L^2_0(\bT_m)$ then $\Delta (G_{\bT_m} * f) = f$. 

It is also true that $G_{\bT_m} * \Delta f = f - \E_{x \in \bT_m}[f]$ for all $f \in L^2(\bT_m)$. To prove this, observe that since $\Delta f \in L_0^2(\bT_m)$, $\Delta( G_{\bT_m} * \Delta f ) = \Delta f$. Only the constant functions are in the kernel of $\Delta$, so $G_{\bT_m} * \Delta f = f - c$ for some constant $c$. Since $G_{\bT_m} * \Delta f$ has mean zero, $c = \E_{x \in \bT_m}[f]$.

Both operators, $\Delta$ and convolution with $G_{\bT_m}$, have image $L_0^2(\bT_m)$. The two observations $\Delta (G_{\bT_m} * f) = f - \E_{x \in \bT_m}[f]$ and $G_{\bT_m} * \Delta f = f - \E_{x \in \bT_m}[f]$ imply that the composition in either order of the two operators results in orthogonal projection onto $L_0^2(\bT_m)$. Restricted to $L_0^2(\bT_m)$, the two operators are inverses. On $L^2(\bT_m)$, $G_{\bT_m}$ is the \emph{Moore-Penrose pseudoinverse} of $\Delta$.

We will require the following statements regarding discrete derivatives of $G_{\bT_m}$. Recall that the notation $A \ll_{a,b} B$ means that there is a constant $0 < C < \infty$ depending on $a,b$ such that $|A| \leq CB$.
\begin{lemma}\label{greens_function_estimate}
	For $a, b \in \zed_{\geq 0}$, $1 \leq a + b $, for $|i|, |j| \leq \frac{m}{2}$,
	\begin{equation}
	D_1^a D_2^b G_{\bT_m}(i,j) \ll_{a,b} \frac{1}{1+\left(i^2 + j^2 \right)^{\frac{a
				+ b}{2}}}.
	\end{equation}
\end{lemma}

In the case $a+b= 1$ the following asymptotic evaluation holds.
\begin{lemma}\label{green_function_differentiated_asymptotic}
	Let $m \geq 2$ and $0 \leq i, j \leq \frac{m}{2}$. Set $R = \sqrt{i^2+j^2}$.  
	There is a constant $c > 0$ such that, as $m \to \infty$, for $0<R < 
	\frac{m^{\frac{1}{2}}}{(\log m)^{\frac{1}{4}}}$,
	\begin{align}
	D_1 G_{\bT_m}(i,j) &= - \frac{ci}{i^2 + j^2} + O\left(\frac{1}{i^2 + 
		j^2}\right),\\ \notag
	D_2 G_{\bT_m}(i,j) &= - \frac{cj}{i^2 + j^2} + O\left(\frac{1}{i^2 + 
		j^2}\right).
	\end{align}
	
\end{lemma}

The proofs of Lemmas \ref{greens_function_estimate} and 
\ref{green_function_differentiated_asymptotic} are given in Appendix 
\ref{local_limit_theorem_appendix}.

\begin{lemma}\label{z_2_approx_lemma}
	If $a + b \geq 2$ then $D_1^a D_2^b G_{\zed^2}$ is in $\ell^2(\zed^2)$ and for 
	each fixed $i, j$, $D_1^a D_2^b G_{\bT_m}(i,j) \to D_1^a D_2^b G_{\zed^2}(i,j)$ as $m 
	\to \infty$.
\end{lemma}
\begin{proof}
	The Fourier transform of $D_1^a D_2^b G_{\zed^2}$ is given by, for $x,y \in \bR/\zed$, not both 0,
	\begin{equation}
	\widehat{D_1^a D_2^b G_{\zed^2}}(x,y) = \frac{ \left(e(x)-1 
		\right)^a\left(e(y)-1 \right)^b}{4- 2(c(x) + c(y))}  . 
	\end{equation}
	This function is bounded on $(\bR/\zed)^2$, which proves the first claim by 
	Parseval.  
	
	The Fourier transform of $D_1^a D_2^b G_{\bT_m}$ at frequency $(x,y) \in (\zed/m\zed)^2$ is given by
	$
	\widehat{D_1^a D_2^b G_{\zed^2}}\left(\frac{x}{m}, \frac{y}{m}\right).
	$
	Taking the group inverse Fourier transform,
	\begin{equation}
	D_1^a D_2^b G_{\bT_m}(i,j) = \frac{1}{m^2} \sum_{x, y \in \zed/m\zed} \widehat{D_1^a 
		D_2^b G_{\zed^2}}\left(\frac{x}{m}, \frac{y}{m}\right) e\left(\frac{ix +jy}{m} 
	\right).
	\end{equation}
	Treating this as a Riemann sum and letting $m \to \infty$ obtains the limit  
	\begin{equation}
	D_1^a D_2^b G_{\zed^2}(i,j) = \int_{(\bR/\zed)^2}\widehat{D_1^a D_2^b 
		G_{\zed^2}}\left(x, y\right) e\left(ix+jy \right)dxdy. \qedhere
	\end{equation}
\end{proof}

\subsection{Exponential sums}
This section collects the two results from the classical theory of exponential sums that are needed for the proof of Lemma \ref{savings_lemma}, which is the key ingredient in the upper bound of Theorem \ref{mixing_time_theorem}. For further references, see \cite{T86,IK04,M94}.

The first result is van der Corput's inequality. We will only need the case $H = 1$. See \cite{S04} for a motivation and proof of this statement.
\begin{theorem}[van der Corput's Lemma] Let $H$ be a positive integer.  Then 
	for 
	any complex numbers $y_1, y_2, ..., y_N$,
	\begin{equation}
	\left|\sum_{n=1}^N y_n\right|^2 \leq \frac{N+H}{H+1}\sum_{n=1}^N |y_n|^2 + 
	\frac{2(N+H)}{H+1} \sum_{h=1}^H \left(1 - 
	\frac{h}{H+1}\right)\left|\sum_{n=1}^{N-h}y_{n+h}\overline{y_n}\right|.
	\end{equation}
	
\end{theorem}

The second result treats summation of a linear phase function and is 
fundamental. 
\begin{lemma}
\label{geometric}
	Let $\alpha \in \bR\setminus \zed$ and let $N \geq 1$. Then
	\begin{equation}
	\left| \sum_{j = 1}^N e(\alpha j) \right| \ll \min\left(N, \|\alpha\|_{\bR/\zed}^{-1}\right).
	\end{equation}
	
\end{lemma}
\begin{proof}
	Sum the geometric series.
\end{proof}

\section{Classification of functions harmonic modulo 1}
\label{classification_section}

This section proves Theorem \ref{H_p_theorem}. Let $1 \leq p < \infty$, and recall that $\sH^p(\zed^2)$ is the set of all harmonic modulo 1 functions in $\ell^p(\zed^2)$. If $f \in \sH^p(\zed^2)$, then as $\|x\|_2 \to \infty$, $f(x) \to 0$ and therefore also $\Delta f(x) \to 0$. Since $\Delta f$ is integer-valued, it must be identically zero outside a ball of finite radius. Thus,
\begin{equation}
\label{Hp}
 \sH^p\left(\zed^2\right) = \left\{f \in \ell^p\left(\zed^2\right): \Delta f \in C^0(\zed^2) \right\}, \qquad 1 \leq p < \infty,
\end{equation}
where $C^0(\zed^2)$ is the space of integer-valued functions on $\zed^2$ with finite support, as in \eqref{C_0123}.

From Theorem \ref{greens_function_asymptotic}, we can derive the following formulas.

\begin{lemma}
\label{greens_function_derivs}
For nonzero $x = (x_1,x_2) \in \zed^2$, we have:
\begin{align}
\label{greens_deriv_formulas}
D_1 G_{\zed^2}(x_1,x_2) &= \frac{1}{2\pi} \cdot \frac{-x_1}{x_1^2 + x_2^2} + O\left( \|x\|_2^{-2} \right) \\
\notag D_2 G_{\zed^2}(x_1,x_2) &= \frac{1}{2\pi} \cdot \frac{-x_2}{x_1^2 + x_2^2} + O\left( \|x\|_2^{-2} \right) \\
\notag D_1^2 G_{\zed^2}(x_1,x_2) &= \frac{1}{2\pi} \cdot \frac{x_1^2 - x_2^2}{(x_1^2 + x_2^2)^2} + O\left( \|x\|_2^{-3} \right) \\
\notag D_1 D_2 G_{\zed^2}(x_1,x_2) &= \frac{1}{2\pi} \cdot \frac{2x_1 x_2}{(x_1^2 + x_2^2)^2} + O\left( \|x\|_2^{-3} \right) \\
\notag D_2^2 G_{\zed^2}(x_1,x_2) &= \frac{1}{2\pi} \cdot \frac{x_2^2 - x_1^2}{(x_1^2 + x_2^2)^2} + O\left( \|x\|_2^{-3} \right) \\
\notag D_1^a D_2^b G_{\zed^2}(x_1,x_2) &= O\left( \|x\|_2^{-3} \right), \qquad a+b = 3.
\end{align}
\end{lemma}

\begin{proof}
For $(x_1,x_2) \notin \{(0,0),(-1,0)\}$, Theorem \ref{greens_function_asymptotic} gives
\begin{multline}
D_1 G_{\zed^2}(x_1,x_2) \\
\begin{aligned}
&= -\frac{1}{4\pi} \log\left( 1 + \frac{2x_1 + 1}{x_1^2 + x_2^2} \right) + b \left[ \frac{1}{(x_1+1)^2 + x_2^2} - \frac{1}{x_1^2 + x_2^2} \right] \\
&\quad - 8b \left[ \frac{(x_1+1)^2 x_2^2}{[(x_1+1)^2 + x_2^2]^3} - \frac{x_1^2 x_2^2}{(x_1^2 + x_2^2)^3} \right] + O\left( \|x\|_2^{-4} \right).
\end{aligned}
\end{multline}
Expand the log term into a Taylor series. The quantities in brackets are $O\left( \|x\|_2^{-3} \right)$, as follows from using a common denominator. Thus
\begin{equation}
D_1 G_{\zed^2}(x_1,x_2) = -\frac{1}{2\pi} \cdot \frac{x_1}{x_1^2 + x_2^2} + \frac{1}{4\pi} \cdot \frac{x_1^2 - x_2^2}{(x_1^2 + x_2^2)^2} + O\left( \|x\|_2^{-3} \right).
\end{equation}
This and the analogous statement for $D_2 G_{\zed^2}$ prove the first two formulas in \eqref{greens_deriv_formulas}. The remainder of the lemma is proved similarly by taking further discrete derivatives; we omit the details.
\end{proof}

\begin{proof}[Proof of Theorem \ref{H_p_theorem}]
Using the terminology introduced in Section \ref{background_section}, the desired statements are:
\begin{align}
\label{H_p_2}
\sH^1(\zed^2) &= \{G_{\zed^2} * v : v \in C^3(\zed^2)\} + C^0(\zed^2) \\
\notag \sH^p(\zed^2) &= \{G_{\zed^2} * v : v \in C^2(\zed^2)\}, \quad 1 < p \leq 2 \\
\notag \sH^p(\zed^2) &= \{G_{\zed^2} * v : v \in C^1(\zed^2)\}, \quad 2 < p < \infty.
\end{align}
If $v \in C^k(\zed^2)$ for $1 \leq k \leq 3$, then $G_{\zed^2} * v$ is a finite integer linear combination of translates of $k$-th derivatives of $G_{\zed^2}$. It follows from Lemma \ref{greens_function_derivs} that $(G_{\zed^2} * v)(x) = O\left( \|x\|_2^{-k} \right)$, so $G_{\zed^2} * v \in \ell^p(\zed^2)$ as long as $p > 2/k$. Since $\Delta (G_{\zed^2} * v) = v$ is $\zed$-valued, we conclude that $G_{\zed^2} * v \in \sH^p(\zed^2)$. Along with the observation that $C^0(\zed^2) \subset \sH^1(\zed^2)$, this proves that for each line of \eqref{H_p_2}, the set on the left side contains the set on the right side.

We prove the forward inclusions in \eqref{H_p_2} in reverse order, from the third line to the first line. Let $f \in \sH^p(\zed^2)$ for some $1 \leq p < \infty$, and let $v = \Delta f$. By \eqref{Hp}, $v \in C^0(\zed^2)$, so there is $R > 0$ such that the support of $v$ is contained in the $\ell^1$-ball of radius $R$ about the origin. In
\begin{equation}
(G_{\zed^2} * v)(x) = \sum_{y \in \zed^2} G_{\zed^2}(x-y) v(y) = \sum_{\|y\|_1 \leq R} G_{\zed^2}(x-y) v(y),
\end{equation}
write $G_{\zed^2}(x) - G_{\zed^2}(x-y)$ as a sum of at most $R$ first derivatives of $G_{\zed^2}$, and use Lemma \ref{greens_function_derivs} to see that $G_{\zed^2}(x) - G_{\zed^2}(x-y) = O_R\left( \|x\|_2^{-1} \right)$. Thus, setting $B = \|v\|_1$ and $a = \sum_{y \in \zed^2} v(y)$,
\begin{equation}
(G_{\zed^2} * v)(x) = a G_{\zed^2}(x) + O_{B,R}\left( \|x\|_2^{-1} \right).
\end{equation}

Set $h(x) = (G_{\zed^2} * v)(x) - f(x)$, so that $\Delta h \equiv 0$. If $a \neq 0$, then as $\|x\|_2 \to \infty$, we have $(G_{\zed^2} * v)(x) \to -\sgn(a) \cdot \infty$ while $f(x) \to 0$, meaning that $h(x) \to -\sgn(a) \cdot \infty$. This violates the maximum principle, so $a = 0$ and $v \in C^1(\zed^2)$. We now have $h(x) \to 0$ as $\|x\|_2 \to \infty$, so again by the maximum principle, $h \equiv 0$ and $f = G_{\zed^2} * v$. This proves the forward inclusion in the third line of \eqref{H_p_2}.

Suppose that $p \leq 2$. Since $v \in C^1(\zed^2)$, we can write $v = \delta_1 * v_1 + \delta_2 * v_2$ for some $v_1,v_2 \in C^0(\zed^2)$. Then
\begin{align}
(G_{\zed^2} * v)(x) &= (D_1 G_{\zed^2} * v_1)(x) + (D_2 G_{\zed^2} * v_2)(x) \\
\notag &= \sum_{\|y\|_1 \leq R+1} D_1 G_{\zed^2}(x-y) v_1(y) + D_2 G_{\zed^2}(x-y) v_2(y) \\
\notag &= b_1 D_1 G_{\zed^2}(x) + b_2 D_2 G_{\zed^2}(x) + O_{B,R}\left( \|x\|_2^{-2} \right),
\end{align}
where each $b_i = \sum_{y \in \zed^2} v_i(y)$. In the last equality we wrote $D_i G_{\zed^2}(x) - D_i G_{\zed^2}(x-y)$ as a sum of $O(R)$ second derivatives of $G_{\zed^2}$ and used the bound from Lemma \ref{greens_function_derivs}. Again using Lemma \ref{greens_function_derivs}, we obtain for nonzero $x = (x_1,x_2)$ that
\begin{equation}
(G_{\zed^2} * v)(x) = \frac{1}{2\pi} \cdot \frac{-b_1 x_1 - b_2 x_2}{x_1^2 + x_2^2} + O_{B,R}\left( \|x\|_2^{-2} \right).
\end{equation}
Suppose $b_1$ and $b_2$ are not both zero. Then, there are $0 \leq \theta_1 < \theta_2 < 2\pi$ such that $(G_{\zed^2} * v)(x) \asymp \|x\|_2^{-1}$ for all $x \neq (0,0)$ with $\theta_1 \leq \arg(x) \leq \theta_2$. This contradicts the assumption that $f = G_{\zed^2} * v \in \ell^2(\zed^2)$. We conclude that $b_1 = b_2 = 0$, so $v_1,v_2 \in C^1(\zed^2)$ and therefore $v \in C^2(\zed^2)$ by \eqref{C2_alt}.

Finally, suppose that $p = 1$. Since $v \in C^2(\zed^2)$, we can write $v = \delta_1^{*2} * w_1 + (\delta_1 * \delta_2) * w_2 + \delta_2^{*2} * w_3$ for some $w_1,w_2,w_3 \in C^0(\zed^2)$. Set $c_i = \sum_{y \in \zed^2} w_i(y)$. By the same reasoning as in the previous case,
\begin{align}
&\quad\, (G_{\zed^2} * v)(x) \\
\notag &= c_1 D_1^2 G_{\zed^2}(x) + c_2 D_1 D_2 G_{\zed^2}(x) + c_3 D_2^2 G_{\zed^2}(x) + O_{B,R}\left( \|x\|_2^{-3} \right) \\
\notag &= \frac{1}{2\pi} \cdot \frac{(c_1 - c_3)(x_1^2 - x_2^2) +2c_2 x_1 x_2}{(x_1^2 + x_2^2)^2} + O_{B,R}\left( \|x\|_2^{-3} \right)
\end{align}
for all nonzero $x = (x_1,x_2)$. This implies that $c_1 = c_3$ and $c_2 = 0$; if not, the first term would have asymptotic order $\|x\|_2^{-2}$ for $\arg(x)$ in some range $[\theta_1,\theta_2]$, contradicting that $f \in \ell^1(\zed^2)$.

Set $c = c_1 = c_3$, and let $v' = \Delta \e_{(0,0)} = -\delta_1^{*2} * \e_{(1,0)} - \delta_2^{*2} * \e_{(0,1)}$. Then
\begin{equation}
v + cv' = \delta_1^{*2} * (w_1 - c\,\e_{(1,0)}) + (\delta_1 * \delta_2) * w_2 + \delta_2^{*2} * (w_3 - c\,\e_{(0,1)}).
\end{equation}
Since all three of $w_1 - c\,\e_{(1,0)}$, $w_2$, and $w_3 - c\,\e_{(0,1)}$ are in $C^1(\zed^2)$, we have $v + cv' \in C^3(\zed^2)$. As well, $G_{\zed^2} * cv' = c\,\e_{(0,0)} \in C^0(\zed^2)$. Hence
\begin{equation}
f = G_{\zed^2} * (v + cv' - cv') \in \{G_{\zed^2} * w : w \in C^3(\zed^2)\} + C^0(\zed^2),
\end{equation}
which completes the proof.
\end{proof}

\section{Stabilization on $\zed^2$}\label{stability_section}

Consider a sandpile $\sigma: \zed^2 \to \zed_{\geq 0}$. The \emph{parallel toppling} procedure attempts to stabilize $\sigma$ by defining a sequence of sandpiles $\sigma = \sigma^0, \sigma^1, \sigma^2, \ldots$ where $\sigma^{n+1}$ is obtained from $\sigma^n$ by simultaneously toppling all vertices $x$ with $\sigma^n(x) \geq 4$. Formally, set $v^n(x) = \one\{\sigma^n(x) \geq 4\}$ and define $\sigma^{n+1} = \sigma^n - \Delta(v^n)$. Define the sequence of \emph{odometer functions} $u^1,u^2,\ldots$ by $u^n = v^0 + v^1 + \cdots + v^{n-1}$, so that $u^n(x)$ is the number of times vertex $x$ has toppled in the first $n$ topplings.  In particular, $\|u^n\|_{\ell^\infty} \leq n$ and $\sigma^n = \sigma - \Delta(u^n)$.  It is shown in \cite{FMR09} that $\sigma$ stabilizes if and only if $u^n \uparrow u^\infty$ for some $u^\infty: \zed^2 \to \zed_{\geq 0}$, in which case the stabilization is given by $\sigma^\infty = \sigma - \Delta u^\infty$.

Our proof uses the following `conservation of density' result of \cite{FMR09}.
\begin{lemma}[\cite{FMR09}, Lemma 2.10]
\label{stabilization_prop}
 Let $(\sigma_x)_{x \in \zed^2}$ be i.i.d.~and stabilize almost surely, with 
stabilization $(\sigma^\infty_x)_{x \in \zed^2}$.  Then $\E[\sigma_0] = \E[\sigma^\infty_0]$.
\end{lemma}

In particular, if the i.i.d.~sandpile $\sigma$ stabilizes almost surely, then $\E[\sigma_0] \leq 3$.

We now show that if $\xi \in \sH^1(\zed^2)$, the pairing $\langle \sigma, \xi \rangle = \sum_{x \in \zed^2} \sigma(x) \xi(x)$ remains invariant modulo 1 when the sandpile $\sigma$ is stabilized.
\begin{lemma}
\label{pairing-invariant}
Let $(\sigma_x)_{x \in \zed^2}$ be an i.i.d.~sandpile which stabilizes almost surely, and let $\xi \in \sH^1(\zed^2)$.  Then
\begin{equation} \langle \sigma, \xi\rangle \equiv \langle \sigma^\infty, \xi\rangle \mod{1}, \qquad a.s.
\end{equation}
\end{lemma}
\begin{proof}
Lemma \ref{stabilization_prop} implies that $\E[\sigma_0] < \infty$.
Since $\xi \in \ell^1(\zed^2)$,
\begin{equation}
\E[ \langle \sigma, |\xi| \rangle ] = \sum_{x \in \zed^2} |\xi_x| \E[\sigma_x] = \|\xi\|_1 \E[\sigma_0] < \infty
\end{equation}
and so $\langle \sigma, \xi\rangle$ converges absolutely almost surely. 
Write $\sigma^n = \sigma - \Delta u^n$ and use self-adjointness of $\Delta$ to obtain
\begin{equation}
\label{finite-invariant}
\langle  \sigma^n, \xi\rangle = \langle \sigma - \Delta u^n, \xi\rangle = \langle  \sigma, \xi \rangle - \langle  u^n, \Delta \xi \rangle.
\end{equation}
Since $u^n$ is integer-valued, increasing and converges almost surely, while $\Delta\xi$ is integer-valued and has finite support, the increment $\langle u^n, \Delta\xi\rangle$ converges a.s.~to $\langle u^\infty, \Delta\xi \rangle \in \zed$.

Note that the parallel toppling property implies that, for $n \geq 0$, \begin{equation}\sigma^{n+1}(x) \leq \max\left(\sigma^{n}(x), 7\right).\end{equation}
Thus, whenever $\langle \sigma, |\xi| \rangle$ is finite and $\sigma$ stabilizes to $\sigma^\infty$,
\begin{equation}
\lim_{n \to \infty} \langle \sigma^n, \xi \rangle = \lim_{n \to \infty} \sum_{x \in \zed^2} \sigma^n(x) \xi_x = \sum_{x \in \zed^2} \lim_{n \to \infty} \sigma^n(x) \xi_x  = \langle \sigma^\infty, \xi \rangle
\end{equation}
where the second equality is justified by dominated convergence:
\begin{equation}
|\sigma^n(x) \xi_x| \leq \max(\sigma(x), 7) |\xi_x|, \qquad \sum_{x \in \zed^2} \max(\sigma(x), 7) |\xi_x| < \infty.
\end{equation}
Sending $n \to \infty$ in \eqref{finite-invariant} completes the proof.
\end{proof}

For definiteness, our argument uses the particular function
\begin{equation}
\xi = G_{\zed^2}* \delta_1^{*3} = D_1^3 G_{\zed^2},
\end{equation}
which is in $\sH^1(\zed^2)$ by Lemma \ref{greens_function_derivs}. The next lemma estimates the tail of $\|\xi\|_2^2$.

\begin{lemma}\label{xi_tail}
Let $R \geq 1$ be a parameter.  As $R \to \infty$,
\begin{equation} \label{four_thirds}
\sum_{x\in \zed^2\,:\, 0 < |\xi_x| < \frac{1}{2R}} |\xi_x|^2 \gg R^{-\frac{4}{3}}.
\end{equation}
\end{lemma}

\begin{proof}
Arguing as in Lemma \ref{greens_function_derivs}, we see that there are $0 \leq \theta_1 < \theta_2 < 2\pi$ such that, for nonzero $x \in \zed^2$ satisfying $\theta_1 \leq \arg(x) \leq \theta_2$,
\begin{equation}
|\xi_x| \asymp \|x\|_2^{-3}.
\end{equation}
Thus
\begin{equation}
\sum_{x \in \zed^2\,:\, 0 < |\xi_x| < \frac{1}{2R}} |\xi_x|^2 \gg \int_{R^{\frac{1}{3}}}^\infty \frac{dr}{r^5} \gg R^{-\frac{4}{3}}. \qedhere
\end{equation}
\end{proof}

As the proof below makes clear, an explicit constant in the lower bound \eqref{four_thirds} would lead to explicit values of $c,d$ in the statement of Theorem \ref{quantitative_theorem}. To obtain a fully quantitative version of Lemma \ref{xi_tail}, it would be enough to bound the error in \eqref{G_expansion} by finding an explicit $C > 0$ such that
\begin{equation} \label{G_error}
\left| G_{\zed^2}(x) + \frac{\log \|x\|_2}{2\pi}  + a + b \frac{\frac{8x_1^2 x_2^2}{\|x\|_2^4}- 1}{\|x\|_2^2} \right| \leq C\|x\|_2^{-4}
\end{equation}
for all $(0,0) \neq x \in \zed^2$. A result in this direction \cite[Section 4]{KS04} is that
\begin{equation}
\left| G_{\zed^2}(x) + \frac{\log \|x\|_2}{2\pi} \right| \leq 0.01721 \|x\|_2^{-2}.
\end{equation}
(Indeed, the constant $0.01721$ is optimal and an exact formula for it is given.) It is likely that extending the techniques developed in \cite{KS04} would lead to a bound of the form \eqref{G_error}, and thence to an explicit numerical bound in Theorem \ref{quantitative_theorem}.

\begin{proof}[Proof of Theorem \ref{quantitative_theorem}]
Consider the characteristic functions
\begin{equation}
\chi(\sigma;\xi) = \E\left[e^{-2\pi i \langle \sigma, \xi\rangle}\right], \qquad \chi(\sigma^\infty;\xi) = \E\left[e^{-2\pi i \langle\sigma^\infty, \xi\rangle}\right]. 
\end{equation}
Since $\langle  \sigma, \xi\rangle \equiv \langle  \sigma^\infty, \xi\rangle \bmod 1$ a.s., $\chi(\sigma;\xi) = \chi(\sigma^\infty; \xi)$.

Let $\E[\sigma_0] =\E[\sigma_0^\infty]= 3-\epsilon$. Using $|1- e^{2\pi i t}| \leq 2 \pi|t|$ and $\sum_{x \in \zed^2}\xi_x = 0$,
\begin{align}\label{sigma_infty_upper_bound}
|1 - \chi(\sigma^\infty; \xi)| &= \left| \E\left[1 - e^{-2\pi i \langle 
 \sigma^\infty - 3, \xi\rangle }\right] \right| \\
 &\notag\leq \E[2\pi |\left\langle \sigma^{\infty}-3, \xi \right\rangle |] \\
 &\notag\leq 2\pi \|\xi\|_1 \epsilon.
\end{align}
Thus, $|\chi(\sigma^\infty; \xi)| \geq 1 - 2\pi \|\xi\|_1 \epsilon$.

Meanwhile, since $(\sigma_x)_{x \in \zed^2}$ is i.i.d.,
\begin{equation}
\chi(\sigma; \xi) = \prod_{x \in \zed^2} \E\left[e^{-2\pi i \xi_x \sigma_0} \right].
\end{equation}
Use the inequality $-\log t \geq \frac{1-t^2}{2}$ in $0 < t \leq 1$ to obtain
\begin{equation}
-\log \left|\chi(\sigma; \xi) \right| \geq \frac{1}{2} \sum_{x \in \zed^2} \left(1 -\left|\E\left[e^{-2\pi i \xi_x \sigma_0} \right] \right|^2 \right).
\end{equation}
Let $X,X'$ be independent and distributed as $\sigma_0$. One has
\begin{equation}
\left|\E\left[e^{-2\pi i \xi_x \sigma_0} \right] \right|^2 = \E\left[ e^{-2\pi i \xi_x X} \right] \E\left[ e^{2\pi i \xi_x X'} \right] = \E\left[ e^{-2\pi i \xi_x (X-X')} \right].
\end{equation}
This quantity is equal to its real part $\E[c(\xi_x (X-X'))]$. (Recall $c(t) = \cos 2\pi t$.) Therefore, using $1- c(t) \geq 8 t^2$ for $|t| \leq \frac{1}{2}$,
\begin{align}
-\log \left|\chi(\sigma; \xi) \right| &\geq \frac{1}{2} \sum_{x \in \zed^2} \Big(1 - \E[c(\xi_x (X-X'))] \Big)\\
 &\notag\geq 4 \E\left[ \sum_{0 < |\xi_x (X-X')| < \frac{1}{2}} \xi_x^2 (X-X')^2 \right] \\
 &\notag = 4 \sum_{k=1}^\infty \E\left[ \one\{|X-X'| = k\} \sum_{0 < |\xi_x| < \frac{1}{2k}} \xi_x^2 k^2 \right].
\end{align}
Lemma \ref{xi_tail} now implies that
\begin{equation}
-\log \left|\chi(\sigma; \xi) \right| \gg \sum_{k=1}^\infty \E\left[ \one\{|X-X'| = k\} k^{2/3} \right] = \E[|X-X'|^{2/3}],
\end{equation}
and therefore
\begin{equation}
1 - \left|\chi(\sigma; \xi)\right| \gg \min\left(1, \E[|X-X'|^{2/3}] \right).
\end{equation}
The result follows on combining this with (\ref{sigma_infty_upper_bound}).
\end{proof}

\section{The sandpile group}\label{sandpile_section}
Recall the designations $\sR_m \subset \sS_m$ for the recurrent and stable states, respectively, of the sandpile model on $\bT_m$ with sink at $(0,0)$. Any sandpile $\sigma: \bT_m \setminus \{(0,0)\} \to \zed_{\geq 0}$ can be stabilized by repeatedly performing legal topplings until the resulting configuration is stable. By the abelian property \cite{D90}, the final state does not depend on the order in which the topplings are performed, and is called the \emph{stabilization} of $\sigma$.

If we view functions on $\bT_m$ as $m^2 \times 1$ column vectors, then the Laplacian operator $\Delta$ on $\bT_m$ can be considered as an $m^2 \times m^2$ matrix, so that for example $\Delta \zed^{\bT_m}$ is the integer span of the columns of $\Delta$. The null space of $\Delta$ is one-dimensional, and is spanned by the all-ones vector. 
The \emph{reduced Laplacian} $\Delta'$ is obtained by omitting the row and column
corresponding to the sink $(0,0)$, and is invertible.

The recurrent states $\sR_m$ of the sandpile model are naturally identified with the abelian group
\begin{equation}
\label{G_m}
\sG_m := \zed^{\bT_m\setminus\{(0,0)\}}/\Delta' \zed^{\bT_m\setminus\{(0,0)\}},
\end{equation}
which is the \emph{sandpile group} of $\bT_m$. Indeed, each equivalence class
\begin{equation}
\sigma + \Delta' \zed^{\bT_m\setminus\{(0,0)\}} \subset \zed^{\bT_m\setminus\{(0,0)\}}, \quad \sigma \in \zed^{\bT_m\setminus\{(0,0)\}},
\end{equation}
contains exactly one recurrent sandpile \cite{HLMPPW08}. Addition in $\sG_m$ corresponds via this bijection to the operation on $\sR_m$ of pointwise addition followed by stabilization.

The \emph{sandpile Markov chain} has state space $\sS_m$ and transition operator $P_m$. To take a single step from a sandpile $\sigma$, choose a site $x \in \bT_m$ uniformly at random. If $x \neq (0,0)$, replace $\sigma$ with the stabilization of $\sigma + \e_x$; if $x = (0,0)$, remain at $\sigma$. The recurrent states of the chain are precisely $\sR_m$, and the chain restricted to $\sR_m$ is a random walk on the group $\sG_m$. See \cite{JLP15}, which develops this construction in the setting of an arbitrary underlying graph, for further background.

Using \eqref{G_m}, the matrix-tree theorem implies that $\sG_m$ is in bijection with the spanning trees of $\bT_m$. It is shown in \cite{JLP15} that $|\sG_m| = \exp\left(\left(\frac{4\beta(2)}{\pi}+o(1)\right)m^2 \right)$ where $\beta(2)$ is the Catalan constant,
\begin{equation}
\frac{4 \beta(2)}{\pi} = 1.1662\ldots.
\end{equation}
Thus the recurrent states make up an exponentially small fraction of the $4^{m^2-1}$ stable states.

The following proposition bounds the hitting time started from a deterministic stable state to reach a recurrent state.  
\begin{proposition}\label{recurrent_state_hitting_time_proposition}
	There is a constant $C > 0$ such that, as $m \to \infty$, for any stable state 
	$\sigma \in \sS_m$, if $n > C m^2 \sqrt{\log m }$ then 
	\[
	\Prob\left(P_m^n \delta_{\sigma} \in \sR_m \right) = 1 - o(1).
	\]
\end{proposition}

\begin{rem}
Starting from $\sigma = 0$, at least order $m^2$ steps are necessary to reach a recurrent state, since only one chip is added at a time. We do not claim that the extra factor of $\sqrt{\log m}$ above is optimal. Because we will show that the mixing time of the sandpile chain has order $m^2 \log m$, the bound in Proposition \ref{recurrent_state_hitting_time_proposition} is sufficient for understanding the mixing behavior.
\end{rem}

\begin{proof}
	We make two initial observations.  First, any state satisfying $\sigma \geq 3$ 
	can be toppled to a stable recurrent state.  This is because such a state can 
	evidently be reached from a recurrent state.  Also, by performing a sequence of 
	topplings, a single vertex with allocation $h$ can be toppled to produce a disc 
	of radius $\gg \sqrt{h}$ with height at least 3.  This follows as a simple 
	consequence of the analysis in \cite{PS13}, which studies the limiting shape of 
	the configuration obtained by repeated toppling of a pile at a single vertex.
	
	Let $A$ be an integer, $A \ll \sqrt{\log m}$, and drop $n \sim \Poisson(Am^2)$ grains of sand on the torus, while performing no 
	topplings.  Note that this is the same as independently dropping $\Poisson(A)$ grains of sand on each vertex. Also, $n < 2Am^2$ with probability $1 - o(1)$.
	
	 The probability that a non-sink vertex $x$ has height at most $a$ is
	\begin{equation}\label{single_point}
	\Prob(h_x \leq a) = e^{-A}\sum_{j=0}^a \frac{A^j}{j!}.
	\end{equation}
	For $ a < \frac{A}{2}$ we obtain
	\begin{align*}
	\Prob(h_x \leq a) \asymp \frac{A^a}{a!}\exp\left(-A \right).
	\end{align*}
	If $x_1, x_2, \ldots, x_s$ denote the points of a disc of area $s \gg a$, then, by 
	independence,
	\begin{align}\label{hole_probability}
	\Prob\left(\bigwedge_{i=1}^s (h_{x_i} \leq a) \right) \leq \exp\left(-sA + sa 
	\log \frac{A}{a} + s(a + O(1)) \right).
	\end{align}
	Choose $s, a \asymp \sqrt{\log m}$ such that a point of height $a$ in a disc 
	of area $s$ topples to cover the disc. Then choose $A$ a sufficiently large 
	constant times $\sqrt{\log m}$ so that the probability of 
	(\ref{hole_probability}) is $o\left(1/m^2 \right)$.  It follows that 
	with probability $1-o(1)$, the event (\ref{hole_probability}) does not occur 
	for any disc on the torus at distance $\gg \sqrt{\log m}$ from the sink. The 
	sites closer to the sink have height $\geq 3$ with probability $1 -o(1)$ by 
	estimating using (\ref{single_point}) and a union bound.
\end{proof}

The following proposition reduces the statements in Theorem \ref{mixing_time_theorem} to estimates started from the fixed recurrent state $\sigma \equiv 3$.
\begin{proposition}
\label{reduction_proposition}
For each constant $C > 0$, for $t = C m^2 \log m$, as $m \to \infty$,
\begin{equation}
\sup_{\sigma_0 \in \sS_m}\Big|\left\|P_m^{t}\delta_{\sigma_0} -\bU_{\sR_m}\right\|_{\TV} - \left\|P_m^{t}\delta_{\sigma \equiv 3}- \bU_{\sR_m} \right\|_{\TV} \Big| = o(1).
\end{equation}
\end{proposition}
\begin{proof}
	Given $\sigma_0 \in \sS_m$, let $\sigma_1 \in \sR_m$ be the unique recurrent state in the equivalence class $\sigma_0 + \Delta' \zed^{\bT_m\setminus\{(0,0)\}}$. By Proposition \ref{recurrent_state_hitting_time_proposition}, $P_m^t\delta_{\sigma_0}(\sR_m) = 1 - o(1)$, and thus $\left\|P_m^t \delta_{\sigma_0} - P_m^t \delta_{\sigma_1}\right\|_{\TV} = o(1)$. Since the chain restricted to $\sR_m$ is transitive, it follows from the triangle inequality that
	\begin{align}
	&\Big|\left\|P_m^{t}\delta_{\sigma_0} -\bU_{\sR_m}\right\|_{\TV} - \left\|P_m^{t}\delta_{\sigma \equiv 3}- \bU_{\sR_m} \right\|_{\TV} \Big| \\
	\notag &\quad = \Big|\left\|P_m^{t}\delta_{\sigma_0} -\bU_{\sR_m}\right\|_{\TV} - \left\|P_m^{t}\delta_{\sigma_1}- \bU_{\sR_m} \right\|_{\TV} \Big| \\
	\notag &\quad \leq \left\|P_m^t \delta_{\sigma_0} - P_m^t \delta_{\sigma_1}\right\|_{\TV} = o(1).	\qedhere
	\end{align}
\end{proof}

\subsection{Random walk on the sandpile group}
\label{Random_walk_on_the_sandpile_group}
Going forward we assume that the sandpile Markov chain is started from the deterministic recurrent state $\sigma \equiv 3$ so that the dynamics is reduced to a random walk on the abelian group $\sG_m$. In general, for any random walk on a finite abelian group $\sG$ driven by the measure $\mu$, the eigenfunctions of the transition kernel are given by the dual group, which is the additive group of characters $\hat{\sG} = \{\xi : \sG \to \bR / \zed\}$. If $\xi \cdot g$ denotes the image of $g \in \sG$ under $\xi \in \hat{\sG}$, then the eigenfunction corresponding to $\xi$ is $f_\xi(g) = e(\xi \cdot g)$. The corresponding eigenvalue is the Fourier coefficient of $\mu$ at frequency $\xi$, namely $\hat{\mu}(\xi) = \sum_{g \in \sG} \mu(g) e(\xi \cdot g)$.

The sandpile chain on $\sG_m$ is driven by the measure
\begin{equation}
 \mu := \frac{1}{m^2}\left(\delta_0 + \sum_{x \in \bT_m \setminus \{(0,0)\}}\delta_{\e_x}
\right)
\end{equation}
where, technically, $\e_x$ refers to the equivalence class $\e_x + \Delta' \zed^{\bT_m \setminus \{(0,0)\}} \in \sG_m$, and $0 \in \sG_m$ is the identity. The dual group of $\sG_m$ is
\begin{equation}
\label{hat_G_m}
 \hat{\sG}_m =
(\Delta')^{-1}\zed^{\bT_m\setminus\{(0,0)\}}/\zed^{\bT_m\setminus\{(0,0)\}}.
\end{equation}
This can be seen by dualizing \eqref{G_m}; a bare-hands proof is given in Section 3 of \cite{JLP15}. To define the meaning of $\xi \cdot g$ in this setting, we can view each frequency $\xi \in \hat{\sG}_m$ as a function from $\bT_m \setminus \{(0,0)\}$ to $\bR / \zed$, and each group element $g \in \sG_m$ as an equivalence class $\sigma + \Delta' \zed^{\bT_m \setminus \{(0,0)\}}$, where $\sigma \in \zed^{\bT_m \setminus \{(0,0)\}}$. Then, $\xi \cdot g = \sum_{x \in \bT_m \setminus \{(0,0)\}} \xi_x \sigma_x \in \bR / \zed$, whose value does not depend on the choice of the representative $\sigma$ in the equivalence class. The eigenvalue corresponding to $\xi$ is
\begin{equation}\label{def_fourier_coefficient}
 \hat{\mu}(\xi) = \frac{1}{m^2}\left(1 + \sum_{x \in \bT_m \setminus \{(0,0)\}}e(\xi_x) \right).
\end{equation}

Given $\xi: \bT_m \setminus \{(0,0)\} \to \bR / \zed$, which may or may not be in $\hat{\sG}_m$, set $v = \Delta' \xi$ (which is also $\bR / \zed$-valued). Extend $\xi$ to the domain $\bT_m$ by setting $\xi(0,0) = 0$. Then $\Delta \xi(x) = v(x)$ for all $x \in \bT_m \setminus \{(0,0)\}$, and since the columns of $\Delta$ all sum to zero, $\Delta \xi(0,0) = -\sum_{x \neq (0,0)} v(x)$. From \eqref{hat_G_m}, $\xi \in \hat{\sG}_m$ if and only if $v \equiv 0$, which holds if and only if $\Delta \xi \equiv 0$. This justifies the description of $\hat{\sG}_m$ in Section \ref{Discussion of method} as the additive group of functions $\xi : \bT_m \to \bR / \zed$ such that $\xi(0,0) = 0$ and $\Delta \xi \equiv 0$ in $\bR / \zed$. From this point forward, when we refer to a frequency $\xi \in \hat{\sG}_m$, we mean a function that meets these conditions.

In \cite{JLP15}, $\hat{\sG}_m$ was identified with the group of `multiplicative harmonic functions,' which in the present setting are the maps from $\bT_m$ to $\bC^*$ given by $x \mapsto e(\xi_x)$.

Abusing notation slightly, define  for any $\bR$-valued or $\bR / \zed$-valued function $\xi$ on $\bT_m$,
\begin{equation}\hat{\mu}(\xi) := \E_{x \in \bT_m}\left[e(\xi_x)\right].\end{equation}
When in fact $\xi \in \hat{\sG_m}$, this definition agrees with \eqref{def_fourier_coefficient}.

\subsection{Representations for frequencies}
\label{Representations_for_frequencies}
  We use a concrete description of
the frequencies in terms of the Green's function, which
associates to the frequencies an approximate partial ordering.
To describe this, given $\xi \in \hat{\sG}_m$ recall that a `prevector' for $\xi$ is any integer-valued vector $\Delta \xi'$, where $\xi' : \bT_m \to \bR$ reduces mod $\zed$ to $\xi$. We choose a particular representative $\xi' : \bT_m \to (-1,1)$ by letting 
\begin{equation}
 C(\xi) = \frac{1}{2\pi} \arg \left( \hat{\mu}(\xi) \right) \in 
\textstyle{\left[ -\frac{1}{2}, \frac{1}{2} \right)}
\end{equation}
and choosing each $\xi_x' \in
\left(C(\xi) - \frac{1}{2}, C(\xi) + \frac{1}{2}\right]$. The `distinguished prevector' of $\xi$ is then given by
\begin{equation}
v = v(\xi) := \Delta \xi'.
\end{equation}
Note that $v : \bT_m \to \zed$ has mean zero and satisfies $\|v\|_{L^\infty} \leq 3$.

\begin{lemma}\label{l_2_fourier_lemma}
For every $\xi \in \hat{\sG}_m$, the distinguished prevector of $\xi$ satisfies
\begin{equation}
 1 - \left|\hat{\mu}(\xi)\right| \gg \frac{\|v(\xi)\|_2^2}{m^2} \geq 
\frac{\|v(\xi)\|_1}{m^2}.
\end{equation}
\end{lemma}

\begin{proof}
Choose $\xi'$ as above, and define $\xi^* : \bT_m \to \left(-\frac{1}{2}, \frac{1}{2}\right]$ by
\begin{equation} 
 \xi^*_x = \xi'_x - C(\xi),
\end{equation}
so that $\Delta \xi^* = 
\Delta \xi' = v$ and
\begin{equation}
0 \leq |\hat{\mu}(\xi)| = \frac{1}{m^2} \sum_{x \in \bT_m} e\left( \xi^*_x 
\right) = \frac{1}{m^2} \sum_{x \in \bT_m} c\left( \xi^*_x \right).
\end{equation}
Approximating $1 - c(t) \gg t^2$ uniformly for $|t| \leq \frac{1}{2}$ yields
\begin{equation}
\label{xi^*_bound}
1 - |\hat{\mu}(\xi)| \gg \frac{\|\xi^*\|_2^2}{m^2}.
\end{equation}
Since $\Delta$ is bounded from $L^2(\bT_m) \to 
L^2(\bT_m)$,
\begin{equation}
\frac{\|v\|_2^2}{m^2} = \frac{\|\Delta \xi^*\|_2^2}{m^2} \ll 
\frac{\|\xi^*\|_2^2}{m^2} \ll 1 - |\hat{\mu}(\xi)|
\end{equation}
as desired. Finally, $\|v\|_2^2 \geq \|v\|_1$ since $v$ is integer-valued.
\end{proof}

To go in the reverse direction, for any $v \in \zed_0^{\bT_m}$ define $\overline{\xi} = G_{\bT_m} * v$, so that $\Delta \overline{\xi} = v$. Let $\xi''_x = \overline{\xi}_x - \overline{\xi}_{(0,0)}$, and set $\xi = \xi(v)$ to be the reduction mod $\zed$ of $\xi''$. Since $\xi''_{(0,0)} = 0$ and $\Delta \xi'' = v$, which is $\zed$-valued, it follows that $\xi \in \hat{\sG}_m$.

If $\xi_0 \in \hat{\sG}_m$ and $v = \Delta \xi'$ is any prevector of $\xi_0$, then $\xi(v) = \xi_0$; this is because $\Delta(\xi' - \xi'') \equiv 0$, so $\xi' - \xi'' \equiv c$ for some $c \in \bR$, and in fact $c = \xi'_{(0,0)} - \xi''_{(0,0)} \in \zed$. Also, if $v_0 \in \zed_0^{\bT_m}$ and $v$ is any prevector of $\xi(v_0)$, then $v_0 - v \in \Delta \zed^{\bT_m}$.

\begin{lemma}
\label{xi_bar}
Given $\xi \in \hat{\sG}_m$, let $v$ be any prevector of $\xi$ and let $\overline{\xi} = G_{\bT_m} * v$. Then $|\hat{\mu}(\xi)| = |\hat{\mu}(\overline{\xi})|$. If $v$ is the distinguished prevector of $\xi$, then in addition
\begin{equation}
\label{asymp_gap}
1 - |\hat{\mu}(\xi)| \asymp \frac{\|\overline{\xi}\|_2^2}{m^2}.
\end{equation}
\end{lemma}

Equation \eqref{asymp_gap} is equivalent to Theorem 3.8 in \cite{JLP15}, and the argument below is the same as the proof given there.

\begin{proof}
Let $v = \Delta \xi'$, where $\xi': \bT_m \to \bR$ reduces mod $\zed$ to $\xi$. Then $\overline{\xi} = G_{\bT_m} * \Delta \xi' = \xi' - c$ where $c = \E_{x \in \bT_m}[\xi']$, so
\begin{equation}
\hat{\mu}(\overline{\xi}) = \E_{x \in \bT_m}[e(\xi'_x - c)] = e(-c) \hat{\mu}(\xi') = e(-c) \hat{\mu}(\xi)
\end{equation}
and therefore $|\hat{\mu}(\overline{\xi})| = |\hat{\mu}(\xi)|$.

To prove the upper bound in \eqref{asymp_gap},
\begin{align}
1 - |\hat{\mu}(\xi)| &= 1 - |\hat{\mu}(\overline{\xi})| \leq 1 - \RE \hat{\mu}(\overline{\xi}) = \frac{1}{m^2} \sum_{x \in \bT_m} \left[ 1 - c(\overline{\xi}_x) \right] \\
\notag &\ll \frac{1}{m^2} \sum_{x \in \bT_m} |\overline{\xi}_x|^2 = \frac{\|\overline{\xi}\|_2^2}{m^2}.
\end{align}
For the lower bound, define $\xi^*$ as in the proof of Lemma \ref{l_2_fourier_lemma} and observe that $\overline{\xi} = G_{\bT_m} * \Delta \xi^* = \xi^* - \E_{x \in \bT_m}[\xi^*]$ is the orthogonal projection of $\xi^*$ onto $L_0^2(\bT_m)$. Thus $\|\overline{\xi}\|_2^2 \leq \|\xi^*\|_2^2$, and the result follows from \eqref{xi^*_bound}.
\end{proof}

\section{Spectral estimates}\label{spectral_gap_section} 
This section reduces the determination of the spectral gap to a finite check, and provides additive savings estimates for separated spectral components. Lemma \ref{l_2_fourier_lemma} implies that each nonzero frequency $\xi \in \hat{\sG}_m$ satisfies $1 - |\hat{\mu}(\xi)| \gg 1/m^2$, and if $1 - |\hat{\mu}(\xi)| \leq c/m^2$, then the $L^1$ norm of the distinguished prevector $v(\xi)$ must be bounded by a constant depending only on $c$. Section \ref{Determination_of_gap} develops tools to deal with prevectors that have bounded $L^1$ norm, providing control over those frequencies that achieve the spectral gap or approach it to within a constant factor. This proves Theorem \ref{spectral_gap_theorem} and does most of the work for the lower bound in Theorem \ref{mixing_time_theorem}.

Section \ref{small_phase_section} extends the analysis to prevectors whose $L^1$ norm increases with $m$, but which are sparse enough that their supports can be partitioned into widely separated clusters. This provides the main ingredient for the upper bound in Theorem \ref{mixing_time_theorem}. As we will show in Section \ref{proof_mixing_theorem_section}, if $\xi$ is a frequency for which $v(\xi)$ is not sparse, then the Lemma \ref{l_2_fourier_lemma} lower bound on $1 - |\hat{\mu}(\xi)|$ shows that the contribution of $\xi$ is negligible when computing the mixing time.

To fix ideas, given $\xi \in \hat{\sG}_m$ recall that $\hat{\mu}(\xi) = \E_{x \in \bT_m}\left[e(\xi_x) \right]$. For any subset $S \subset \bT_m$, it is evident that
\begin{equation}
\left| \sum_{x \in S} e(\xi_x) \right| \leq |S|.
\end{equation}
The `savings from $S$' for the frequency $\xi$, denoted by $\sav(\xi;S)$, is the amount by which the left side falls short of this upper bound:
\begin{equation}
\label{savings_def}
\sav(\xi; S) := |S| - \left|\sum_{x \in S}e(\xi_x)\right|.
\end{equation}
By the triangle inequality, if $S_1,S_2 \subset \bT_m$ are disjoint then
\begin{equation}
\sav(\xi; S_1) + \sav(\xi; S_2) \leq \sav(\xi; S_1 \cup S_2).
\end{equation}
The `total savings' for $\xi$ is defined by
\begin{equation}
\label{savings_def_2}
\sav(\xi) := \sav(\xi; \bT_m) = m^2 - \left|\sum_{x \in \bT_m} e(\xi_x) \right|
\end{equation}
and satisfies
\begin{equation}
1 - \left|\hat{\mu}(\xi)\right| = \frac{\sav(\xi)}{m^2}.
\end{equation}
The notion of savings is well-suited for proving lower bounds on the gap $1 - |\hat{\mu}(\xi)|$. Specifically, if $S_1,\ldots,S_k$ are disjoint subsets of $\bT_m$ then
\begin{equation}
1 - |\hat{\mu}(\xi)| \geq \frac{1}{m^2} \sum_{i=1}^k \sav(\xi; S_i).
\end{equation}
The spectral gap of the sandpile Markov chain is
\begin{equation}
\gap_m = \min_{0 \neq \xi \in \hat{\sG}_m} \frac{\sav(\xi)}{m^2}.
\end{equation}
Observe that if $v$ is the distinguished prevector of $\xi \in \hat{\sG}_m$, then Lemma \ref{l_2_fourier_lemma} gives $\sav(\xi)\gg \|v\|_1$.  Also, given a set $S \subset \bT_m$ and a function $w$ on $\bT_m$, write $w|_S$ for the function which is equal to $w$ on $S$ and 0 on $S^c$.

\subsection{Determination of spectral gap up to finite check}
\label{Determination_of_gap}

Given constants $B,R > 0$, define the finite set
\begin{equation}
\sC(B,R) := \{v \in C^2(B_R(0)) :  \|v\|_1 \leq B \}.
\end{equation}
Here $B_R(0)$ is the $\ell^1$ ball of radius $R$ about $0$ in $\zed^2$, and $C^2(\cdot)$ is given by \eqref{C2}. Since $B_R(0)$ embeds into $\bT_m$ for each $m > 2R$, we can view each $v \in \sC(B,R)$ as an element either of $C^2(\zed^2)$ or of $C^2(\bT_m)$ by setting $v \equiv 0$ outside $B_R(0)$.

For any $\bR$- or $\bR / \zed$-valued function $\xi$ on $\zed^2$, define the functional
\begin{equation}
\label{f_define}
f(\xi) := \sum_{x \in \zed^2} (1 - c(\xi_x)).
\end{equation}
We will see that this is the appropriate analogue to savings for functions on $\zed^2$. If $v \in C^2(\zed^2)$, then $f(G_{\zed^2} * v) < \infty$ by the bound $1 - c(t) \ll t^2$ combined with Lemma \ref{z_2_approx_lemma} or Lemma \ref{greens_function_derivs}. For such $v$, $f(G_{\zed^2} * v) = 0$ if and only if $G_{\zed^2} * v$ is $\zed$-valued. Since $G_{\zed^2} * v \in \ell^2(\zed^2)$, if it is $\zed$-valued then it must be finitely supported, and in addition we have $\Delta (G_{\zed^2} * v) = v$. Thus, $f(G_{\zed^2} * v) = 0$ precisely for those $v$ in the subset
\begin{equation}
\cI := \{\Delta w : w \in C^0(\zed^2)\} \subset C^2(\zed^2).
\end{equation}
If $v,v' \in C^2(\zed^2)$ and $v - v' \in \cI$, then $f(G_{\zed^2} * v) = f(G_{\zed^2} * v')$.

Set
\begin{equation}
\label{def_gamma}
\gamma := \inf\left\{ f(G_{\zed^2} * v) : v \in C^2(\zed^2) \,\setminus\, \cI \right\}.
\end{equation}
The following are the main results of this section. Together with the computation in Appendix \ref{spectral_gap_appendix}, they lead to a quick proof of Theorem \ref{spectral_gap_theorem}.

\begin{proposition}
\label{gap_achievers}
We have $\gamma > 0$, and there exist constants $B_0,R_0 > 0$ such that:
\begin{enumerate}[label=\arabic*.]
\item For sufficiently large $m$, any $\xi \in \hat{\sG}_m$ that achieves the spectral gap, $\sav(\xi) = m^2 \gap_m$, has a prevector $v$ which is a translate of some $v' \in \sC(B_0,R_0) \subset C^2(\bT_m)$.
\item For any $v \in C^2(\zed^2)$ satisfying $f(G_{\zed^2} * v) < \frac{3}{2}\gamma$, there exists $v' \in \sC(B_0,R_0) \subset C^2(\zed^2)$ such that a translate of $v'$ differs from $v$ by an element of $\cI$. In particular, $f(G_{\zed^2} * v) = f(G_{\zed^2} * v')$.
\end{enumerate}
\end{proposition}

\begin{proposition}
\label{T_m_to_Z^2}
Fix $B,R_1 > 0$. For any $v \in \sC(B,R_1)$ and $m > 2R_1$, let $\xi^{(m)} = \xi^{(m)}(v)$ be the frequency in $\hat{\sG}_m$ corresponding to $v$, namely
\begin{equation}
\xi^{(m)}_x = (G_{\bT_m} * v)(x) - (G_{\bT_m} * v)(0,0) \quad \textnormal{(reduced mod $\zed$),}
\end{equation}
and let $\xi = \xi(v) = G_{\zed^2} * v$. Then
\begin{equation}
\sav(\xi^{(m)}) \to f(\xi) \quad \text{as } m \to \infty.
\end{equation}
\end{proposition}

Part 2 of Proposition \ref{gap_achievers} implies that
\begin{equation}
\label{gamma_2}
\gamma = \min\left\{ f(G_{\zed^2} * v) : v \in \sC(B_0,R_0) \,\setminus\, \cI \right\},
\end{equation}
which reduces the computation of $\gamma$ to a finite check. In Appendix \ref{spectral_gap_appendix} we verify that $\gamma$ is obtained by $\xi = G_{\zed^2} * \delta_1 *\delta_2$ with numerical value
\begin{equation}
\label{gamma_value}
\gamma = 2.868114013(4).
\end{equation}

\begin{proof}[Proof of Theorem \ref{spectral_gap_theorem}]
In this proof we use the notation $\xi(v) = G_{\zed^2} * v$. Take the constants $B_0,R_0$ from Proposition \ref{gap_achievers} and find $\gamma' > \gamma$ such that if $v \in \sC(B_0,R_0)$ and $f(\xi(v)) > \gamma$, then $f(\xi(v)) \geq \gamma'$. Applying Proposition \ref{T_m_to_Z^2} with $B = B_0$ and $R_1 = R_0$, choose $m$ large enough that
\begin{equation}
g(m) := \sup_{v \in \sC(B_0,R_0)} |\sav(\xi^{(m)}(v)) - f(\xi(v))| < \frac{\gamma' - \gamma}{2}.
\end{equation}
Let $v_0 \in \sC(B_0,R_0)$ satisfy $f(\xi(v_0)) = \gamma$, and let $\xi^{(m)}_0 = \xi^{(m)}(v_0) \in \hat{\sG}_m$. Then
\begin{equation}
\sav(\xi^{(m)}_0) < \frac{\gamma + \gamma'}{2}.
\end{equation}
Now suppose that $\xi^{(m)} \in \hat{\sG}_m$ achieves the spectral gap. By translating, we may assume that $\xi^{(m)}$ has a prevector $v \in \sC(B_0,R_0)$. We claim that $f(\xi(v)) = \gamma$: if not, then $f(\xi(v)) \geq \gamma'$ and
\begin{equation}
\sav(\xi^{(m)}) > \frac{\gamma + \gamma'}{2} > \sav(\xi^{(m)}_0),
\end{equation}
a contradiction. Thus $f(\xi(v)) = \gamma$, and
\begin{equation}
|m^2 \gap_m - \gamma| = |\sav(\xi^{(m)}) - f(\xi(v))| \leq g(m),
\end{equation}
with $g(m) \to 0$ as $m \to \infty$. Along with the formula \eqref{gamma_value} for $\gamma$, which is proved in Appendix \ref{spectral_gap_appendix}, this concludes the proof.
\end{proof}

In the process of proving Propositions \ref{gap_achievers} and \ref{T_m_to_Z^2}, we show two lemmas, Lemmas \ref{C_0_1_lemma} and \ref{C_2_lemma}, regarding savings in the neighborhood of the support of $v$ for prevectors $v \in \zed_0^{\bT_m}$ that have bounded $L^1$ norm. Note that if $\xi \in \hat{\sG}_m$ is the frequency corresponding to $v$ and $\overline{\xi} = G_{\bT_m} * v$, then for any $S \subset \bT_m$, the proof of Lemma \ref{xi_bar} implies that
\begin{equation}
\left| \sum_{x \in S} e(\xi_x) \right| = \left| \sum_{x \in S} e(\overline{\xi}_x) \right|,
\end{equation}
so all savings computations can be done using $\overline{\xi}$. Indeed, if we extend the definitions \eqref{savings_def}, \eqref{savings_def_2} from elements of $\hat{\sG}_m$ to all $\bR$- or $\bR/\zed$-valued functions on $\bT_m$, then $\sav(\xi;S) = \sav(\overline{\xi};S)$ and $\sav(\xi) = \sav(\overline{\xi})$.

\begin{lemma}\label{C_0_1_lemma}
	For all $A,B,R_1 > 0$ there exists an $R_2(A, B, R_1)>2R_1$ such that if $m$ is sufficiently large, then for any $x \in \bT_m$ and any $v \in \zed^{\bT_m}$ satisfying the following conditions:
	\begin{enumerate}
		\item $\|v\|_1 \leq B$
		\item  $v|_{B_{R_1}(x)} \not \in C^2(\bT_m)$
		\item $d\left(x, \supp v|_{B_{R_1}(x)^c} \right)> 2 R_2$
	\end{enumerate}
	we have
	\begin{equation}
	\sav\left( G_{\bT_m} * v; B_{R_2}(x) \right) \geq A.
	\end{equation}
	Thus, if $v$ has mean zero, then the corresponding frequency $\xi \in \hat{\sG}_m$ satisfies $\sav\left(\xi; B_{R_2}(x)\right) \geq A$.
\end{lemma}

\begin{proof}
	Given $v \in \zed^{\bT_m}$, decompose $\overline{\xi} = G_{\bT_m} * v$ into an internal and external component, $\overline{\xi} = \xi^i + \xi^e$, setting 
		\begin{equation}
	\xi^i := G_{\bT_m} * v|_{B_{R_1}(x)}, \qquad \xi^e := G_{\bT_m}*v|_{B_{R_1}(x)^c}.
	\end{equation}
	
	Treat $R_2$ as a parameter growing to infinity, and let $R$ be a second  parameter growing with $R_2$ such that $\frac{R_2}{R^3} \to \infty$ as $R_2 \to \infty$.  In practice, these parameters are chosen large, but fixed, so that there is uniformity in all $m$ sufficiently large. Since $|D_1 G_{\bT_m}(y)|$ and $|D_2 G_{\bT_m}(y)|$ have size $\ll 1/\|y\|_1$ as $\|y\|_1 \to\infty$, we have $\xi^e_{x +y} = \xi^e_x + O(\frac{BR}{R_2})$ for all $\|y\|_1 \leq R$.  Hence, by Taylor expansion,
	\begin{equation}\label{external_eliminated}
	\left|\sum_{\|y\|_1 \leq R}e(\overline{\xi}_{x+y})\right| = O\left(\frac{BR^3}{R_2} \right)+ \left|\sum_{\|y\|_1 \leq R} e(\xi^i_{x+y})\right|.
	\end{equation}
	Since the error tends to 0 as $R_2 \to \infty$, it suffices to prove that  
	\begin{equation}
	\#\left\{y:\|y\|_1 \leq R \right\} - \left|\sum_{\|y\|_1 \leq R}e\left(\xi_{x+y}^i\right)\right| \to \infty \qquad \text{as } R \to \infty.
	\end{equation}
	
	First suppose that $v|_{B_{R_1}(x)} \not \in C^1\left(\bT_m\right)$. For all $y = (y_1,y_2) \in \bT_m$ with $|y_1|,|y_2| \leq m/2$,
	\begin{equation}
	\xi_{x+y}^i = \sum_{\|z\|_1 \leq R_1} G_{\bT_m}(y-z) v(x+z).
	\end{equation}
	Let $r = \sqrt{y_1^2 + y_2^2}$. Using the Lemma \ref{greens_function_estimate} bound on the first derivatives of $G_{\bT_m}$ to approximate $G_{\bT_m}(y - z)$ with $G_{\bT_m}(y)$ yields
	\begin{equation}
	\xi_{x+y}^i = a G_{\bT_m}(y) + O_{B,R_1}\left( r^{-1} \right),
	\end{equation}
	where $a = \sum_{\|z\|_1 \leq R_1} v(x+z) \neq 0$. The asymptotic for the first derivative of the Green's function in Lemma \ref{green_function_differentiated_asymptotic} now implies that $|\xi_{x + (j,0)}^i - \xi_x^i| \to \infty$ as $j \to \infty$, while $|\xi_{x + (j+1,0)}^i - \xi_{x + (j,0)}^i| \to 0$, so that $\{\xi^i_{x + (j,0)}\}_{j=0}^\infty$ is dense in $\bR/\zed$, and hence 
	\begin{equation}
	R - \left|\sum_{j=1}^R e(\xi_{x + (j,0)}^i - \xi_x^i)\right| \to \infty \qquad \text{as } R \to \infty,
	\end{equation}
	which suffices for the claim.
	
	Now suppose that $v|_{B_{R_1}(x)} \in C^1\left(\bT_m\right) \setminus C^2\left(\bT_m\right)$, so that it can be written as $\delta_1 * w_1 + \delta_2 * w_2$ where $w_1,w_2$ are $\zed$-valued, supported on $B_{R_1+1}(x)$, and not both in $C^1(\bT_m)$ by \eqref{C2_alt}. For all $y = (y_1,y_2)$,
	\begin{equation}
	\label{xi_convolution}
	\xi_{x+y}^i = \sum_{\|z\|_1 \leq R_1 + 1} D_1 G_{\bT_m}(y-z) w_1(x+z) + D_2 G_{\bT_m}(y-z) w_2(x+z).
	\end{equation}
	Use the Lemma \ref{greens_function_estimate} bound on the second derivatives of $G_{\bT_m}$ to approximate $D_k G_{\bT_m}(y-z)$ with $D_k G_{\bT_m}(y)$ for $k = 1,2$. The result is
	\begin{equation}
	\label{xi_convolution_2}
	\xi_{x+y}^i = a D_1 G_{\bT_m}(y) + b D_2 G_{\bT_m}(y) + O_{B,R_1}\left( r^{-2} \right)
	\end{equation}
	for constants $a,b = O_{B,R_1}(1)$, not both zero. Lemma \ref{green_function_differentiated_asymptotic} now shows that for $1 \leq r < m^{1/2}/(\log m)^{1/4}$,
	\begin{equation}
	\xi_{x+y}^i = \frac{-c(ay_1 + by_2)}{y_1^2 + y_2^2} + O_{B,R_1}\left( r^{-2} \right),
	\end{equation}
	where $c > 0$ is a fixed constant. Thus $|\xi_{x+y}^i| \ll 1/r$, and	there are $0 \leq \theta_1 < \theta_2 < 2\pi$ such that if $\theta_1 \leq \arg(y) \leq \theta_2$, then $|\xi_{x+y}^i| \asymp 1/r$.
	It follows that
	\begin{gather}
	\label{real_xi_i}
	\sum_{\|y\|_1 \leq R} (1 - c(\xi_{x+y}^i)) \asymp \log R, \\
	\label{imag_xi_i}
	\left|\sum_{\|y\|_1 \leq R}s(\xi_{x+y}^i)\right| \leq \sum_{\|y\|_1 \leq R} |s(\xi_{x+y}^i)| \ll R.
	\end{gather}
	To combine \eqref{real_xi_i} and \eqref{imag_xi_i} we use that for all $a,b \in \bR$ with $a > 0$,
	\begin{equation}
	\label{real_part_approx}
	\sqrt{a^2+b^2} - \sqrt{a^2} = \int_{a^2}^{a^2+b^2} \frac{dt}{2\sqrt{t}} \leq \frac{b^2}{2a}.
	\end{equation}
	Letting $a$ and $b$ be the real and imaginary parts of $\sum_{\|y\|_1 \leq R} e(\xi_{x+y}^i)$, we conclude that \begin{equation}	\#\left\{y:\|y\|_1 \leq R \right\} - \left|\sum_{\|y\|_1 \leq R}e\left(\xi_{x+y}^i\right)\right| \asymp \log R,\end{equation} as required.
\end{proof}

\begin{proof}[Proof of Proposition \ref{T_m_to_Z^2}]
Given $v \in \sC(B,R_1)$, set $\xi^* = G_{\bT_m} * v$; we suppress the dependence on $m$ for notational convenience. It will suffice to show that $\sav(\xi^*) \to f(\xi)$ as $m \to \infty$.

Write $v$ as a sum of $O_{B, R_1}(1)$ translates of $\pm \delta_1^{*2}$, $\pm \delta_1 * \delta_2$, $\pm \delta_2^{*2}$. Since the second derivatives of the Green's function decay like the inverse square of the radius, an argument parallel to the one given in equations \eqref{xi_convolution}-\eqref{xi_convolution_2} shows that $|\xi^*_y| = O_{B,R_1}(1/r^2)$, where $r = \sqrt{y_1^2 + y_2^2}$ and $|y_1|,|y_2| \leq m/2$. For all $R_1 < R < m/2$, Taylor expansion yields
\begin{align}
\label{Taylor_estimates}
\sum_{\|y\|_1 > R} (1 - c(\xi^*_y)) &= O_{B,R_1}\left(R^{-2}\right), \\
\notag \sum_{y \in \bT_m} (1 - c(\xi^*_y)) &= O_{B,R_1}(1), \\
\notag \left| \sum_{y \in \bT_m} s(\xi^*_y) \right| &= O_{B,R_1}(1).
\end{align}
In the last estimate, we use that $\xi^*$ is mean zero over $\bT_m$ so that the contribution of the linear term in the Taylor expansion of $s(\xi^*_y)$ vanishes. Therefore, using \eqref{real_part_approx} in the first equality,
\begin{align}
\label{savings_real}
\sav(\xi^*) &= O_{B,R_1}\left( m^{-2} \right) + \sum_{y \in \bT_m} (1 - c(\xi^*_y)) \\
\notag &= O_{B,R_1}\left( R^{-2} \right) + \sum_{\|y\|_1 \leq R} (1 - c(\xi^*_y)).
\end{align}
Sending $m \to \infty$ for fixed $R$, Lemma \ref{z_2_approx_lemma} shows that each $\xi^*_y \to \xi_y$. Thus
\begin{equation}
\lim_{m \to \infty} \sav(\xi^*) = O_{B,R_1}\left( R^{-2} \right) + \sum_{\|y\|_1 \leq R} (1 - c(\xi_y))
\end{equation}
for each $R > R_1$. Sending $R \to \infty$ completes the proof.
\end{proof}

\begin{lemma}\label{C_2_lemma}
	For all $B, R_1 > 0$ and $\alpha < 1$, there exists $R_2(\alpha, B, R_1)>2R_1$ such that if $m$ is sufficiently large, then for any $x \in \bT_m$ and any $v \in \zed^{\bT_m}$ satisfying the following conditions:
	\begin{enumerate}
		\item $\|v\|_1 \leq B$
		\item $v|_{B_{R_1}(x)}  \in C^2(\bT_m)$
		\item $d\left(x, \supp v|_{B_{R_1}(x)^c} \right)> 2 R_2$
	\end{enumerate}
	we have
	\begin{equation}
	\label{savings_alpha}
	\sav\left(G_{\bT_m} * v; B_{R_2}(x)\right) \geq \alpha \sav(\xi^*); \qquad \xi^*=G_{\bT_m} * v|_{B_{R_1}(x)}.
	\end{equation}
	Thus, if $v$ has mean zero, then the corresponding frequency $\xi \in \hat{\sG}_m$ satisfies $\sav\left(\xi; B_{R_2}(x)\right) \geq \alpha \sav(\xi^*)$.
\end{lemma}
\begin{proof}
	First we show that there is $\delta = \delta(B,R_1) > 0$ such that for sufficiently large $m$, either $\sav(\xi^*) = 0$ or $\sav(\xi^*) \geq \delta$. Translating $v$ by $-x$ shows that $\sav(\xi^*) = \sav(G_{\bT_m} * v')$ for some $v' \in \sC(B,R_1)$. Let
	\begin{equation}
	\gamma' = \min\left\{ f(G_{\zed^2} * v') : v' \in \sC(B,R_1) \,\setminus\, \cI \right\},
	\end{equation}
	so $\gamma' > 0$. By Proposition \ref{T_m_to_Z^2}, if $m$ is large enough then
	\begin{equation}
	|\sav(G_{\bT_m} * v') - f(G_{\zed^2} * v')| < \gamma' / 2 \quad \text{for all } v' \in \sC(B,R_1).
	\end{equation}
	Thus, if $v' \in \sC(B,R_1) \,\setminus\, \cI$, then $\sav(\xi^*) > \gamma'/2$.
	
	If $v' \in \sC(B,R_1) \cap \cI$, we will show that $\sav(\xi^*) = 0$, allowing us to take $\delta = \gamma'/2$. Write $v' = \Delta w$ where $w \in C^0(\zed^2)$. Observe that $\supp(w)$ is a finite set, and any $(i,j) \in \zed^2 \,\setminus\, \supp(w)$ that is adjacent to exactly one point in $\supp(w)$ must have $(\Delta w)(i,j) \neq 0$. Since $\supp(\Delta w) \subset B_{R_1}(0)$, it follows that $\supp(w) \subset B_{R_1-1}(0)$. Hence we can consider $v'$ and $w$ as $\zed$-valued functions on $\bT_m$ for $m > 2R_1$, and the equation $v' = \Delta w$ still holds in this context. Therefore, $G_{\bT_m} * v' = w - c$ where $c$ is the mean value of $w$ on $\bT_m$, and $\sav(G_{\bT_m} * v') = 0$.
	
	With $\delta$ in hand, we turn to the proof of \eqref{savings_alpha}. Set $\epsilon = \epsilon(\alpha,B,R_1) = (1-\alpha)\delta > 0$. We will show that if $m$ is sufficiently large,
	\begin{equation}
	\sav(G_{\bT_m} * v; B_{R_2}(x)) > \sav(\xi^*) - \epsilon.
	\end{equation}
	This implies \eqref{savings_alpha}, because if $\sav(\xi^*) = 0$ then \eqref{savings_alpha} is trivial, while if $\sav(\xi^*) \geq \delta$ then $\sav(\xi^*) - \epsilon \geq \alpha \sav(\xi^*)$. By arguing as in Lemma \ref{C_0_1_lemma} up to equation \eqref{external_eliminated}, it suffices to prove that if $R$ is fixed but sufficiently large then
	\begin{equation}
	\label{xi_star_desired}
	\sav(\xi^*; B_R(x)) > \sav(\xi^*)-\epsilon/2
	\end{equation}
	for all $m$ sufficiently large. Writing $v|_{B_{R_1}(x)}$ as a sum of $O_{B, R_1}(1)$ translates of $\pm \delta_1^{*2}$, $\pm \delta_1 * \delta_2$, $\pm \delta_2^{*2}$, it follows as in the proof of Proposition \ref{T_m_to_Z^2} that for $y = (y_1,y_2) \in \bT_m$ with $|y_1|,|y_2| \leq m/2$ and $r = \sqrt{y_1^2 + y_2^2}$, $|\xi_{x+y}^*| = O_{B,R_1}(1/r^2)$. Taylor expansion gives
	\begin{align}
	\label{Taylor_2}
	\sum_{\|y\|_1 \leq R} (1 - c(\xi^*_{x+y})) &= O_{B,R_1}(1), \\
	\notag \sum_{\|y\|_1 > R} (1 - c(\xi^*_{x+y})) &= O_{B,R_1}\left( R^{-2} \right), \\
	\notag \left| \sum_{\|y\|_1 \leq R} s(\xi^*_{x+y}) \right| \leq \sum_{\|y\|_1 \leq R} \left| s(\xi^*_{x+y}) \right| &= O_{B,R_1}(\log R).
	\end{align}
	Thus,
	\begin{align}
	\label{xi_star_approx}
	\sav(\xi^*) &= m^2 - \left| \sum_{z \in \bT_m} e(\xi_z^*) \right| \leq \sum_{z \in \bT_m} (1 - c(\xi_z^*)) \\
	\notag &= O_{B,R_1}\left( R^{-2} \right) + \sum_{\|y\|_1 \leq R} (1 - c(\xi_{x+y}^*)) \\
	\notag &= O_{B,R_1}\left( \frac{\log^2 R}{R^2} \right) + \# B_R(x) - \left| \sum_{\|y\|_1 \leq R} e(\xi_{x+y}^*) \right|,
	\end{align}
	using \eqref{real_part_approx} in the last equality. Since this is $\sav(\xi^*; B_R(x))$ plus a quantity tending to zero with $R$, \eqref{xi_star_desired} is verified.
\end{proof}

\begin{proof}[Proof of Proposition \ref{gap_achievers}]
First we find a constant $B_0$ such that:
\begin{enumerate}[label=(\Roman*)]
\item For sufficiently large $m$, if $\xi^{(m)} \in \hat{\sG}_m$ achieves the spectral gap, then its distinguished prevector $v^{(m)} = v(\xi^{(m)})$ must satisfy $\|v^{(m)}\|_1 \leq B_0$. \label{I}
\item If $v \in C^2(\zed^2)$ satisfies $f(G_{\zed^2} * v) \leq \frac{3}{2} \gamma + 1$, then $v$ differs by an element of $\cI$ from some $\tilde{v} \in C^2(\zed^2)$ with $\|\tilde{v}\|_1 \leq B_0$. \label{II}
\end{enumerate}

To this end, fix any $v' \in C^2(\zed^2) \,\setminus\, \cI$ and let $\gamma' = f(G_{\zed^2} * v') \geq \gamma$. Choose $B',R'$ large enough that $v' \in \sC(B',R')$, and let $m > 2R'$ so that $\sC(B',R')$ embeds into $C^2(\bT_m)$. Applying Proposition \ref{T_m_to_Z^2} shows that if $\xi'^{(m)}$ is the frequency in $\hat{\sG}_m$ corresponding to $v'$, then
\begin{equation}
\sav(\xi'^{(m)}) \to f(G_{\zed^2} * v') \quad \text{as } m \to \infty
\end{equation}
and therefore $\sav(\xi'^{(m)}) < \gamma' + 1$ for sufficiently large $m$.

Suppose that $\xi^{(m)} \in \hat{\sG}_m$ achieves the spectral gap, and let $v^{(m)}$ be the distinguished prevector of $\xi^{(m)}$. By Lemma \ref{l_2_fourier_lemma},
\begin{equation}
\|v^{(m)}\|_1 \ll \sav(\xi^{(m)}) \leq \sav(\xi'^{(m)}) < \gamma' + 1,
\end{equation}
that is, $\|v^{(m)}\|_1$ is bounded by a universal constant. This verifies \ref{I}.

By choosing $\gamma'$ arbitrarily close to $\gamma$, we could have obtained that for any $\epsilon > 0$, if $m$ is sufficiently large then any $\xi^{(m)} \in \hat{\sG}_m$ achieving the spectral gap must satisfy $\sav(\xi^{(m)}) < \gamma + \epsilon$. This will be used later.

For \ref{II}, given $v \in C^2(\zed^2)$, let $\xi = G_{\zed^2} * v$. Lemma \ref{z_2_approx_lemma} or Lemma \ref{greens_function_derivs} shows that $\xi \in \ell^2(\zed^2)$, so there are only finitely many $x \in \zed^2$ such that $|\xi_x| \geq \frac{1}{2}$. Reduce $\xi$ to $\tilde{\xi}: \zed^2 \to \left[ -\frac{1}{2}, \frac{1}{2} \right)$ by subtracting $w \in C^0(\zed^2)$, and let $\tilde{v} = \Delta \tilde{\xi} = v - \Delta w$, which differs from $v$ by $\Delta w \in \cI$ and is therefore in $C^2(\zed^2)$. Because $\tilde{v}$ is integer-valued and $\Delta$ is bounded from $\ell^2(\zed^2) \to \ell^2(\zed^2)$,
\begin{equation}
\label{v_tilde}
\|\tilde{v}\|_1 \leq \|\tilde{v}\|_2^2 = \|\Delta \tilde{\xi}\|_2^2 \ll \|\tilde{\xi}\|_2^2 = \sum_{x \in \zed^2} |\tilde{\xi}_x|^2 \ll \sum_{x \in \zed^2} (1 - c(\tilde{\xi}_x)),
\end{equation}
where the last inequality uses $1 - c(t) \gg t^2$ for $|t| \leq \frac{1}{2}$. The right side is $f(\tilde{\xi}) = f(G_{\zed^2} * v)$, so an upper bound on $f(G_{\zed^2} * v)$ translates to an upper bound on $\|\tilde{v}\|_1$, confirming \ref{II}.

Fix $B_0$ to satisfy \ref{I} and \ref{II}. For any $v \in \zed^{\bT_m}$ with $\|v\|_1 \leq B_0$, we perform a clustering on $\supp(v)$, as follows. 
\begin{enumerate}
	\item  Initially all of $\supp v$ is uncovered and initialize a list $\sX$ of centers of balls to be empty.
	\item Iterate until $\supp v$ is covered:
	\begin{enumerate}
		\item Choose $x \in \supp v$ which is uncovered and append $x$ to $\sX$.
		\item Beginning from an initial guess $R_1(x) = 1$:
		\begin{enumerate}
			\item If $v|_{B_{R_1(x)}(x)} \not \in C^2(\bT_m)$, then choose $R_2(x)$ according to Lemma \ref{C_0_1_lemma} with $A = \frac{3}{2}\gamma + 1$, $B = B_0$, and $R_1=R_1(x)$.  If $v|_{B_{R_1(x)}(x)} \in C^2(\bT_m)$, then choose $R_2(x)$ according to Lemma \ref{C_2_lemma} with $\alpha = \frac{7}{8}$, $B = B_0$, and $R_1=R_1(x)$.
			\item If the condition of those lemmas holds,
			\begin{equation*}
				d\left(x, \supp v|_{B_{R_1(x)}(x)^c} \right)> 2 R_2(x),
			\end{equation*}
			then declare those $y \in \supp (v) \cap B_{R_1}(x)$ covered and continue to (c).  Otherwise, replace $R_1(x):= 2R_2(x)$ and repeat step (i).  
		\end{enumerate} 
		\item If all of $\supp v$ is covered, finish. If not, return to (a).
	\end{enumerate}
	\item Since $\|v\|_1 \leq B_0$, the process stops after boundedly many steps.  
\end{enumerate}
	At the end of this process,
	\begin{gather}
	\supp v \subset \bigcup_{x \in \sX} B_{R_1(x)}(x), \\
	\label{distance_condition} d\left(x, \supp v|_{B_{R_1(x)}(x)^c} \right)> 2 R_2(x) \quad \text{for each } x \in \sX.
	\end{gather}

	We claim that a subset $\sX' \subset \sX$ can be chosen such that 
	\begin{equation}\label{disjointness_condition}
	\supp v \subset \bigsqcup_{x \in \sX'} B_{R_1(x)}(x).
	\end{equation}
	Note a particular consequence of \eqref{distance_condition} and \eqref{disjointness_condition} is that the balls \begin{equation}\{B_{R_2(x)}(x)\}_{x \in \sX'}\end{equation} are pairwise disjoint.
		
	To verify \eqref{disjointness_condition}, let $x$ and $x'$ be centers of balls of the process with $x$ appearing prior to $x'$ in the list. We will show that either $B_{R_1(x)}(x)$ and $B_{R_1(x')}(x')$ are disjoint, or $B_{R_1(x)}(x) \subset B_{R_1(x')}(x')$. First, since $x' \notin B_{R_1(x)}(x)$, we have $d(x,x') > 2R_2(x) \geq 2R_1(x)$. Suppose $y$ is in the intersection of $B_{R_1(x)}(x)$ and $B_{R_1(x')}(x')$, so that
	\begin{equation}
	\label{d_x_x'}
	d(x,x') \leq d(x,y) + d(y,x') \leq R_1(x) + R_1(x').
	\end{equation}
	Combining this with the lower bound on $d(x,x')$ gives $R_1(x) < R_1(x')$. Therefore, \eqref{d_x_x'} implies that $d(x,x') < 2R_1(x') \leq 2R_2(x')$, whence $d(x,x') \leq R_1(x')$. For any $z \in B_{R_1(x)}(x)$,
	\begin{equation}
	d(z,x') \leq d(z,x) + d(x,x') \leq R_1(x) + R_1(x') < 2R_2(x')
	\end{equation}
	so that, in fact, $z \in B_{R_1(x')}(x')$ and $ B_{R_1(x)}(x) \subset B_{R_1(x')}(x')$.  Hence, starting from $\sX$ we obtain the desired list $\sX'$ by discarding any $x$ which satisfies $x \in B_{R_1(x')}(x')$ for some $x'$ later in the list.  
	
	Let $\overline{R}_1(v) = \max\{R_1(x) : x \in \sX'\}$. From the description of the clustering algorithm, there is a uniform in $m$ upper bound on $\overline{R}_1(v)$ that depends only on $B_0$:
	\begin{equation}
	\label{R_1_bar}
	\overline{R}_1(v) \leq R_0 = R_0(B_0).
	\end{equation}
	Fix
	\begin{equation}
	\label{gamma_0}
	\gamma_0 = \min\left\{ f(G_{\zed^2} * v) : v \in \sC(B_0,R_0) \,\setminus\, \cI \right\}.
	\end{equation}
	We will show that $\gamma = \gamma_0$, but a priori we only know that $\gamma_0 \geq \gamma$ and $\gamma_0 > 0$. Proposition \ref{T_m_to_Z^2} implies that if $m$ is large enough,
	\begin{equation}
	\label{savings_8}
	|\sav(G_{\bT_m} * v) - f(G_{\zed^2} * v)| < \gamma_0/8 \quad \text{for all } v \in \sC(B_0,R_0).
	\end{equation}
	
	We can now prove Part 1 of Proposition \ref{gap_achievers}. Let $\xi \in \hat{\sG}_m$ achieve the spectral gap, and take $m$ large enough that $\sav(\xi) < \min(\frac{3}{2}\gamma + 1, \frac{3}{2}\gamma_0)$, noting that the upper bound is strictly greater than $\gamma$. Let $v \in \zed_0^{\bT_m}$ be the distinguished prevector of $\xi$, with $\|v\|_1 \leq B_0$, and run the clustering algorithm on $v$. Set
	\begin{equation}
	\sX'' = \left\{ x \in \sX' : \left. v \right|_{B_{R_1(x)}(x)} \notin \cI \right\}
	\end{equation}
	and define $v'$ to equal $v$ on each $B_{R_1(x)}(x)$ for $x \in \sX''$, while $v' \equiv 0$ elsewhere. To get from $v$ to $v'$, we subtracted finitely many elements of $\cI$. It follows that $v' = v - \Delta w$ for some $w \in \zed^{\bT_m}$; the proof of Lemma \ref{C_2_lemma} explains why this holds on $\bT_m$ as well as on $\zed^2$. We conclude that $v'$ is also a prevector of $\xi$.
	
	Given $x \in \sX''$, for notational convenience set $u^{(x)} = \left. v \right|_{B_{R_1(x)}(x)}$. If some $u^{(x)} \notin C^2(\bT_m)$, then by construction, $\sav(\xi) \geq \frac{3}{2}\gamma + 1$, a contradiction. Therefore each $u^{(x)} \in C^2(\bT_m)$. Since $R_1(x) \leq R_0$ and $u^{(x)} \notin \cI$, we have $f(G_{\zed^2} * u^{(x)}) \geq \gamma_0$. It follows from \eqref{savings_8} that $\sav(G_{\bT_m} * u^{(x)}) > \frac{7}{8} \gamma_0$. Then, by step (i) of the clustering,
	\begin{equation}
	\sav(\xi; B_{R_2(x)}(x)) > \left( \frac{7}{8} \right)^2 \gamma_0 > \frac{3}{4} \gamma_0.
	\end{equation}
	If $|\sX''| \geq 2$, then $\sav(\xi) > \frac{3}{2} \gamma_0$, another contradiction. We conclude that $|\sX''| = 1$, and moreover, for the unique $x \in \sX''$, $v' = u^{(x)} \in C^2(\bT_m)$. Translating $v'$ by $-x$ yields an element of $\sC(B_0,R_0)$. This proves Part 1 of Proposition \ref{gap_achievers}.
	
	Part 2 is proved along similar lines. Let $v \in C^2(\zed^2)$ satisfy
	\begin{equation}
	\label{gamma_bound} \textstyle
	f(G_{\zed^2} * v) < \min(\frac{3}{2} \gamma + 1, \frac{3}{2} \gamma_0).
	\end{equation}
	By property \ref{II}, we may assume that $\|v\|_1 \leq B_0$ (by subtracting an element of $\cI$ if necessary). Since $\supp v$ is finite, it embeds into $\bT_m$ for large enough $m$, and then $v$ can be seen as an element of $C^2(\bT_m)$. Let $\xi^{(m)} \in \hat{\sG}_m$ be the frequency corresponding to $v$, so that $\sav(\xi^{(m)}) \to f(G_{\zed^2} * v)$ as $m \to \infty$, by Proposition \ref{T_m_to_Z^2}.
	
	Run the clustering algorithm on $v$, noting that independent of the value of $m$, the algorithm will follow exactly the same steps and produce identical clusters. Define $\sX''$, $v'$, and the notation $u^{(x)}$ as in the proof of Part 1. If $u^{(x)} \notin C^2(\bT_m)$ for some $x \in \sX''$, then $\sav(\xi^{(m)}) \geq \frac{3}{2} \gamma + 1$. Since the property ``$u^{(x)} \notin C^2(\bT_m)$'' is independent of $m$, this is a uniform lower bound on all $\sav(\xi^{(m)})$ and so $f(G_{\zed^2} * v) \geq \frac{3}{2} \gamma + 1$. Likewise, if each $u^{(x)} \in C^2(\bT_m)$ but $|\sX''| \geq 2$, then $\sav(\xi^{(m)}) > \frac{3}{2} \gamma_0$ for all $m$ and so $f(G_{\zed^2} * v) \geq \frac{3}{2} \gamma_0$. Both possibilities contradict \eqref{gamma_bound}.
	
	We conclude that any $v \in C^2(\zed^2)$ satisfying \eqref{gamma_bound} differs by an element of $\cI$ from some $v' \in C^2(\zed^2)$ that has a translate in $\sC(B_0,R_0)$. It follows that $\gamma$ is equal to the right side of \eqref{gamma_0}, that is, $\gamma = \gamma_0$. In particular, $\gamma > 0$. Finally, the right side of \eqref{gamma_bound} simplifies to $\frac{3}{2} \gamma$, proving Part 2 of Proposition \ref{gap_achievers}.
\end{proof}
	
The following lemma provides the additive savings needed to prove the lower bound in Theorem \ref{mixing_time_theorem}.

\begin{lemma}\label{C_2_sum_lemma}
Let $k \geq 1$ be fixed, and let $v_1, \ldots, v_k \in C^2(\bT_m)$ be bounded functions of bounded support 
which are 
$R$-separated, in the sense that their supports have pairwise $\ell^1$ distance at least $R$.  Set $v = \sum_{i=1}^k v_i$.  Then as $R \to \infty$,
\begin{equation}
 1 - \left|\hat{\mu}(\xi(v))\right| = O\left(\frac{\log(1+R)}{R^2m^2} \right) + 
\sum_{i=1}^k \Big( 1 - \left|\hat{\mu}(\xi(v_i)) \right| \Big).
\end{equation}
The implicit constant depends upon $k$ and the bounds for the functions and their supports.
\end{lemma}

\begin{proof}
Set $\overline{\xi} = G_{\bT_m} * v$ and $\overline{\xi}_i = G_{\bT_m} * v_i$, so that $|\hat{\mu}(\xi(v))| = |\hat{\mu}(\overline{\xi})|$ and $|\hat{\mu}(\xi(v_i))| = |\hat{\mu}(\overline{\xi}_i)|$. Fix a point $x_i$ in the support of each $v_i$, so that the balls $B_{R'}(x_i)$ are disjoint where $R' = \lfloor (R-1)/2 \rfloor$. As in the proof of Proposition \ref{T_m_to_Z^2}, if $y = (y_1,y_2)$ with $|y_1|,|y_2| \leq m/2$ and $r = \sqrt{y_1^2 + y_2^2}$,
\begin{equation}
\label{inverse_square}
|\overline{\xi}_i(x_i + y)| = O(1/r^2).
\end{equation}
We obtain the analogue to \eqref{savings_real},
\begin{equation}
\label{real_reduction}
1 - |\hat{\mu}(\overline{\xi}_i)| = O\left( \frac{1}{R^2 m^2} \right) + \frac{1}{m^2} \sum_{\|y\|_1 \leq R'} \Big(1 - c\left(\overline{\xi}_i(x_i + y)\right)\Big).
\end{equation}
If $\|y\|_1 \leq R'$, then $\overline{\xi}(x_i + y) = \overline{\xi}_i(x_i + y) + O(R^{-2})$, so that
\begin{equation}
\label{cosine_approx}
c\left(\overline{\xi}(x_i + y)\right) = c\left(\overline{\xi}_i(x_i + y)\right) + O\left(\frac{\left|s(\overline{\xi}_i(x_i + y))\right|}{R^2} \right) + 
O\left( \frac{1}{R^4}\right).
\end{equation}
As in \eqref{Taylor_2},
\begin{equation}
\label{sine_bound}
\sum_{\|y\|_1 \leq R'} \left|s\left(\overline{\xi}_i(x_i + y)\right)\right| = O(\log(1 + R)).
\end{equation}
Combining \eqref{real_reduction}, \eqref{cosine_approx}, and \eqref{sine_bound} yields
\begin{equation}
\label{xi_i_approx}
1 - |\hat{\mu}(\overline{\xi}_i)| = O\left( \frac{\log(1+R)}{R^2 m^2} \right) + \frac{1}{m^2} \sum_{\|y\|_1 \leq R'} \Big(1 - c\left(\overline{\xi}(x_i + y)\right)\Big).
\end{equation}
Take the sum of \eqref{xi_i_approx} over $i = 1,2,\ldots,k$. For $z \notin \bigcup_{i=1}^k B_{R'}(x_i)$, let $r_i$ be the $\ell^2$ distance from $z$ to $x_i$, so that $|\overline{\xi}(z)| = O(1/r_1^2 + \cdots + 1/r_k^2)$. Use the inequality
\begin{equation}
\left( \frac{1}{r_1^2} + \cdots + \frac{1}{r_k^2} \right)^2 \leq k\left( \frac{1}{r_1^4} + \cdots + \frac{1}{r_k^4} \right)
\end{equation}
to conclude that
\begin{equation}
\sum_{z \,\notin\, \bigcup_{i=1}^k B_{R'}(x_i)} \Big(1 - c\left(\overline{\xi}(z)\right)\Big) = O\left( \frac{1}{R^2} \right).
\end{equation}
In combination with \eqref{xi_i_approx}, this yields
\begin{equation}
\sum_{i=1}^k \Big( 1 - \left|\hat{\mu}(\overline{\xi}_i) \right| \Big) = O\left( \frac{\log(1+R)}{R^2 m^2} \right) + \frac{1}{m^2} \sum_{z \in \bT_m} \Big(1 - c\left(\overline{\xi}(z)\right)\Big),
\end{equation}
or equivalently,
\begin{equation}
\label{xi_real}
1 - \RE\left( \hat{\mu}(\overline{\xi}) \right) = O\left( \frac{\log(1+R)}{R^2 m^2} \right) + \sum_{i=1}^k \Big( 1 - \left|\hat{\mu}(\overline{\xi}_i) \right| \Big).
\end{equation}

We will finish the proof by applying \eqref{real_part_approx}. Proposition \ref{T_m_to_Z^2} implies that each $1 - \left| \hat{\mu}(\overline{\xi}_i) \right| = O(1/m^2)$, so $\RE\left( \hat{\mu}(\overline{\xi}) \right) \gg 1$. Meanwhile,
\begin{equation}
\IM \left( \hat{\mu}(\overline{\xi}) \right) = \frac{1}{m^2} \sum_{z \in \bT_m} s\left(\overline{\xi}(z)\right), \qquad s(t) = 2\pi t + O\left( |t|^3 \right).
\end{equation}
In the Taylor expansion, the linear term vanishes since $\overline{\xi}$ has mean zero on $\bT_m$. Also,
\begin{equation}
\left| \overline{\xi}(z) \right|^3 \leq \left( \sum_{i=1}^k \left| \overline{\xi}_i(z) \right| \right)^3 \leq k^2 \sum_{i=1}^k \left| \overline{\xi}_i(z) \right|^3,
\end{equation}
and $\sum_{z \in \bT_m} \left| \overline{\xi}_i(z) \right|^3 = O(1)$ by \eqref{inverse_square}. Hence $\IM \left( \hat{\mu}(\overline{\xi}) \right) = O(1/m^2)$, and we obtain the desired result from \eqref{xi_real} using \eqref{real_part_approx}.
\end{proof}

\subsection{Estimation of moderate size phases}
\label{small_phase_section}
In this section we give estimates for the savings of frequencies $\xi$ whose distinguished prevectors $v = v(\xi)$  have $\|v\|_1$ growing with $m$.  In particular, we prove an approximate additive savings estimate for separated parts of $v$, which is what is needed to prove the upper bound of Theorem \ref{mixing_time_theorem}. 

Let $R  > 1$ be a large fixed parameter. Given any $v \in \zed^{\bT_m}$, for each $x 
\in \supp v$ let \begin{equation}\nbd(x) := B_R(x) = \{y \in \bT_m: \left\|y-x\right\|_1 \leq R\}.\end{equation} Perform a simple 
agglomeration scheme, in which any two points $x, y \in \supp v$ whose 
neighborhoods overlap are joined in a common $R$-cluster.  In other words, $x$ and $y$ belong to a common cluster if and only if there is a sequence of points $\{z_i\}_{i=0}^n \subset \supp v$ such that $x = z_0$, $y = z_n$ and, for $0 \leq i < n$, $\|z_i - z_{i+1}\|_1 \leq 2R$.  Write $\sC$ for the 
collection of clusters formed in this way.  Given $C \in \sC$, write
\begin{equation}
\nbd(C) := \bigcup_{x \in C} \nbd(x)
\end{equation}
for the neighborhood of $C$, so that $\supp v \subset \bigsqcup_{C \in \sC} \nbd(C)$.

Let $\sP \subset \sC$ be the collection of all clusters $C$ such that $\left. v \right|_C = \Delta w$ for some $w \in \zed^{\bT_m}$, and let $S$ be the union of all clusters $C \in \sC \,\setminus\, \sP$. The `$R$-reduction' of $v$ is defined to be $\tilde{v} = \left. v \right|_S$, which differs from $v$ by a sum of terms of the form $\Delta w$, and whose $L^1$ and $L^\infty$ norms are bounded by $\|v\|_1, \|v\|_{L^\infty}$ respectively. We say that $v$ is `$R$-reduced' if $\tilde{v} = v$. For any frequency $\xi \in \hat{\sG}_m$, the `$R$-reduced prevector' of $\xi$ is the $R$-reduction of the distinguished prevector $v(\xi)$, which is indeed a prevector of $\xi$.

The following is the main result of this section. It is similar to Lemmas \ref{C_0_1_lemma} and \ref{C_2_lemma}, but does not require the prevector $v$ to have bounded $L^1$ norm.

\begin{lemma}\label{savings_lemma}
	Let $B \geq 1$ be a fixed parameter.  There is a function $\eta(B, R)$ tending to 0 as $R \to \infty$ such that for all $m$ sufficiently large, if $v \in \zed^{\bT_m}$ satisfies the following conditions:
	\begin{enumerate}
	\item $v$ is $R$-reduced
	\item $\|v\|_{L^\infty} \leq 3$
	\item $v$ has an $R$-cluster $C$ for which $\big\| \left. v \right|_C \big\|_1 \leq B$
	\end{enumerate}
	then
	\begin{equation}
	\label{eta_bound}
	\sav(G_{\bT_m} * v; \nbd(C)) \geq m^2 \gap_m - \eta(B, R).
	\end{equation}
	Thus, if $v$ has mean zero, then the corresponding frequency $\xi \in \hat{\sG}_m$ satisfies $\sav(\xi; \nbd(C)) \geq m^2 \gap_m - \eta(B, R)$.
	
	The sufficiently large value of $m$ above which \eqref{eta_bound} holds is allowed to depend on both $B$ and $R$.
\end{lemma}

The upper bound on $\|v\|_{L^\infty}$ could be replaced by any fixed constant; we chose $3$ because the distinguished prevector of every $\xi \in \hat{\sG}_m$ satisfies $\|v(\xi)\|_{L^\infty} \leq 3$, so the $R$-reduced prevector has the same bound.

\begin{proof}
	Suppose that $v \in \zed^{\bT_m}$ satisfies the conditions of the lemma. We decompose the phase function $\overline{\xi} = G_{\bT_m} * v$ into an internal and external component, $\overline{\xi} = \xi^i + \xi^e$, where
	\begin{equation}
	\xi^i := G_{\bT_m} * \left. v \right|_C, \qquad \xi^e := G_{\bT_m} * \left. v \right|_{C^c}.
	\end{equation}

	Our first observation is that the third derivatives of $\xi^e$ are uniformly bounded over all $x \in \nbd(C)$:
	\begin{equation}
	\label{3rd_deriv_bound}
	\left|D_1^a D_2^b \xi_x^{e}\right| \ll \frac{1}{R}, \quad \text{for $x \in \nbd(C)$ and $a,b \geq 0$, $a+b = 3$}.
	\end{equation}
	To see this, note that if $x \in \nbd(C)$, then every $y \in \supp(v) \,\setminus\, C$ satisfies $\|x-y\|_1 > R$. Therefore,
	\begin{align}
	\left| D_1^a D_2^b \xi_x^e \right| &= \left| \sum_{y \in C^c} v(y) D_1^a D_2^b G_{\bT_m}(x-y) \right| \\
	\notag &\leq \|v\|_{L^\infty} \sum_{y \in B_R(x)^c} |D_1^a D_2^b G_{\bT_m}(x-y)|.
	\end{align}
	The bound \eqref{3rd_deriv_bound} then follows from the asymptotic of Lemma \ref{greens_function_estimate}.
	
	The rest of the proof is divided into three cases. Heuristically, $\xi^e$ could be roughly constant over $\nbd(C)$, vary linearly over $\nbd(C)$, or vary quadratically over $\nbd(C)$. The bound on the third derivatives of $\xi^e$ ensures that these are the only possibilities. If $\xi^e$ is roughly constant, then we can prove \eqref{eta_bound} using the arguments developed in Section \ref{Determination_of_gap}.
	
	If $\xi^e$ varies linearly over $\nbd(C)$, then we can find a region of $\nbd(C)$ far enough away from $C$ that the internal phase is nearly constant, so $\overline{\xi} = \xi^i + \xi^e$ varies linearly. We then cite the geometric series bound of Lemma \ref{geometric} to show that for any $A > 0$, if $R$ is large enough then $\sav(\overline{\xi}; \nbd(C)) \geq A$. This is much stronger than the desired bound \eqref{eta_bound}: as long as $A > \gamma = \lim_{m \to \infty} m^2 \gap_m$, we do not even need to subtract $\eta(B,R)$.
	
	Finally, if $\xi^e$ varies quadratically over $\nbd(C)$, we use van der Corput's inequality to reduce to the linear case.
	
	The proof will use three auxiliary parameters $R_1,R_2,R_3$ which tend to infinity with $R$ and satisfy $R_1 < R_2 < R_3 < R$. We require that
	\begin{equation}
	\label{R_123}
	R_1 \to \infty, \quad \frac{R_2}{R_1^4} \to \infty, \quad \frac{R_3}{R_1 R_2^2} \to \infty, \quad \frac{R}{R_1^2 R_3^2} \gg 1, \quad \text{as } R \to \infty.
	\end{equation}
	These properties are all satisfied if, for example, $R_1 = R^{1/26}$, $R_2 = R^{5/26}$, $R_3 = R^{12/26}$.
	
	For the first case, suppose that for all $x \in C$ and $\|y - x\|_1 \leq R_1$, we have $\left\|\xi_y^e - \xi_x^e\right\|_{\bR/\zed} < 1/R_1^3$. Perform the clustering algorithm from the proof of Proposition \ref{gap_achievers} on $\left. v \right|_C$, using the same parameters (e.g.~$\alpha = 7/8$). This partitions $C$ into sub-clusters indexed by a set $\sX'$: for each $x \in \sX'$ there are radii $2\tilde{R}_1(x) < \tilde{R}_2(x)$ such that
	\begin{equation}
	C \subset \bigsqcup_{x \in \sX'} B_{\tilde{R}_1(x)}(x),
	\end{equation}
	while the balls $\{B_{\tilde{R}_2(x)}(x)\}_{x \in \sX'}$ are disjoint, and each sub-cluster meets the conditions of either Lemma \ref{C_0_1_lemma} or Lemma \ref{C_2_lemma}, as appropriate. As in \eqref{R_1_bar}, the radii $\tilde{R}_2(x)$ are uniformly bounded by some $R_0$ depending only on $B$. By taking $R$ large enough with respect to $B$, we may assume that $R_0$ is arbitrarily small relative to $R_1$.
	
	Let $x \in \sX'$ and $R' \leq R_1$. We use the assumption that $\left\|\xi_y^e - \xi_x^e\right\|_{\bR/\zed} < 1/R_1^3$ for all $y \in B_{R_1}(x)$ to compute, by Taylor expansion,
	\begin{equation}
	\left| \sum_{y \in B_{R'}(x)} e(\xi^i_y + \xi^e_y) \right| = \left| \sum_{y \in B_{R'}(x)} e(\xi^i_y) \right| + O\left( \frac{1}{R_1} \right).
	\end{equation}
	In other words, for all $R' \leq R_1$,
	\begin{equation}
	\label{xi_i_reduction}
	\sav(\overline{\xi}; B_{R'}(x)) = \sav(\xi^i; B_{R'}(x)) + O\left( R_1^{-1} \right).
	\end{equation}
	
	Define $\sX'' \subset \sX'$ as in the proof of Proposition \ref{gap_achievers}. Since $v$ is $R$-reduced, $\sX''$ is nonempty. For $x \in \sX''$, let $u^{(x)}$ be the restriction of $v$ to $B_{\tilde{R}_1(x)}(x)$. By step (i) of the clustering,
	\begin{equation}
	\label{xi_i_cases}
	\sav(\xi^i; B_{\tilde{R}_2(x)}(x)) \geq \begin{cases} \frac{3}{2} \gamma + 1, & u^{(x)} \notin C^2(\bT_m), \\
	\frac{7}{8} \sav(G_{\bT_m} * u^{(x)}), & u^{(x)} \in C^2(\bT_m). \end{cases}
	\end{equation}
	Note that $\frac{3}{2} \gamma + 1 > m^2 \gap_m$ for large enough $m$, and $\sav(G_{\bT_m} * u^{(x)}) \geq m^2 \gap_m$ by definition of $\sX''$. Thus, the combination of \eqref{xi_i_reduction} with \eqref{xi_i_cases} verifies the desired bound \eqref{eta_bound} except when $|\sX''| = 1$ and $u^{(x)} \in C^2(\bT_m)$. In that remaining situation, we observe from \eqref{xi_star_approx} that
	\begin{align}
	\sav(\xi^i; B_{R_1}(x)) &= \sav(G_{\bT_m} * u^{(x)}; B_{R_1}(x)) \\
	\notag &= \sav(G_{\bT_m} * u^{(x)}) - O_B\left( \frac{\log^2 R_1}{R_1^2} \right) \\
	\notag &\geq m^2 \gap_m - O_B\left( \frac{\log^2 R_1}{R_1^2} \right),
	\end{align}
	which along with \eqref{xi_i_reduction} completes the proof.
	
	In the second and third cases, we assume that there exist $x \in C$ and $y \in B_{R_1}(x)$ such that $d := \|\xi_y^e - \xi_x^e\|_{\bR/\zed} \geq 1/R_1^3$. Set $w = y - x$, so $\|w\|_1 \leq R_1$. For the second case, suppose that for all integers $1 \leq n \leq \frac{R_2}{\|w\|_1}$,
	\begin{equation}\label{linear_constraint}
	\left\|\xi_{x + nw}^e - 
	\xi_x^e - n\left(\xi_y^e 
	- \xi_x^e\right)\right\|_{\bR/\zed} < \frac{1}{R_1}.
	\end{equation} 
	Effectively, the external phase varies linearly along the discrete line $\{x + nw : n \in \zed,\, 0 \leq n \leq \frac{R_2}{\|w\|_1} \}$.
	
	We now find a segment along the line that is far away from $C$. Set
	\begin{equation}
	\ell = \left\lfloor \frac{R_2}{3^B \|w\|_1} \right\rfloor,
	\end{equation}
	and consider the $\ell^1$-balls of radius $3^k \ell \|w\|_1$ centered at $x + 2 \cdot 3^k \ell w$, for $0 \leq k \leq B-1$. The interiors of these balls are disjoint, and $x \in C$ is not in any of the interiors. By the pigeonhole principle, the interior of at least one ball contains no elements of $C$. Choose $k$ corresponding to one such ball, and set $U = 2 \cdot 3^k \ell$, so that $x + Uw$ is the center.
	
	Set $V = \lfloor \sqrt{\ell d^{-1}} \rfloor$. By \eqref{R_123}, $\frac{\ell}{V} \geq \sqrt{\ell d} \to \infty$ with $R$, and certainly $V \to \infty$ with $R$. Any point $y$ along a shortest path from $x + Uw$ to $x + nw$, with $U < n \leq U+V$, satisfies $d(y,C) \gg U \|w\|_1$. Since the first derivatives of $G_{\bT_m}$ decay like the inverse of the radius, it follows that $\xi^i_{x + nw} - \xi^i_{x + Uw} = O\left( \frac{V}{U} \right) = o_R(1)$.
	
	Consider the exponential sum
	\begin{align}
	&\quad\, \sum_{n = U}^{U+V} e\left(\overline{\xi}_{x + nw}\right) = \sum_{n = U}^{U+V} e\left(\xi_{x + nw}^e + \xi_{x + nw}^i\right) \\
	&= \notag \sum_{n = U}^{U+V} e\left(\xi_{x}^e + 
	\xi_{x + Uw}^i + n\left(\xi_y^e - \xi_x^e \right) + 
	O\left(\frac{1}{R_1} \right) + O\left(\frac{V}{U} 
	\right)\right).
	\end{align}
	Taylor expanding the error in the exponential, then 
	summing the geometric series, we obtain
	\begin{equation}
	\sum_{n = U}^{U+V} e\left(\overline{\xi}_{x + nw}\right) = 
	O\left(\frac{V}{R_1} + \frac{V^2}{U} + \frac{1}{d} \right) = o_R(V).
	\end{equation}
	Hence this segment of the line provides savings of $\gg V$ for $\overline{\xi}$. That is, $\sav(\overline{\xi}; \nbd(C))$ is bounded below by a constant that may be made arbitrarily high by taking $R$ large enough.
	
	For the third case, suppose that \eqref{linear_constraint} fails for some $n = n_1 \leq \frac{R_2}{\|w\|_1}$. Set $W = \lfloor \frac{R_3}{\|w\|_1} \rfloor$, and apply van der Corput's inequality with $H=1$ to estimate
	\begin{equation}
	\label{van_der_Corput}
	\left| \sum_{n = 1}^W e\left(\overline{\xi}_{x + nw} 
	\right)\right|^2 \leq \frac{W(W+1)}{2} + \frac{W+1}{2} \left| \sum_{n = 1}^{W-1} e\left(\overline{\xi}_{x + (n+1)w} - \overline{\xi}_{x + nw} \right)\right|.
	\end{equation}
	Set $z_n = \xi_{x + (n+1)w}^e - \xi_{x + nw}^e$. By the definition of $n_1$,
	\begin{equation}
	\frac{1}{R_1} \leq \left\| \sum_{i=0}^{n_1-1} (z_i) - n_1 z_0 \right\|_{\bR/\zed} = \left\| \sum_{i=0}^{n_1-2} (n_1-1-i)(z_{i+1} - z_i) \right\|_{\bR/\zed}
	\end{equation}
	and therefore there is $n_0 \leq n_1-2$ for which
	\begin{equation}
	\delta := \|z_{n_0+1} - z_{n_0}\|_{\bR/\zed} \geq \frac{2}{R_1 n_1^2} \geq \frac{2 \|w\|_1^2}{R_1 R_2^2}.
	\end{equation}
	For all $1 \leq n,p \leq W$, the quantity $z_n - z_p -(n-p)(z_{n_0+1} - z_{n_0})$ is a sum of $O\left( W^2 \|w\|_1^3 \right)$ terms of the form $D_1^a D_2^b \xi_{x'}^e$ where $a+b = 3$ and $x' \in \nbd(C)$. (Each $z_i$ is a sum of $O\left( \|w\|_1 \right)$ first derivatives of $\xi^e$, so $z_{i+1} - z_i$ is a sum of $O\left( \|w\|_1^2 \right)$ second derivatives of $\xi^e$, and then $(z_{i+1} - z_i) - (z_{n_0+1} - z_{n_0})$ is a sum of $O\left( \|w\|_1^2 \cdot W \|w\|_1 \right)$ third derivatives of $\xi^e$. Finally, sum over all $i$ between $p$ and $n$.) Thus, \eqref{3rd_deriv_bound} gives
	\begin{equation}
	\left\|z_n - z_p -(n-p)(z_{n_0+1} - z_{n_0})\right\|_{\bR/\zed} = O\left( \frac{W^2 \|w\|_1^3}{R} \right).
	\end{equation}
	By the definition of $W$ and \eqref{R_123}, this quantity is $O\left( 1/R_1 \right)$.
	
	We now repeat the argument of the previous case, using
	\begin{equation}
	\ell' = \left\lfloor \frac{R_3}{3^B \|w\|_1} \right\rfloor
	\end{equation}
	to define $U' = 2 \cdot 3^k \ell'$ for an appropriately chosen $0 \leq k \leq B-1$, and $V' = \lfloor \sqrt{\ell' \delta^{-1}} \rfloor$. By \eqref{R_123}, $\frac{\ell'}{V'} \geq \sqrt{\ell' \delta} \to \infty$ with $R$, and $\delta^{-1} = o_R(V')$. Arguing as before, we obtain
	\begin{equation}
	\sum_{n = U'}^{U'+V'} e\left(\overline{\xi}_{x + (n+1)w} - \overline{\xi}_{x + nw}\right) = O\left( \frac{V'}{R_1} + \frac{(V')^2}{U'} + \frac{1}{\delta} \right) = o_R(V').
	\end{equation}
	Thus, we have saved an arbitrary constant in the sum
	\begin{equation}
	\sum_{n = 1}^{W-1} e\left(\overline{\xi}_{x + (n+1)w} - \overline{\xi}_{x + nw} \right),
	\end{equation}
	and hence also in $\sum_{n=1}^W e\left(\overline{\xi}_{x + nw} 
	\right)$, by \eqref{van_der_Corput}.
\end{proof}

\section{Proof of Theorem \ref{mixing_time_theorem}}
\label{proof_mixing_theorem_section}
In the process of proving Theorem \ref{mixing_time_theorem}, we also prove the following mixing result in $L^2$.
\begin{theorem}\label{L_2_mixing_theorem}
	Let $m \geq 2$, let $c_0 = \gamma^{-1}$ be the constant of Theorem \ref{mixing_time_theorem}, and as there, set $t_m^{\mix} = c_0 m^2 \log m$.  For each fixed $\epsilon > 0$,
	\begin{align}
		\lim_{m \to \infty} \min_{\sigma \in \sR_m}\left\|P_m^{\lceil(1-\epsilon)t_m^{\mix}\rceil} \delta_\sigma - \bU_{\sR_m}\right\|_{L^2(d\bU_{\sR_m})} &= \infty\\
		\notag \lim_{m \to \infty} \max_{\sigma \in \sR_m}\left\|P_m^{\lfloor(1+\epsilon)t_m^{\mix}\rfloor} \delta_\sigma - \bU_{\sR_m}\right\|_{L^2(d\bU_{\sR_m})} &= 0.
	\end{align}
\end{theorem}
Note that, since we restrict to recurrent states, Parseval gives the following characterization of the $L^2(d\bU_{\sR_m})$ norm,
\begin{equation}
\left\|P_m^N \delta_{\sigma} - \bU_{\sR_m} \right\|_{L^2(d\bU_{\sR_m})}^2 = \sum_{\xi \in \hat{\sG}_m \setminus \{0\}}\left|\hat{\mu}(\xi) \right|^{2N}.
\end{equation}
\subsection{Proof of the lower bound}
Our proof of the lower bound in Theorem \ref{mixing_time_theorem} uses the
 following second moment lemma, a variant of the method used by Diaconis and 
Shahshahani \cite{DS87} to show cutoff in the Bernoulli--Laplace diffusion model 
(see also \cite{D88}).

Given any probability measure $\mu$ on a finite abelian group $\sG$, recall from Section \ref{Random_walk_on_the_sandpile_group} the definitions of the dual group $\hat{\sG}$ and the Fourier coefficients $\hat{\mu}(\xi)$, for $\xi \in \hat{\sG}$.

\begin{lemma}\label{lower_bound_lemma}
Let $\sG$ be a finite abelian group, let $\mu$ be a probability measure on $\sG$ and 
let $N \geq 1$. Let $\sX \subset \hat{\sG} \,\setminus\, \{0\}$.  Suppose that the following inequalities
hold for some parameters $0 < \epsilon_1, \epsilon_2 < 1$,
\begin{align}
  \sum_{\xi \in \sX}\left|\hat{\mu}(\xi)\right|^N &\geq
\frac{|\sX|^{\frac{1}{2}}}{\epsilon_1}\\ \notag
 \sum_{\xi_1, \xi_2 \in \sX}\left|\hat{\mu}(\xi_1-\xi_2)\right|^N & \leq (1
+\epsilon_2^2)\left(\sum_{\xi \in \sX}\left|\hat{\mu}(\xi)\right|^N \right)^2.
\end{align}
Then
\begin{equation}
 \left\|\mu^{*N}- \bU_{\sG}\right\|_{\TV(\sG)} \geq 1 - 4\epsilon_1^2 - 4
\epsilon_2^2.
\end{equation}
\end{lemma}
\begin{proof}
Define, for $\xi \in \sX$, $w_\xi =
\left(\frac{\overline{\hat{\mu}(\xi)}}{\left|\hat{\mu}(\xi)\right|}\right)^N$, 
and $f \in
L^2(\sG)$ by
\begin{equation}
f(x) = \sum_{\xi\in \sX}w_\xi e(\xi\cdot x).
\end{equation}
Then
\begin{align}
 \E_{\bU}[f] = 0, \qquad &\E_{\bU}\left[\left|f\right|^2\right] = |\sX|,\\ 
\notag
 \E_{\mu^{*N}}[f] = \sum_{\xi \in \sX}\left|\hat{\mu}(\xi)\right|^N, \qquad
 &\E_{\mu^{*N}}\left[|f|^2 \right] = \sum_{\xi_1, \xi_2 \in \sX}
w_{\xi_1}\overline{w_{\xi_2}}\hat{\mu}(\xi_1 - \xi_2)^N.
\end{align}
Define $A = \left\{g \in \sG: |f(g)| >
\frac{1}{2}E_{\mu^{*N}}[f]\right\}$.  By Chebyshev's inequality, $\bU(A) \leq
4\epsilon_1^2$, while by the same inequality, $\mu^{*N}(A) \geq 1-
4\epsilon_2^2$, from which the claim follows.
\end{proof}

\begin{proof}[Proof of Theorem \ref{mixing_time_theorem}, lower bound]
In light of Proposition \ref{gap_achievers}, choose $v \in \sC(B_0,R_0)$ such that the frequency $\xi = \xi(v) \in \hat{\sG}_m$ generates the spectral gap. Choose a large fixed constant $R > R_0$ and let 
$\{v_i\}_{i=1}^{M}$ be a 
collection of $R$-separated translates of $v$, with $M \asymp \frac{m^2}{R^2}$.  The corresponding frequencies $\xi_i = \xi(v_i)$ all satisfy $|\hat{\mu}(\xi_i)| =  |\hat{\mu}(\xi)| = 1 - \gap_m$.
  
Given $c > 0$, set $N = \left\lfloor(\log m -c) \gap_m^{-1} \right\rfloor$ and apply Lemma 
\ref{lower_bound_lemma} with 
set of frequencies $\sX = \{\xi_i\}_{i=1}^M$. Calculate 
\begin{equation}
\label{hat_mu}
 \left|\hat{\mu}(\xi_i)\right|^N = \frac{e^c}{m} \left[ 1 + O\left(\frac{\log 
m}{m^2} \right) \right].
\end{equation}
If $c$ is sufficiently large, then the first condition of Lemma 
\ref{lower_bound_lemma} is satisfied with $\epsilon_1 = O\left(R e^{-c} 
\right).$

Write $d(v_i,v_j)$ for the $\ell^1$ distance between the supports of $v_i$ and $v_j$. If $d(v_i,v_j) \geq \rho$, then by Lemma \ref{C_2_sum_lemma},
\begin{equation}
1- \left|\hat{\mu}(\xi_i -\xi_j)\right| =2(1- |\hat{\mu}(\xi)|) +  
O\left(\frac{\log(1+\rho)}{\rho^2 m^2}\right)
\end{equation}
and therefore we can compute
\begin{equation}
\label{hat_mu_n}
\left|\hat{\mu}(\xi_i -\xi_j)\right|^N = e^{2c} m^{-2 + O(\log(1+\rho) / \rho^2)}.
\end{equation}
Choose $R$ large enough that when $\rho = R$ in \eqref{hat_mu_n}, the power of $m$ is less than $-1$. Then, by separating the cases $i=j$ and $i \neq j$,
\begin{equation}
\sum_{\substack{1 \leq i,j \leq M \\ d(v_i,v_j) < \log m}} \left|\hat{\mu}(\xi_i -\xi_j)\right|^N = O\left(\frac{m^2}{R^2} \right) + e^{2c} O\left( \frac{m}{R^4} \right),
\end{equation}
since the number of pairs $(i,j)$ in the sum is $O(m^2 \log^2(m) / R^4)$. In addition, using $\rho = \log m$ in \eqref{hat_mu_n} and plugging in \eqref{hat_mu},
\begin{align}
\sum_{\substack{1 \leq i,j \leq M \\ d(v_i,v_j) \geq \log m}} \left|\hat{\mu}(\xi_i -\xi_j)\right|^N &= \sum_{\substack{1 \leq i,j \leq M \\ d(v_i,v_j) \geq \log m}} e^{2c} m^{-2} \left( 1 + O\left( \frac{\log \log m}{\log m} \right) \right) \\
\notag &\leq M^2 |\hat{\mu}(\xi)|^{2N} \left( 1 + O\left( \frac{\log \log m}{\log m} \right) \right).
\end{align}
Therefore, since $M^2 |\hat{\mu}(\xi)|^{2N} \asymp e^{2c} m^2 / R^4$,
\begin{equation}
\sum_{1 \leq i,j \leq M} \left|\hat{\mu}(\xi_i -\xi_j)\right|^N \leq M^2 |\hat{\mu}(\xi)|^{2N} \left( 1 + O\left( \frac{\log \log m}{\log m} + \frac{R^2}{e^{2c}} + \frac{1}{m} \right) \right)
\end{equation}
and the second condition of Lemma \ref{lower_bound_lemma} is met with $\epsilon_2 = O\left( Re^{-c} \right)$.
\end{proof}

\begin{proof}[Proof of Theorem \ref{L_2_mixing_theorem}, lower bound]
	By Cauchy-Schwarz, the condition $ \sum_{\xi \in \sX}\left|\hat{\mu}(\xi)\right|^N \geq
	\frac{|\sX|^{\frac{1}{2}}}{\epsilon_1}$ implies
	\begin{equation}
		\sum_{\xi \in \sX}\left|\hat{\mu}(\xi)\right|^{2N} \geq \frac{1}{\epsilon_1^2}.
	\end{equation}
	By the proof of the lower bound above, since $\epsilon$ is fixed, $\epsilon_1$ may be taken arbitrarily small, which proves the $L^2$ lower bound.
\end{proof}

\subsection{Proof of the upper bound}
\label{Proof_of_the_upper_bound}
Recall that Proposition \ref{reduction_proposition} reduces Theorem \ref{mixing_time_theorem} to the case where the starting state is recurrent. We prove the upper bound of Theorem \ref{L_2_mixing_theorem}, which implies the upper bound of Theorem \ref{mixing_time_theorem} by Cauchy-Schwarz.
We consider mixing at step
\begin{equation}
 N = \lfloor (1+\epsilon)\gap_m^{-1}\log m \rfloor \asymp m^2 \log m.
\end{equation}

Let $R = R(\epsilon)$ be a parameter which is fixed as a function of $m$, to be determined at the end of the argument. Given frequency $\xi \in \hat{\sG}_m$, let $v \in \zed_0^{\bT_m}$ be its $R$-reduced prevector, and perform the clustering algorithm of Section \ref{small_phase_section} on $v$ with parameter $R$.

Let $\sN(V,K)$ denote the number of $R$-reduced prevectors $v$ of $L^1$ mass $V$ in $K$ clusters.  
\begin{lemma}
	The following upper bound holds:
	\begin{equation}
	\sN(V,K) \leq \exp\left(K \log(m^2) 
	+ O(V \log R) \right).
	\end{equation}
\end{lemma}

\begin{proof}
We provide a recipe to generate all possible prevectors by adding mass one point at a time. Let $\Gamma$ be a lattice path from $(0,0)$ to $(V,V)$, that moves either upward or rightward at each step, and that never passes above the main diagonal. Assume that $\Gamma$ has exactly $K-1$ intersection points with the main diagonal strictly between $(0,0)$ and $(V,V)$. Let the rightward edges go from $(i, k(i))$ to $(i+1, k(i))$, for $0 \leq i \leq V-1$. The sequence $\{k(i)\}_{i=0}^{V-1}$ is non-decreasing, with $0 \leq k(i) \leq i$, and there are $K$ values of $i$ for which $k(i) = i$ (including $i = 0$).

To generate a prevector $v$ using the path $\Gamma$:
\begin{enumerate}
\item Iterate from $i = 0$ to $V-1$:
\begin{enumerate}
\item If $k(i) = i$, start a new cluster by adding one unit of mass to a point $x_i$ that is separated from the set of previously placed points $\{x_j\}_{j < i}$ by a distance greater than $2R$.
\item If $k(i) < i$, add one unit of mass to a point $x_i$ whose distance from the previously placed point $x_{k(i)}$ is at most $2R$. The possibility $x_i = x_{k(i)}$ is allowed.
\end{enumerate}
\item For each $x \in \supp(v) = \bigcup_{i=0}^{V-1} \{x_i\}$, let $w(x) = \#\{i : x_i = x\}$ be the total mass at $x$, and choose $v(x) \in \{-w(x), w(x)\}$.
\end{enumerate}

Every prevector $v$ with $L^1$ mass $V$ in $K$ clusters can be generated by this procedure. For each path $\Gamma$, since step (a) is taken $K$ times, the number of possible prevectors is $O\left( (m^2)^K \cdot (R^2)^V \cdot 2^V \right)$. The number of paths $\Gamma$ is bounded by the $V$-th Catalan number, $\frac{1}{V+1} \binom{2V}{V} \leq 2^{2V}$. Hence
\begin{equation}
\sN(V,K) = O\left( m^{2K} R^{2V} 8^V \right),
\end{equation}
which has the desired form.
\end{proof}

\begin{proof}[Proof of Theorem \ref{L_2_mixing_theorem}, upper bound]
In
 \begin{equation}
  \left\|P_m^N \delta_{\sigma} - \bU_{\sR_m} \right\|_{L^2(d\bU_{\sR_m})}^2 = \sum_{0 \neq \xi \in 
\hat{\sG}_m}\left|\hat{\mu}(\xi)\right|^{2N},
 \end{equation}
write $\Xi(V,K)$ for the collection of nonzero frequencies $\xi \in \hat{\sG}_m$ such that the $R$-reduced prevector of $\xi$ has $L^1$ norm $V$ in $K$ $R$-clusters. Thus
\begin{equation}
\label{full_mu_sum}
\left\|P_m^N \delta_{\sigma} - \bU_{\sR_m} \right\|_{L^2(d\bU_{\sR_m})}^2 = \sum_{K \geq 1} \sum_{V \geq K} \sum_{\xi \in \Xi(V,K)}\left|\hat{\mu}(\xi)\right|^{2N}.
\end{equation}
From the definition of $R$-reduction in Section \ref{small_phase_section}, the bound of Lemma \ref{l_2_fourier_lemma} applies also to $R$-reduced prevectors. Thus, there is a universal constant $c > 0$ such that every $\xi \in \Xi(V,K)$ satisfies
\begin{equation}
|\hat{\mu}(\xi)|^{2N} \leq \exp(-cV \log m).
\end{equation}

Let $A>0$ be a fixed integer constant. Then,
\begin{align}
& \quad\, \sum_{K \geq 1} \sum_{V \geq AK} \sum_{\xi \in 
\Xi(V,K)}\left|\hat{\mu}(\xi)\right|^{2N} \\
 &\notag \leq \sum_{K \geq 1} \sum_{V \geq AK} \sN(V,K) 
\exp\left(-c V \log m \right)\\
 &\notag \leq \sum_{K \geq 1} \sum_{V \geq AK} \exp\Big(K \log(m^2) - V [c\log m - O(\log R)] 
\Big).
\end{align}
For sufficiently large $m$, the coefficient of $V$ in the last expression is at least $\frac{c}{2} \log m$. Then, if $Ac > 4$, we sum the two geometric series:
\begin{align}
& \quad\, \sum_{K \geq 1} \sum_{V \geq AK} \exp\left( K \log(m^2) -V((c/2) \log m) \right) \\
\notag &= \sum_{K \geq 1} \exp\left( 2K \log m - AK(c/2) \log m \right) (1 + o(1)) \\
\notag &= m^{2 - Ac/2} (1 + o(1)),
\end{align}
where the $o(1)$ is as $m \to \infty$. Choose $A$ so that $2 - Ac/2 \leq -1$.

To estimate the remaining sum over $K \leq V < AK$, let $\delta = \epsilon/3$ and set $B = A \delta^{-1}$. Choose $R = R(\epsilon)$ according to Lemma \ref{savings_lemma}, so that the savings from each $R$-cluster of size at most $B$ is at least $m^2 \gap_m \left(1 - \epsilon/2 \right)$. If $\xi \in \Xi(V,K)$ with $V < AK$, then its $R$-reduced prevector has at least $(1-\delta)K$ clusters of size at most $B$. Hence
\begin{equation}
1 - |\hat{\mu}(\xi)| \geq (1-\delta)K \cdot \gap_m \left( 1 - \frac{\epsilon}{2} \right) \geq \left( 1 - \frac{5\epsilon}{6} \right) \gap_m K,
\end{equation}
and therefore
\begin{align}
|\hat{\mu}(\xi)|^{2N} &\leq \exp\left( \left[ -2(1+\epsilon)\left(1-\frac{5\epsilon}{6}\right) \log m + O\left(m^{-2}\right) \right] K \right) \\
\notag &\leq \exp(-(2 + \beta)(\log m) K)
\end{align}
for some constant $\beta = \beta(\epsilon) > 0$, as long as $\epsilon$ is sufficiently small.

We compute, for sufficiently large $m$,
\begin{align}
&\quad\, \sum_{K \geq 1} \sum_{K \leq V < AK} \sum_{\xi \in \Xi(V,K)} |\hat{\mu}(\xi)|^{2N} \\
\notag &\leq \sum_{K \geq 1} \sum_{K \leq V < AK} \sN(V,K) \exp(-(2 + \beta)(\log m) K) \\
\notag &\leq \sum_{K \geq 1} \sum_{K \leq V < AK} \exp(-\beta(\log m) K + O(V \log R)) \\
\notag &\leq \sum_{K \geq 1} \exp\Big( \left[ -\beta \log m + O(A \log R) \right] K \Big) \\
\notag &\leq \sum_{K \geq 1} \exp(-(\beta/2) (\log m) K) = O\left( m^{-\beta/2} \right).
\end{align}
Thus the entire sum \eqref{full_mu_sum} tends to zero like a small negative power of $m$, completing the proof.
\end{proof}

\appendix
\section{Local limit theorem}\label{local_limit_theorem_appendix}
Let $\nu_{\zed^2}$ be the measure on $\zed^2$ given by
\begin{equation}
 \nu_{\zed^2} := \frac{1}{4}\left(\delta_{(1,0)} + \delta_{(-1,0)} + 
\delta_{(0,1)} + \delta_{(0,-1)}  \right),
\end{equation}
while $\nu$ is the same measure on $\bT_m$.
Below we prove a local limit theorem for repeated convolutions of 
$\nu_{\zed^2}$.  Before doing so, we recall a Chernoff-type tail inequality.
\begin{theorem}[Chernoff's Inequality]
 Let $X_i$, $1 \leq i \leq n$, be mutually independent random variables with
 \begin{equation}
  \Prob\left(X_i = +1\right) = \Prob\left(X_i = -1\right) = \frac{1}{2}.
 \end{equation}
Let  $S_n = X_1 + \cdots + X_n$.  For any $a > 0$,
\begin{equation}
 \Prob\left(|S_n|>a\right) \leq 2 e^{-\frac{a^2}{2n}}.
\end{equation}
\end{theorem}
See \cite{AS16}, pp. 321--322.

\begin{theorem}[Local Limit Theorem on $\zed^2$]
There are polynomials $\{P_k\}_{k=0}^\infty$ with $P_k$ of degree at most $k$, 
such that for any $i, j \in \zed$ and any $a, b, N \geq 0$, we have  
 \begin{align}
 & \delta_1^{*a}* \delta_2^{*b} * \frac{\nu_{\zed^2}^{*N}+ \nu_{\zed^2}^{*(N+1)}}{2} 
(i,j)=\exp\left(-\frac{i^2 + 
j^2}{ N} \right) \times\\\notag& 
\Biggl(\frac{ P_a\left(\frac{i}{\sqrt{N}} \right)P_b\left(\frac{j}{\sqrt{N}} 
\right)}{N^{\frac{a + 
b+ 2}{2}}}\\\notag&+O\left(\frac{|i|^a|j|^b(|i|+|j|)}{N^{a+b+2}} + 
\frac{|i|^{a+1}}{N^{a+2 + \frac{b}{2}}} + \frac{ 
|j|^{b+1}}{N^{b+2+\frac{a}{2}}} + \frac{1}{N^{\frac{a+b+3}{2}}} \right)\Biggr) 
\\\notag&+O_\epsilon\left(\exp\left(-N^{\frac{1}{2}-\epsilon}\right) \right).
\end{align}
Moreover, $P_k$ is an even function if $k$ is even and an odd function if $k$ is odd.
 \end{theorem}

\begin{proof}
Apply Chernoff's inequality to reduce to $i^2 + j^2 \leq 
N^{\frac{3}{2}-\epsilon}$.

The quantity in question is given by
\begin{align}
 I(i,j; a,b; N) = &(-2\sqrt{-1})^{a+b}\int_{(\bR/\zed)^2} s\left(\frac{x}{2}\right)^a 
s\left(\frac{y}{2}\right)^b \left(\frac{2 + c(x)+c(y)}{4} \right) \\& 
\notag \left(\frac{c(x) + c(y)}{2} \right)^N 
e\left(x\left(i-\frac{a}{2}\right) + 
y\left(j-\frac{b}{2}\right)\right)  dx dy.
\end{align}
Truncate the integral at $\|x\|_{\bR/\zed}^2 + \|y\|_{\bR/\zed}^2 \leq 
N^{-\frac{1}{2}}$, since the remainder of 
the integral trivially satisfies the claimed bound.  Now treat $x, y$ as complex variables. Set 
\begin{equation}
 x_1 = \sqrt{2\pi^2 N} x, \qquad y_1 =\sqrt{2\pi^2 N} y.
\end{equation}
  Now replace
\begin{align}
 x_2 &:= x_1 - \sqrt{\frac{-2}{N}}\left(i -\frac{a}{2}\right)\\ \notag
 y_2 &:= y_1 - \sqrt{\frac{-2}{N}}\left(j -\frac{b}{2}\right),
\end{align}
and shift the contour to $\IM(x_2) = \IM(y_2) = 0$.  In doing so, an integral 
on $\RE (x_2)^2 + \RE (y_2)^2 = 2\pi^2 \sqrt{N}$ is created, with $|\IM (x_2)| 
\leq \frac{|i| + O(1)}{\sqrt{N}}$, $|\IM (y_2)| \leq \frac{|j| + 
O(1)}{\sqrt{N}}$ on this integral.  
On this integral, for some $C>0$, $\RE \log \left(\frac{c(x) + c(y)}{2} \right) 
\leq -\frac{C}{\sqrt{N}}$, so that this integral satisfies the claimed bound and 
may be discarded. 

Expand
\begin{align}
& \left(\frac{c(x) + c(y)}{2} \right)^N e\left(x\left(i-\frac{a}{2}\right) + 
y\left(j-\frac{b}{2}\right)\right) \\ \notag &= \exp\left(-\frac{x_2^2 + 
y_2^2}{2} -\frac{\left(i - \frac{a}{2} \right)^2 + \left(j - \frac{b}{2} 
\right)^2}{N}\right)\\\notag & \qquad \times \left(1 + O\left(\frac{x_2^4 + 
y_2^4}{N} + \frac{i^4 + 
j^4+1}{N^3} \right) \right).
\end{align}
Taylor expand $s\left(\frac{x}{2}\right)^a, s\left(\frac{y}{2}\right)^b$, $ 
\left(\frac{2 + 
c(x)+c(y)}{4} \right)$, dropping all but the lowest order terms to obtain the 
claimed asymptotic. The final claim regarding the parity of the polynomials $P_k$ follows 
since in the main term, integration against odd powers of $x$ or $y$ vanishes 
by symmetry.
\end{proof}

\begin{proof}[Proof of Lemma \ref{greens_function_estimate}]
One has
\begin{align}
D_1^a D_2^b G_{\bT_m}(i,j) &=  \sum_{n=0}^\infty\sum_{k, \ell \in \zed}
\delta_1^{*a}*
\delta_2^{*b}*\nu_{\zed^2}^{*n}(i+ km,j+\ell m).
\end{align}

Set $R = i^2 + j^2$ and write
\begin{align} \notag
&D_1^a D_2^b G_{\bT_m}(i,j)\\ \label{time_domain}
&=  \sum_{k, \ell \in \zed}
\delta_1^{*a}*
\delta_2^{*b}*\sum_{0 \leq n < R}\nu_{\zed^2}^{*n}(i+ km,j+\ell m)\\
\label{frequency_domain}&+
\delta_1^{*a}*
\delta_2^{*b}*\sum_{R \leq n }\nu^{*n}(i,j).
\end{align}

By the local limit theorem on $\bR^2$,
\begin{align*}
 (\ref{time_domain}) 
& \ll  \sum_{0 < n < R } \sum_{k,
\ell \in \zed} e^{- \frac{(i+km)^2 + (j+\ell m)^2}{2n}} \frac{(1 +
|i +km|)^a (1 + |j + \ell m|)^b}{n^{1+a+b}}\\& \notag \ll 
\frac{1}{R^{\frac{a+b}{2}}}.
\end{align*}
Expand (\ref{frequency_domain}) in characters of $(\zed/m\zed)^2$ to obtain
\begin{align*}
 (\ref{frequency_domain}) \ll& \frac{1}{m^2}\sum_{(0,0) \neq (\xi, \eta) \in
(\zed/m\zed)^2}\left|1-e\left(\frac{\xi}{m}\right) \right|^a\left|1 -
e\left(\frac{\eta}{m}\right) \right|^b\\
&\notag \frac{\left|\frac{c\left(\frac{\xi}{m} \right)+ c\left(\frac{\eta}{m}
\right)}{2} \right|^{R}}{1 - \left(\frac{c\left(\frac{\xi}{m} \right)+
c\left(\frac{\eta}{m}
\right)}{2} \right)}.
\end{align*}
Estimate $\left|1-e\left(\frac{\xi}{m} \right) \right|^a\left|1 -
e\left(\frac{\eta}{m}\right) \right|^b \ll \frac{|\xi|^a |\eta|^b}{m^{a+b}}$ 
and approximate the sum with an integral to obtain
 $\ll \frac{1}{R^{\frac{a+b}{2}}}$.
\end{proof}

\begin{proof}[Proof of Lemma 
\ref{green_function_differentiated_asymptotic}]
Without loss of generality, let $a = 1$, $b = 0$.  Let, for a large constant 
$C$, $T = \frac{m^2}{C\log m}$ and write
\begin{equation*}
 D_1 G_{\bT_m}(i,j) = \sum_{k, \ell \in \zed}\sum_{0 \leq n < T} \delta_1 * 
\nu_{\zed^2}^{*n}(i + km, j +\ell m) + \sum_{T \leq n} \delta_1 * \nu^{*n}(i,j).
\end{equation*}
For $n < T$, $\nu_{\zed^2}^{*n}$ has a distribution at scale $\sqrt{n}$, and 
thus if $C$ is sufficiently large, the $(k, \ell) \neq (0,0)$ terms of the 
first sum contribute $O\left(\frac{1}{m^2} \right)$.  Use the asymptotic 
from the local limit theorem to write the $(k,\ell) = (0,0)$ term as
\begin{align*}
 \sum_{0 \leq n < T} \delta_1 * \nu_{\zed^2}^{*n}(i,j) &= O\left(\frac{1}{i^2 
+ j^2} \right) + C i \sum_{0 \leq n <T} \frac{\exp\left(- \frac{i^2 + j^2}{n} 
\right)}{n^2}\\
&= \frac{C' i}{i^2 + j^2} + O\left(\frac{|i|+1}{T} \right) + 
O\left(\frac{1}{i^2 + j^2} \right).
\end{align*}
This gives the main term.  Bound the sum over large $n$ as before, by taking 
Fourier transform.  This obtains the bound, for some $C'' > 0$,
\begin{align*}
 &\frac{1}{m^2} \sum_{(0,0) \neq (\xi, \eta) \in (\zed/m\zed)^2} \left|1 - 
e\left(\frac{\xi}{m}\right)\right| \frac{\left| \frac{c\left(\frac{\xi}{m} 
\right)+ c\left(\frac{\eta}{m} \right) 
}{2}\right|^T}{1 - 
\frac{1}{2}\left(c\left(\frac{\xi}{m} \right)+ c\left(\frac{\eta}{m} \right) 
\right)}\\
&\ll \frac{1}{m} \sum_{(0,0) \neq (\xi, \eta) \in (\zed/m\zed)^2} 
\frac{|\xi|}{\xi^2 + \eta^2} \exp\left(-\frac{C'' T}{m^2}(\xi^2 + \eta^2) 
\right)
\\& \ll \frac{1}{\sqrt{T}}.
\end{align*}
The claimed error holds, since $T \gg (i^2 + j^2)^2$.
\end{proof}

\section{Determination of spectral gap}\label{spectral_gap_appendix}

In this appendix, we compute the value of
\begin{equation}
\label{gamma_appendix}
\gamma = \inf\left\{ f(G_{\zed^2} * v) : v \in C^2(\zed^2) \,\setminus\, \cI \right\}.
\end{equation}
Recall from Section \ref{Determination_of_gap} that
\begin{equation}
f(\xi) = \sum_{(i,j) \in \zed^2}\left(1- c\left( 
\xi_{(i,j)}\right)\right),
\end{equation}
while $\cI$ is the set of those $v \in C^2(\zed^2)$ for which $G_{\zed^2} * v$ is $\zed$-valued (equivalently, those $v \in C^2(\zed^2)$ for which $f(G_{\zed^2} * v) = 0$).

Given any $v \in C^2(\zed^2)$, reduce $\xi = G_{\zed^2} * v$ to $\tilde{\xi}: \zed^2 \to \left[ -\frac{1}{2}, \frac{1}{2} \right)$ by subtracting $w: \zed^2 \to \zed$. As we observed in the paragraph containing \eqref{v_tilde}, $\tilde{v} = \Delta \tilde{\xi}$ is also in $C^2(\zed^2)$ and satisfies $f(G_{\zed^2} * \tilde{v}) = f(G_{\zed^2} * v)$. In addition, $\tilde{v} \in \cI$ only if $\tilde{\xi} \equiv 0$ (so also $\tilde{v} \equiv 0$). Therefore we may write
\begin{equation}
\gamma = \inf\left\{ f(\xi) : \xi \in \{ G_{\zed^2} * v : 0 \not\equiv v \in C^2(\zed^2) \} \cap \left[ -\frac{1}{2}, \frac{1}{2} \right)^{\zed^2} \right\}.
\end{equation}

Write $d$ for the $\ell^1$ distance on $\zed^2$.  Given a set $S \subset 
\zed^2$, let $N = N(S)$ be its distance-1 enlargement
\begin{equation}
 N = \{(i,j) \in \zed^2 : d((i,j), S) \leq 1\}.
\end{equation}
A lower bound for $f(\xi)$ is obtained as the non-linear program $P(S, v)$,
\begin{align*}
 \text{minimize: } &\sum_{(i,j) \in N} \left(1 - \cos\left(2 \pi 
x_{(i,j)}\right) \right) \\
  \qquad \text{subject to: } &(x_{(i,j)})_{(i,j) \in N} \in 
\left[0, \frac{1}{2}\right]^N,\\ &\forall\, (k,\ell) \in S, \; 4x_{(k,\ell)} +\sum_{\|(i,j)-(k,\ell)\|_1 = 1} x_{(i,j)} \geq |v_{(k,\ell)}|.
\end{align*}
Indeed, if $\xi = G_{\zed^2} * v \in \left[ -\frac{1}{2}, \frac{1}{2} \right)^{\zed^2}$, then for any $S \subset \zed^2$, the function $x: N(S) \to \bR$ given by $x_{(i,j)} = |\xi_{(i,j)}|$ satisfies the constraints and so
\begin{equation}
f(\xi) \geq \sum_{(i,j) \in N} (1 - \cos(2\pi x_{(i,j)})) \geq P(S,v).
\end{equation}

\begin{lemma}
 The program $P(S,v)$ satisfies the following properties.  
 \begin{enumerate}[label=\arabic*.]
  \item If $S, T \subset 
\zed^2$ satisfy $d(S, T) \geq 3$ then $P(S \cup T, v) = P(S, v) + P(T, v)$.
 \item If $S \subset T \subset \zed^2$ then $P(S, v) \leq P(T, v)$.
 \item Denote by $P'(S, v)$ the more constrained program in which \[(x_{(i,j)})_{(i,j) 
\in N} \in \left[0, \frac{1}{4}\right]^N,\] with the same linear 
constraints.  This program has a unique local minimum.
 \end{enumerate}
\end{lemma}

\begin{proof}
 The first claim holds because the variables in the neighborhoods of $S$ and 
$T$ are disjoint and do not have a common constraint.  The second claim holds since for $S \subset T$, the program $P(T, v)$ is more constrained.  The last 
claim holds since the objective function is convex on $\left[0, 
\frac{1}{4}\right]^N$.
\end{proof}

We will also use the non-linear program $Q(S,v)$,
\begin{align*}
 \text{minimize: } &\sum_{(i,j) \in N} \left(1 - \cos\left(2 \pi 
x_{(i,j)}\right) \right) \\
  \qquad \text{subject to: } &(x_{(i,j)})_{(i,j) \in N} \in 
\left[-\frac{1}{2}, \frac{1}{2}\right]^N,\\ &\forall\, (k,\ell) \in S, \; 4x_{(k,\ell)} - \sum_{\|(i,j)-(k,\ell)\|_1 = 1} x_{(i,j)} = v_{(k,\ell)},
\end{align*}
which is also a lower bound for $f(\xi)$.

We performed the following steps to confirm that the constant $\gamma$ is obtained by 
$\xi^* = G_{\zed^2} * (\delta_1 * \delta_2)$.
\begin{enumerate}
 \item Calculate $f(\xi^*)$ by evaluating 
 \begin{equation}\label{fourier_inversion}
  \xi^*(m,n) = \frac{1}{4} \int_{(\bR/\zed)^2} \frac{(e(x)-1)(e(y)-1)}{1 - 
\frac{1}{2}(c(x) + c(y))} e(mx + ny)\,dx dy
 \end{equation}
 for $|m|, |n| \leq M.$  It is known that $\|\xi^*\|_2^2 = \frac{1}{2\pi}$, so 
we estimated
\begin{equation}
\sum_{|m|, |n| \leq M} \left(1 - \cos\left(2\pi \xi^*_{(m,n)}\right) 
-2 \pi^2 \left(\xi^*_{(m,n)}\right)^2 \right).
\end{equation}
By the decay of the Green's function, this determines $f(\xi^*)$ to within 
precision $M^{-6}$.  We thus obtained
\[
 f(\xi^*) = 2.868114013(4).
\]
The precision was verified by estimating
\begin{equation}
\sum_{\max(|m|, |n|)> M} \left|1 - c(\xi^*_{(m,n)}) - 2\pi^2 \left(\xi^*_{(m,n)}\right)^2\right| \leq \frac{2\pi^4}{3}\sum_{\max(|m|, |n|)> M}\left(\xi^*_{(m,n)}\right)^4
\end{equation}
and
\begin{equation}
\sum_{\max(|m|, |n|)> M}\left(\xi^*_{(m,n)}\right)^4 \leq 4 \left(\sum_{\substack{0 \leq m, n\\ \max(m,n)>M}}\left(\xi^*_{(m,n)}\right)^2\right)^2,
\end{equation}
using the symmetries $|\xi^*_{(m,n)}| = |\xi^*_{(-1-m, n)}|$, $|\xi^*_{(m,n)}| = |\xi^*_{(m, -1-n)}|$.

\item Suppose $\xi = G_{\zed^2} * v \in \left[ -\frac{1}{2}, \frac{1}{2} \right)^{\zed^2}$ satisfies $f(\xi) < 2.869$. The condition $\Delta \xi = v$ implies that $|v| \leq 3$. We ruled out prevectors with some $|v_{(i,j)}| = 3$ by 
considering $P(S,v)$ with $S = \{(0,0)\}$ and $|v_{(0,0)}| = 3$.
\item We ruled out $v_{(0,0)} = 2$ by first considering $Q(S,v)$ with $S = \{(0,0),(1,0)\}$ and $v_{(0,0)} = 2$, $v_{(1,0)} \in \{-2,-1,0,1,2\}$. The only choices giving $Q(S,v) < 2.869$ were $v_{(1,0)} = -1$ and $v_{(1,0)} = 0$. Then we considered $Q(S,v)$ for
\begin{equation}
S = \{(0,0),(-1,0),(1,0),(0,-1),(0,1)\}
\end{equation}
with $v_{(0,0)} = 2$ and $v_{(i,j)} \in \{-1,0\}$ for all $(i,j) \in S \setminus \{(0,0)\}$. All possibilities led to $Q(S,v) \geq 2.869$. It follows that any $v$ with $f(G_{\zed^2} * v) < 2.869$ satisfies $|v| \leq 1$.
\item Consider a set $S \subset \zed^2$ to be connected if $N(S)$ is connected 
in the usual sense.  Using the increasing property, we were able to enumerate 
all connected $S$ containing $(0,0)$, and such that $P(S, 1) < 2.869$ (all had 
$|S| \leq 6$).  This can be done iteratively starting from $S = \{(0,0)\}$ 
using the increasing property of $P(S, v)$ with set inclusion.   The minima in 
$P(S, 1)$ were rapidly calculated in each case using the SLSQP algorithm in 
SciPy's minimize package.  The minima can be verified, for instance, by noting 
that at most two variables can satisfy $x_{(i,j)} > \frac{1}{4}$ and by discretizing 
their values -- the remaining variables are then confirmed if the point is a 
local minimum. (This approach also verifies the minima in steps (2) and (3).)
\item Using the addition property of $P(S \cup T, v)$ for disconnected $S$ and 
$T$ it was verified that there is not a configuration of $v \in C^2(\zed^2)$ 
with disconnected support that meets the condition $P(\supp v, 1) < 2.869$.  
Of the remaining connected components, those which admit a configuration in 
$C^2(\zed^2)$ were estimated using Fourier inversion, as in 
(\ref{fourier_inversion}). The optimum was found to be $\xi^*$.
\end{enumerate}


\bibliographystyle{plain-arxiv}
\bibliography{sandpile-square-lattice}

\begin{thebibliography}{10}

\bibitem{AS16}
Noga Alon and Joel~H. Spencer.
\newblock {\em The probabilistic method}.
\newblock Wiley Series in Discrete Mathematics and Optimization. John Wiley \&
  Sons, Inc., Hoboken, NJ, fourth edition, 2016.

\bibitem{AJ04}
Siva~R. Athreya and Antal~A. J\'arai.
\newblock Infinite volume limit for the stationary distribution of abelian
  sandpile models.
\newblock {\em Comm. Math. Phys.}, 249(1):197--213, 2004.
\newblock \textsc{doi:}
  \href{http://dx.doi.org/10.1007/s00220-004-1080-0}{\texttt{10.1007/s00220-004-1080-0}}.

\bibitem{BTW87}
Per Bak, Chao Tang, and Kurt Wiesenfeld.
\newblock Self-organized criticality: An explanation of the 1/f noise.
\newblock {\em Phys. Rev. Lett.}, 59:381--384, Jul 1987.
\newblock \textsc{doi:}
  \href{http://dx.doi.org/10.1103/PhysRevLett.59.381}{\texttt{10.1103/PhysRevLett.59.381}}.

\bibitem{BTW88}
Per Bak, Chao Tang, and Kurt Wiesenfeld.
\newblock Self-organized criticality.
\newblock {\em Phys. Rev. A (3)}, 38(1):364--374, 1988.
\newblock \textsc{doi:}
  \href{http://dx.doi.org/10.1103/PhysRevA.38.364}{\texttt{10.1103/PhysRevA.38.364}}.

\bibitem{BS13}
Matthew Baker and Farbod Shokrieh.
\newblock Chip-firing games, potential theory on graphs, and spanning trees.
\newblock {\em J. Combin. Theory Ser. A}, 120(1):164--182, 2013.
\newblock arXiv: \href{http://arxiv.org/abs/1107.1313}{\texttt{1107.1313
  [math.CO]}}.

\bibitem{BHJ16}
Sandeep {Bhupatiraju}, Jack {Hanson}, and Antal~A. {J{\'a}rai}.
\newblock {Inequalities for critical exponents in $d$-dimensional sandpiles}.
\newblock {\em ArXiv e-prints}, February 2016.
\newblock arXiv: \href{http://arxiv.org/abs/1602.06475}{\texttt{1602.06475
  [math.PR]}}.

\bibitem{BL16}
Benjamin Bond and Lionel Levine.
\newblock Abelian networks {I}. {F}oundations and examples.
\newblock {\em SIAM J. Discrete Math.}, 30(2):856--874, 2016.
\newblock arXiv: \href{http://arxiv.org/abs/1309.3445}{\texttt{1309.3445
  [cs.FL]}}.

\bibitem{C15}
Hannah {Cairns}.
\newblock {Some halting problems for abelian sandpiles are undecidable in
  dimension three}.
\newblock {\em ArXiv e-prints}, August 2015.
\newblock arXiv: \href{http://arxiv.org/abs/1508.00161}{\texttt{1508.00161
  [math.CO]}}.

\bibitem{CE02}
Fan Chung and Robert~B. Ellis.
\newblock A chip-firing game and {D}irichlet eigenvalues.
\newblock {\em Discrete Math.}, 257(2-3):341--355, 2002.
\newblock \textsc{doi:}
  \href{http://dx.doi.org/10.1016/S0012-365X(02)00434-X}{\texttt{10.1016/S0012-365X(02)00434-X}}.
\newblock Kleitman and combinatorics: a celebration (Cambridge, MA, 1999).

\bibitem{DRSV95}
D.~Dhar, P.~Ruelle, S.~Sen, and D.-N. Verma.
\newblock Algebraic aspects of abelian sandpile models.
\newblock {\em J. Phys. A}, 28(4):805--831, 1995.
\newblock arXiv:
  \href{http://arxiv.org/abs/cond-mat/9408020}{\texttt{cond-mat/9408020}}.

\bibitem{D90}
Deepak Dhar.
\newblock Self-organized critical state of sandpile automaton models.
\newblock {\em Phys. Rev. Lett.}, 64(14):1613--1616, 1990.
\newblock \textsc{doi:}
  \href{http://dx.doi.org/10.1103/PhysRevLett.64.1613}{\texttt{10.1103/PhysRevLett.64.1613}}.

\bibitem{D99}
Deepak Dhar.
\newblock The abelian sandpile and related models.
\newblock {\em Physica A: Statistical Mechanics and its Applications}, 263(1):4
  -- 25, 1999.
\newblock arXiv:
  \href{http://arxiv.org/abs/cond-mat/9808047}{\texttt{cond-mat/9808047}}.

\bibitem{D88}
Persi Diaconis.
\newblock {\em Group representations in probability and statistics}, volume~11
  of {\em Institute of Mathematical Statistics Lecture Notes---Monograph
  Series}.
\newblock Institute of Mathematical Statistics, Hayward, CA, 1988.

\bibitem{DGM90}
Persi Diaconis, R.~L. Graham, and J.~A. Morrison.
\newblock Asymptotic analysis of a random walk on a hypercube with many
  dimensions.
\newblock {\em Random Structures Algorithms}, 1(1):51--72, 1990.
\newblock \textsc{doi:}
  \href{http://dx.doi.org/10.1002/rsa.3240010105}{\texttt{10.1002/rsa.3240010105}}.

\bibitem{DS87}
Persi Diaconis and Mehrdad Shahshahani.
\newblock Time to reach stationarity in the {B}ernoulli-{L}aplace diffusion
  model.
\newblock {\em SIAM J. Math. Anal.}, 18(1):208--218, 1987.
\newblock \textsc{doi:}
  \href{http://dx.doi.org/10.1137/0518016}{\texttt{10.1137/0518016}}.

\bibitem{FLP10}
Anne Fey, Lionel Levine, and Yuval Peres.
\newblock Growth rates and explosions in sandpiles.
\newblock {\em J. Stat. Phys.}, 138(1-3):143--159, 2010.
\newblock arXiv: \href{http://arxiv.org/abs/0901.3805}{\texttt{0901.3805
  [math.CO]}}.

\bibitem{FLW10}
Anne Fey, Lionel Levine, and David~B. Wilson.
\newblock Driving sandpiles to criticality and beyond.
\newblock {\em Phys. Rev. Lett.}, 104:145703, Apr 2010.
\newblock arXiv: \href{http://arxiv.org/abs/0912.3206}{\texttt{0912.3206
  [cond-mat.stat-mech]}}.

\bibitem{FMR09}
Anne Fey, Ronald Meester, and Frank Redig.
\newblock Stabilizability and percolation in the infinite volume sandpile
  model.
\newblock {\em Ann. Probab.}, 37(2):654--675, 2009.
\newblock arXiv: \href{http://arxiv.org/abs/0710.0939}{\texttt{0710.0939
  [math.PR]}}.

\bibitem{FR05}
A.~Fey-den Boer and F.~Redig.
\newblock Organized versus self-organized criticality in the abelian sandpile
  model.
\newblock {\em Markov Process. Related Fields}, 11(3):425--442, 2005.
\newblock arXiv:
  \href{http://arxiv.org/abs/math-ph/0510060}{\texttt{math-ph/0510060}}.

\bibitem{FU96}
Yasunari Fukai and K\^ohei Uchiyama.
\newblock Potential kernel for two-dimensional random walk.
\newblock {\em Ann. Probab.}, 24(4):1979--1992, 1996.
\newblock \textsc{doi:}
  \href{http://dx.doi.org/10.1214/aop/1041903213}{\texttt{10.1214/aop/1041903213}}.

\bibitem{HLMPPW08}
Alexander~E. Holroyd, Lionel Levine, Karola M\'esz\'aros, Yuval Peres, James
  Propp, and David~B. Wilson.
\newblock Chip-firing and rotor-routing on directed graphs.
\newblock In {\em In and out of equilibrium. 2}, volume~60 of {\em Progr.
  Probab.}, pages 331--364. Birkh\"auser, Basel, 2008.
\newblock arXiv: \href{http://arxiv.org/abs/0801.3306}{\texttt{0801.3306
  [math.CO]}}.

\bibitem{H15}
Bob {Hough}.
\newblock {Mixing and cut-off in cycle walks}.
\newblock {\em ArXiv e-prints}, December 2015.
\newblock arXiv: \href{http://arxiv.org/abs/1512.00571}{\texttt{1512.00571
  [math.NT]}}.

\bibitem{IK04}
Henryk Iwaniec and Emmanuel Kowalski.
\newblock {\em Analytic number theory}, volume~53 of {\em American Mathematical
  Society Colloquium Publications}.
\newblock American Mathematical Society, Providence, RI, 2004.
\newblock \textsc{doi:}
  \href{http://dx.doi.org/10.1090/coll/053}{\texttt{10.1090/coll/053}}.

\bibitem{JR08}
Antal~A. J\'arai and Frank Redig.
\newblock Infinite volume limit of the abelian sandpile model in dimensions
  {$d\geq 3$}.
\newblock {\em Probab. Theory Related Fields}, 141(1-2):181--212, 2008.
\newblock arXiv:
  \href{http://arxiv.org/abs/math/0408060}{\texttt{math/0408060}}.

\bibitem{JRS15}
Antal~A. J\'arai, Frank Redig, and Ellen Saada.
\newblock Approaching criticality via the zero dissipation limit in the abelian
  avalanche model.
\newblock {\em J. Stat. Phys.}, 159(6):1369--1407, 2015.
\newblock arXiv: \href{http://arxiv.org/abs/0906.3128}{\texttt{0906.3128
  [math.PR]}}.

\bibitem{JLP15}
Daniel~C. {Jerison}, Lionel {Levine}, and John {Pike}.
\newblock {Mixing time and eigenvalues of the abelian sandpile Markov chain}.
\newblock {\em ArXiv e-prints}, November 2015.
\newblock arXiv: \href{http://arxiv.org/abs/1511.00666}{\texttt{1511.00666
  [math.PR]}}.

\bibitem{JOP01}
Eric Jones, Travis Oliphant, Pearu Peterson, et~al.
\newblock {SciPy}: Open source scientific tools for {Python}, 2001--.
\newblock [Online; accessed 2017-02-13].

\bibitem{KS04}
Gady Kozma and Ehud Schreiber.
\newblock An asymptotic expansion for the discrete harmonic potential.
\newblock {\em Electron. J. Probab.}, 9:no. 1, 1--17, 2004.
\newblock arXiv:
  \href{http://arxiv.org/abs/math/0212156}{\texttt{math/0212156}}.

\bibitem{L15}
Lionel Levine.
\newblock Threshold state and a conjecture of {P}oghosyan, {P}oghosyan,
  {P}riezzhev and {R}uelle.
\newblock {\em Comm. Math. Phys.}, 335(2):1003--1017, 2015.
\newblock arXiv: \href{http://arxiv.org/abs/1402.3283}{\texttt{1402.3283
  [math.PR]}}.

\bibitem{LMPU16}
Lionel Levine, Mathav Murugan, Yuval Peres, and Baris~Evren Ugurcan.
\newblock The divisible sandpile at critical density.
\newblock {\em Ann. Henri Poincar\'e}, 17(7):1677--1711, 2016.
\newblock arXiv: \href{http://arxiv.org/abs/1501.07258}{\texttt{1501.07258
  [math.PR]}}.

\bibitem{MRS04}
C.~Maes, F.~Redig, and E.~Saada.
\newblock The infinite volume limit of dissipative abelian sandpiles.
\newblock {\em Comm. Math. Phys.}, 244(2):395--417, 2004.
\newblock \textsc{doi:}
  \href{http://dx.doi.org/10.1007/s00220-003-1000-8}{\texttt{10.1007/s00220-003-1000-8}}.

\bibitem{MQ05}
R.~Meester and C.~Quant.
\newblock Connections between `self-organised' and `classical' criticality.
\newblock {\em Markov Process. Related Fields}, 11(2):355--370, 2005.

\bibitem{M94}
Hugh~L. Montgomery.
\newblock {\em Ten lectures on the interface between analytic number theory and
  harmonic analysis}, volume~84 of {\em CBMS Regional Conference Series in
  Mathematics}.
\newblock Published for the Conference Board of the Mathematical Sciences,
  Washington, DC; by the American Mathematical Society, Providence, RI, 1994.
\newblock \textsc{doi:}
  \href{http://dx.doi.org/10.1090/cbms/084}{\texttt{10.1090/cbms/084}}.

\bibitem{PS13}
Wesley Pegden and Charles~K. Smart.
\newblock Convergence of the {A}belian sandpile.
\newblock {\em Duke Math. J.}, 162(4):627--642, 2013.
\newblock arXiv: \href{http://arxiv.org/abs/1105.0111}{\texttt{1105.0111
  [math.AP]}}.

\bibitem{SV09}
Klaus Schmidt and Evgeny Verbitskiy.
\newblock Abelian sandpiles and the harmonic model.
\newblock {\em Comm. Math. Phys.}, 292(3):721--759, 2009.
\newblock arXiv: \href{http://arxiv.org/abs/0901.3124}{\texttt{0901.3124
  [math.DS]}}.

\bibitem{SMKW15}
Andrey Sokolov, Andrew Melatos, Tien Kieu, and Rachel Webster.
\newblock Memory on multiple time-scales in an abelian sandpile.
\newblock {\em Physica A: Statistical Mechanics and its Applications},
  428:295--301, 2015.
\newblock \textsc{doi:}
  \href{http://dx.doi.org/10.1016/j.physa.2015.02.001}{\texttt{10.1016/j.physa.2015.02.001}}.

\bibitem{S04}
J.~Michael Steele.
\newblock {\em The {C}auchy-{S}chwarz master class}.
\newblock MAA Problem Books Series. Mathematical Association of America,
  Washington, DC; Cambridge University Press, Cambridge, 2004.
\newblock \textsc{doi:}
  \href{http://dx.doi.org/10.1017/CBO9780511817106}{\texttt{10.1017/CBO9780511817106}}.
\newblock An introduction to the art of mathematical inequalities.

\bibitem{T86}
E.~C. Titchmarsh.
\newblock {\em The theory of the {R}iemann zeta-function}.
\newblock The Clarendon Press, Oxford University Press, New York, second
  edition, 1986.
\newblock Edited and with a preface by D. R. Heath-Brown.

\end{thebibliography}

\end{document}